\documentclass[a4paper,10pt]{article}
\usepackage{mathtools}
\usepackage{geometry}
\usepackage[pdftex]{graphicx}
\usepackage{amsfonts,amssymb,amsmath,amsthm}
\usepackage{epstopdf}
\usepackage{csquotes}
\usepackage{color}
\usepackage{tikz}
\usetikzlibrary{matrix}

\numberwithin{equation}{section}
\newtheorem{theorem}{Theorem}
\newtheorem{remark}{Remark}

\newtheorem{lemma}[theorem]{Lemma}
\newtheorem{corollary}[theorem]{Corollary}
\newtheorem{proposition}[theorem]{Proposition}
\newtheorem{definition}[theorem]{Definition}

\numberwithin{theorem}{section}

\usepackage{tcolorbox}
\usepackage{enumerate}
\usepackage{algorithm2e}
\usepackage[title]{appendix}
\usepackage{cite}
\usepackage{pdfpages}
\usepackage{authblk}
\usepackage{titling}
\begin{document}

\title{Computing Spectral Measures and Spectral Types}
\author{Matthew J. Colbrook\thanks{Department of Applied Mathematics and Theoretical Physics, University of Cambridge, UK.\\Email: m.colbrook@damtp.cam.ac.uk}}

\date{}

\maketitle

\begin{abstract}
Spectral measures arise in numerous applications such as quantum mechanics, signal processing, resonance phenomena, and fluid stability analysis. Similarly, spectral decompositions (into pure point, absolutely continuous and singular continuous parts) often characterise relevant physical properties such as the long-time dynamics of quantum systems. Despite new results on computing spectra, there remains no general method able to compute spectral measures or spectral decompositions of infinite-dimensional normal operators. Previous efforts have focused on specific examples where analytical formulae are available (or perturbations thereof) or on classes of operators that carry a lot of structure. Hence the general computational problem is predominantly open. We solve this problem by providing the first set of general algorithms that compute spectral measures and decompositions of a wide class of operators. Given a matrix representation of a self-adjoint or unitary operator, such that each column decays at infinity at a known asymptotic rate, we show how to compute spectral measures and decompositions. We discuss how these methods allow the computation of objects such as the functional calculus, and how they generalise to a large class of partial differential operators, allowing, for example, solutions to evolution PDEs such as the linear Schr\"odinger equation on $L^2(\mathbb{R}^d)$. Computational spectral problems in infinite dimensions have led to the Solvability Complexity Index (SCI) hierarchy, which classifies the difficulty of computational problems. We classify the computation of measures, measure decompositions, types of spectra, functional calculus, and Radon--Nikodym derivatives in the SCI hierarchy. The new algorithms are demonstrated to be efficient on examples taken from orthogonal polynomials on the real line  and the unit circle (giving, for example, computational realisations of Favard's theorem and Verblunsky's theorem, respectively), and are applied to evolution equations on a two-dimensional quasicrystal.
\end{abstract}

\linespread{1}\selectfont{}
\tableofcontents
\normalsize
\linespread{1}\selectfont{}

\section{Introduction}
The analysis and computation of spectral properties of operators form core parts of many branches of science and mathematics, arising in diverse fields such as differential and integral equations, orthogonal polynomials, quantum mechanics, statistical mechanics, integrable systems and optics \cite{ponomarenko2013cloning,dean2013hofstadter,szabo2012modern,bender2007making,davies2007linear,simon1998classical,gray2006toeplitz}. Correspondingly, the problem of numerically computing the spectrum, $\sigma(T)$, of an operator $T$ acting on the canonical separable Hilbert space $l^2(\mathbb{N})$ has attracted a large amount of interest over the last 60 years or so \cite{aronszajn1951approximation, Arveson_cnum_lin94, Arveson_role_of94,Arveson_Improper93, Bottcher_appr, brown2007quasi, Lyonell_2, Sharg, deift1985toda, digernes1994finite, hansen2008, hansen2011, Levitin, Marletta_pollution, Marletta_Spec_gaps, Arieh1, Arieh2, Seidel_JFA,Hansen_PRS, SeSi_Pseudospectra, Eugene3, Silbermann22, riddell1967spectral,colbrookinfinite,colb1,colbrookPSEUDO}. However, the richness, beauty and difficulties that are encountered in infinite dimensions lie not just in the set $\sigma(T)\subset\mathbb{C}$, but also in the generalisation of projections onto eigenspaces and the possibility of different spectral types. Specifically, given a normal operator $T$, there is an associated projection-valued measure (resolution of the identity), which we denote by $E^T$, whose existence is guaranteed by the spectral theorem and whose support is $\sigma(T)$ \cite{kadison1997fundamentals_1,kadison1997fundamentals_2,reed1972methods}. This allows the representation of the operator $T$ as an integral over $\sigma(T)$, analogous to the finite-dimensional case of diagonalisation:
\begin{equation}
\label{spectral_theorem_classic}
Tx=\left[\int_{\sigma(T)} \lambda dE^T(\lambda)\right] x,\quad \forall x\in\mathcal{D}(T),
\end{equation}
where $\mathcal{D}(T)$ denotes the domain of $T$. For example, if $T$ is compact, then $E^T$ corresponds to projections onto eigenspaces, familiar from the finite-dimensional setting. However, in general, the situation is much richer and more complicated, with different types of spectra (pure point, absolutely continuous and singular continuous). An excellent and readable introduction can be found in Halmos' article \cite{halmos1963does}.

The computation of $E^T$, along with its various decompositions and their supports, is of great interest, both theoretically and for practical applications. For example, spectral measures are intimately related to correlation functions in signal processing, resonance phenomena in scattering theory, and stability analysis for fluids. Moreover, the computation of $E^T$ allows one to compute many additional objects, which we provide the first general algorithms for in this paper, such as the functional calculus (Theorem \ref{F_calc}), the Radon--Nikodym derivative of the absolutely continuous component of the measure (Theorem \ref{meas_comp3}), and the spectral measures and spectral set decompositions (Theorem \ref{spec_decomp_comp} and Theorem \ref{spec_decomp_as_sets}). For instance, in \S \ref{PDE_SEC_WOW} we show how our results allow the computation of spectral measures and the functional calculus of almost arbitrary self-adjoint partial differential operators on $L^2(\mathbb{R}^d)$. An important class of examples is given by solutions of evolution equations such as the Schr\"odinger equation on $L^2(\mathbb{R}^d)$ with a potential of locally bounded total variation. We prove that this is computationally possible even when an algorithm is only allowed to point sample the potential. A numerical example of fractional diffusion for a discrete quasicrystal is also given in \S \ref{penrose_numerics}.

Despite its importance, there has been no general method able to compute spectral measures of normal operators. Although there is a rich literature on the theory of spectral measures, most of the efforts to develop computational tools have focused on specific examples where analytical formulae are available (or perturbations thereof) or on classes of operators that carry a lot of
structure. Indeed, apart from the work of Webb and Olver \cite{webb2017spectra} (which deals with compact perturbations of tridiagonal Toeplitz operators) and methods for computing spectral density functions of Sturm--Liouville problems and other highly structured operators, there has been limited success in computing the measure $E^T$.\footnote{See \S \ref{prev_work} for connections with previous work.} In some sense, this is not surprising given the difficulty in rigorously computing spectra. One can consider computing spectral measures/projections as the infinite-dimensional analogue of computing projections onto eigenspaces.\footnote{Of course eigenvectors exist in the infinite-dimensional case, but not all of the spectrum consists of eigenvalues. The projection-valued measure generalises the notion of projections onto eigenspaces.} Thus, from a numerical/computational point of view, the current state of affairs in infinite-dimensional spectral computations is highly deficient, analogous in finite dimensions to being able to compute the location of eigenvalues but not eigenvectors! It has been unknown if the general computation of spectral measures is possible, even for simple subclasses such as discrete Schr\"odinger operators. In other words, the computational problem of ``diagonalisation'' through computing spectral measures remains an important and predominantly open problem.

In this paper, we solve this problem by providing the first set of algorithms for the computation of spectral measures for a large class of self-adjoint and unitary operators, namely, those whose matrix columns decay at a known asymptotic rate. This is a very weak assumption and covers the majority of operators, even unbounded, found in applications. In particular, those whose representation is sparse (such as many of the graph or lattice operators typically dealt with in physics) and also partial differential operators, once a suitable basis has been selected (see Theorem \ref{WowPDE} and Appendix \ref{append2}). We also show how to compute spectral measure decompositions, the functional calculus, the density of the absolutely continuous part of the measure (Radon--Nikodym derivative) and different types of spectra (pure point, absolutely continuous and singular continuous - these sets often characterise different physical properties in quantum mechanics \cite{ruelle1969remark,amrein1973characterization,enss1978asymptotic,simon1990absence,geisel1991new,combes1993connections,last1996quantum,cycon2009schrodinger}). A central ingredient of these new algorithms is the computation of the resolvent operator with error control through appropriate rectangular truncations (Theorem \ref{res_est1}). Furthermore, we demonstrate the applicability of our algorithms. The algorithms are efficient and parallelisable, allowing large scale computations.

\subsection{The Solvability Complexity Index and classification of problems}

A surprise thrown up by the infinite-dimensional spectral problem, which turns out to be quite generic, is the Solvability Complexity Index (SCI) \cite{hansen2011}. The SCI provides a hierarchy for classifying the difficulty of computational problems. In classical numerical analysis, when computing spectra, one hopes to construct an algorithm, $\Gamma_n$, with one limit such that for an operator $T$,
\begin{equation}\label{eq:one_limit}
\Gamma_n(T) \rightarrow \sigma(T), \quad\text{as } n \rightarrow \infty,
\end{equation}
preferably with some form of error control or rate of convergence. However, this is not always possible. For example, when considering the class of bounded operators, the best possible alternative is an algorithm depending on three indices $n_1, n_2, n_3$ such that 
  \begin{equation*}
\lim_{n_3 \rightarrow \infty} \lim_{n_2 \rightarrow \infty} \lim_{n_1 \rightarrow \infty} \Gamma_{n_3,n_2,n_1}(T) = \sigma(T).
\end{equation*}
Any algorithm with fewer than three limits will fail on some bounded operator, and neither error control nor convergence rates on any of the limits are possible since these would reduce the required number of limits. However, for self-adjoint operators, it is possible to reduce the number of limits to two, but not one \cite{hansen2011,ben2015can}. With more structure (such as sparsity or column decay) it is possible to compute the spectrum in one limit with a certain type of error control \cite{colb1}. Hence, the only way to characterise the computational spectral problem is through a hierarchy, classifying the difficulty of computing spectral properties of different subclasses of operators. The SCI classifies difficulty by considering the minimum number of limits that one must take to calculate the quantity of interest (see Appendix \ref{append1} for a full definition). The SCI has roots in the work of Smale \cite{Smale2, Smale_Acta_Numerica}, and his program on the foundations of computational mathematics and scientific computing, though it is quite distinct. The notions of Turing computability \cite{turing1937computable} and computability in the Blum--Shub--Smale (BSS) \cite{BCSS} sense become special cases, and impossibility results that are proven in the SCI hierarchy hold in all models of computation. The phenomenon of needing several limits also covers general numerical analysis problems, such as Smale's question on the existence of purely iterative algorithms for polynomial root finding \cite{smale_question}. As demonstrated by McMullen \cite{McMullen1, McMullen2} and Doyle and McMullen \cite{Doyle_McMullen}, this is a case where several limits are needed in the computation, and their results become special cases of classification in the SCI hierarchy \cite{ben2015can}. Extensions of the hierarchy to error control \cite{colbrook3,colbrook4,colbrook2020foundations} also have potential applications in the growing field of computer-assisted proofs, where one must perform a computation with absolute certainty. See, for example, the work of Fefferman and Seco on the Dirac--Schwinger conjecture \cite{fefferman1990,fefferman1996interval} and Hales on Kepler's conjecture (Hilbert's 18th problem) \cite{Hales_Annals}, both of which implicitly provide classifications in the SCI hierarchy. As well as spectral problems \cite{colb1,colbrookinfinite,colbrook4,colbrook3,colbrook2020computingSM,colbrook2020foundations,hansen2011,ben2015can}, the SCI hierarchy has recently been used to solve problems related to computing resonances \cite{Jonathan_res,benartzi2020computing}, computing solutions of semigroups and evolution PDEs \cite{colb_semigp,becker2020computing}, and computational barriers for stable and accurate neural networks \cite{antun2021can}.

An informal definition of the SCI hierarchy is as follows, with a detailed summary contained in Appendix \ref{append1}. The SCI hierarchy is based on the concept of a computational problem. This is described by a function 
\[
\Xi:\Omega \rightarrow \mathcal{M}
\] that we want to compute, where $\Omega$ is some domain, and $(\mathcal{M},d)$ is a metric space. For example, we could take $\Xi(T) = \sigma(T)$ (the spectrum) for some class of operators $\Omega$ and $\mathcal{M}$ the collection of non-empty closed subsets of $\mathbb{C}$ equipped with the Attouch--Wets metric. The SCI of a computational problem is the smallest number of limits needed in order to compute the solution. For a given set of evaluation functions (the information our algorithm is allowed to read - in our case, $\Lambda_1$ or $\Lambda_2$ defined in (\ref{evals_def_nene})), class of objects (in our case, subclasses of operators acting on $l^2(\mathbb{N})$) and model of computation $\alpha$ (in this paper general, $G$, or arithmetic, $A$) we define:
\begin{itemize}
\item[(i)] $\Delta^{\alpha}_0$ is the set of problems that can be computed in finite time, the SCI $=0$.
\item[(ii)] $\Delta^{\alpha}_1$ is the set of problems that can be computed using one limit (the SCI $=1$) with control of the error, i.e. $\exists$ a sequence of algorithms $\{\Gamma_n\}$ such that $d(\Gamma_n(A), \Xi(A)) \leq 2^{-n}, \, \forall A \in \Omega$.
\item[(iii)] $\Delta^{\alpha}_2$ is the set of problems that can be computed using one limit (the SCI $=1$) without error control, i.e. $\exists$ a sequence of algorithms $\{\Gamma_n\}$ such that $\lim_{n\rightarrow \infty}\Gamma_n(A) = \Xi(A), \, \forall A \in \Omega$.
\item[(iv)] $\Delta^{\alpha}_{m+1}$, for $m \in \mathbb{N}$, is the set of problems that can be computed by using $m$ limits, (the SCI $\leq m$), i.e. $\exists$ a family of algorithms $\{\Gamma_{n_m, \hdots, n_1}\}$ such that 
$$
\lim_{n_m \rightarrow\infty}\hdots \lim_{n_1\rightarrow\infty}\Gamma_{n_m,\hdots, n_1}(A) = \Xi(A), \, \forall A \in \Omega.
$$
 \end{itemize}
The class $\Delta_1$ is, of course, a highly desired class; however, non-trivial spectral problems are higher up in the hierarchy. For example, the following classifications are known \cite{hansen2011,ben2015can}:
\begin{itemize}
\itemsep0em
\item[(i)] The general spectral problem is in $\Delta_4\setminus \Delta_3$.
\item[(ii)] The self-adjoint spectral problem is in $\Delta_3\setminus \Delta_2$.
\item[(iii)] The compact spectral problem is in $\Delta_2\setminus \Delta_1$.
\end{itemize}
Here, the notation $\setminus$ indicates the standard ``setminus''. Hence, the computational spectral problem becomes an infinite classification theory to characterise the above hierarchy. In order to do so, there will, necessarily, have to be many different types of algorithms. Characterising the hierarchy will yield a myriad of different approaches, as different structures on the various classes of operators will require specific algorithms.

This paper provides classifications of spectral problems associated with $E^T$ (such as decompositions of the measure and spectrum) in the SCI hierarchy, some of which can be computed in one limit. We provide algorithms for these problems, and one of the main tools used is the computation of the resolvent operator $R(z,T):=(T-zI)^{-1}$ with error control (Theorem \ref{res_est1}). This is possible through appropriate \textit{rectangular} truncations of the infinite-dimensional operator. This approach differs from finite-dimensional techniques, which typically consider square truncations.

\begin{remark}[Recursivity and independence of the model of computation]\label{Qremark}
The constructive inclusion results we provide hold for arithmetic algorithms and the impossibility results hold for general algorithms. We refer the reader to Appendix \ref{append1} for a detailed explanation. Put simply, this means that the algorithms constructed can be recursively implemented with inexact input and restrictions to arithmetic operations over the rationals (it is also straightforward to implement them using interval arithmetic), whereas the impossibility results hold in any model of computation (such as the Turing or BSS models).
\end{remark}

\subsection{Summary of the main results}
\label{main_res}

\subsubsection{Partial differential operators}
\label{PDE_SEC_WOW}

For $N \in \mathbb{N}$, consider the operator formally defined on $L^2(\mathbb{R}^d)$ by
\begin{equation}\label{eq:diff_op}
Lu(x)=\sum_{k\in\mathbb{Z}_{\geq0}^d,\left|k\right|\leq N}a_k(x)\partial^ku(x),
\end{equation}
where throughout we use multi-index notation with $\left|k\right|=\max\{\left|k_1\right|,...,\left|k_d\right|\}$ and $\partial^k=\partial^{k_1}_{x_1}\partial^{k_2}_{x_2}...\partial^{k_d}_{x_d}$. We will assume that the coefficients $a_k(x)$ are complex-valued measurable functions on $\mathbb{R}^d$ and that $L$ is self-adjoint. For dimension $d$ and $r>0$ consider the space
\begin{equation}
\mathcal{A}_r=\{f\in M([-r,r]^d):\left\|f\right\|_{\infty}+\mathrm{TV}_{[-r,r]^d}(f)<\infty\},
\end{equation}
where $M([-r,r]^d)$ denotes the set of measurable functions on the hypercube $[-r,r]^d$ and $\mathrm{TV}_{[-r,r]^d}$ denotes the total variation norm in the sense of Hardy and Krause \cite{niederreiter1992random}. This space becomes a Banach algebra when equipped with the norm \cite{blumlinger1989topological}
$$
\left\|f\right\|_{\mathcal{A}_r}:=\left\|f|_{[-r,r]^d}\right\|_{\infty}+(3^d+1)\mathrm{TV}_{[-r,r]^d}(f).
$$
Let $\Omega_{\mathrm{PDE}}$ consist of all such $L$ such that the following assumptions hold:
\begin{enumerate}
	\item[(1)] The set $C_{0}^\infty(\mathbb{R}^d)$ of smooth, compactly supported functions forms a core of $L$.
	\item[(2)] For each of the functions $a_k(x)$, there exists a positive constant $A_k$ and an integer $B_k$ (both possibly unknown) such that 
	$$
	\left|a_k(x)\right|\leq A_k\left(1+\left|x\right|^{2B_k}\right),
	$$
	almost everywhere on $\mathbb{R}^d$, that is, we have at most polynomial growth of the coefficients.
	\item[(3)] The restrictions $a_k|_{[-r,r]^d}\in\mathcal{A}_r$ for all $r>0$.
\end{enumerate}
We consider the case where our algorithms can do the following:
\begin{enumerate}
	\item Evaluate any coefficient $a_k(x)$ to a given precision at $x\in\mathbb{Q}^d$, where $\mathbb{Q}$ denotes the field of rationals, and output an approximation in $\mathbb{Q}+i\mathbb{Q}$.
	\item For each $n$, evaluate a positive number $b_n(L)$ such that the sequence $\{b_n(L)\}_{n\in\mathbb{N}}$ satisfies
\begin{equation}
\label{tot_var_bound}
\sup_{n\in\mathbb{N}}\max_{\left|k\right|\leq N}\frac{\left\|a_k\right\|_{\mathcal{A}_n}}{b_n(L)}<\infty.
\end{equation}
\end{enumerate}

In Appendix \ref{append2}, we prove (and state a more precise version of) the following theorem.

\begin{theorem}[Spectral properties of self-adjoint partial differential operators can be computed]
\label{WowPDE}
Given the above set-up, there exist sequences of arithmetic algorithms that compute spectral measures, the functional calculus, and Radon--Nikodym derivatives of the absolutely continuous part of the measure over the class $\Omega_\mathrm{PDE}$. In other words, these objects can be computed in one limit with $\mathrm{SCI}=1$.
\end{theorem}

\begin{remark}
As noted in Appendix \ref{append2}, we can extend this result to computing the decompositions into pure point, absolutely continuous and singular continuous parts of measures and spectra (with $\mathrm{SCI}>1$).
\end{remark}

The above properties characterising $\Omega_{\mathrm{PDE}}$ are deliberately very weak, and hence the class $\Omega_{\mathrm{PDE}}$ is very large. For example, Schr\"odinger operators $L=-\Delta+V$ with polynomially bounded potentials of locally bounded total variation are a subclass of $\Omega_{\mathrm{PDE}}$. Hence, in this case, the theorem says that we can compute the spectral properties (measures, functional calculus etc.) of $L$ by point sampling the potential $V$, if we have an asymptotic bound on the total variation of $V$ over finite rectangles. In particular, we can solve the Schr\"odinger equation
\begin{equation}
\label{evolution_we_can_handleII}
\frac{du}{dt}=-iLu,\quad u_{t=0}=u_0
\end{equation}
by computing $\exp(-itL)u_0$ with guaranteed convergence. Applications of some of our results to semigroups (where error control can also be provided) are given in \cite{colb_semigp}. The proof of Theorem \ref{WowPDE} can also be extended to other domains such as the half-line, and can be adapted to cope with other types of coefficients that are not of locally bounded total variation (for instance, Coulombic potentials for Dirac or Schr\"odinger operators). In order to prove Theorem \ref{WowPDE}, we prove theorems for operators on $l^2(\mathbb{N})$.

\subsubsection{Operators on $l^2(\mathbb{N})$}
\label{DISCRETE_SEC_WOW}

We consider self-adjoint operators given as an infinite matrix $T$ whose columns decay at a known asymptotic rate:
\begin{equation}
\label{need_def}
\|(P_{f(n)}-I)TP_n\|=O(\alpha_n)
\end{equation}
for a sequence $\alpha_n\downarrow 0$ and function $f:\mathbb{N}\rightarrow\mathbb{N}$, where $P_n$ denotes the orthogonal projection onto the linear span of the first $n$ canonical basis vectors. This set-up is adapted in Appendix \ref{append2} to prove Theorem \ref{WowPDE}, where we make use of the following theorems proven for operators on $l^2(\mathbb{N})$. 

We consider the problem of computing spectral measures. Specifically, for operators of the form (\ref{need_def}) we develop the first algorithms and SCI classifications for:
\begin{itemize}
	\item \textbf{Theorem \ref{meas_comp1}:} $\Delta_2^A$ classification for the projection-valued spectral measure and, through taking inner products, the computation of the scalar spectral measures defined via
	$$
	\mu_{x,y}^T(U)=\langle E^T({U}) x,y\rangle.
	$$
This is done for open sets $U$ and can be extended to other types of sets such as closed intervals or singletons (Theorem \ref{meas_err_impossible} shows that the problem $\notin\Delta_1^G$ in certain cases). These scalar-valued spectral measures play an important role in, for example, quantum mechanics, where they correspond to the distribution of the physical observable modelled by a Hamiltonian $T$.
	\item \textbf{Theorem \ref{spec_decomp_comp}:} $\Delta_3^A\setminus \Delta_2^G$ classification for the decompositions of the projection-valued and scalar-valued spectral measures into absolutely continuous, singular continuous and pure point parts. These decompositions often characterise different physical properties in quantum mechanics \cite{ruelle1969remark,amrein1973characterization,enss1978asymptotic,simon1990absence,geisel1991new,combes1993connections,last1996quantum,cycon2009schrodinger}
	\item \textbf{Theorem \ref{F_calc}:} $\Delta_2^A$ classification for the functional calculus of operators, i.e. the computation of $F(T)$ for suitable functions $F$. For brevity, we consider functions that are bounded and continuous on the spectrum of $T$. However, the proof makes clear that wider classes can be dealt with. In some cases, $\Delta_1^A$ error control is possible, for instance, when considering the holomorphic functional calculus (see \S \ref{penrose_numerics}). A key application of this result is the computation of solutions of many evolution PDEs, such as the Schr\"odinger equation through the function $F(z)=\exp(-izt)$.
	\item \textbf{Theorem \ref{meas_comp3}:} $\Delta_2^A$ classification for the Radon--Nikodym derivatives (densities) of the absolutely continuous parts of the scalar spectral measures with convergence in the $L^1$ sense on an open set. This requires a certain separation condition, without which our algorithm converges (Lebesgue) almost everywhere.
\end{itemize}

We also consider the computation of spectra as sets in the complex plane. Convergence is measured using the Hausdorff metric in the bounded case and using the Attouch--Wets metric in the unbounded case (i.e. uniform convergence on compact subsets of $\mathbb{C}$). Specifically, we provide the first algorithms computing these quantities and prove in \textbf{Theorem \ref{spec_decomp_as_sets}} that:
\begin{itemize}
	\item The absolutely continuous spectrum $\sigma_{\mathrm{ac}}(T)$ can be computed in two limits but not one limit ($\Delta_3^A\setminus \Delta_2^G$ classification).
	\item The pure point spectrum $\sigma_{\mathrm{pp}}(T)$ can be computed in two limits but not one limit ($\Delta_3^A\setminus \Delta_2^G$ classification).
	\item The singular continuous spectrum $\sigma_{\mathrm{sc}}(T)$ can be computed in three limits ($\Delta_4^A$ classification). If $f(n)-n\geq \sqrt{2n}+1/2$, then the computation cannot be done in two limits ($\notin\Delta_3^G$ classification). That is, if the local asymptotic bandwidth is allowed to grow sufficiently rapidly, three limits are needed, and this computational problem is exceedingly difficult. We do not know whether this growth condition on $f$ can be dropped. However, without it, the problem still requires at least two limits ($\notin\Delta_2^G$ classification).
\end{itemize}

The main tool used to prove the above results is
\begin{itemize}
	\item \textbf{Theorem \ref{res_est1} and Corollary \ref{res_est2}:} The action of the resolvent $x\rightarrow R(z,T)x$ can be computed with \textit{error control}. This also opens up potential applications in computer-assisted proofs.
\end{itemize}

We demonstrate that the ``one-limit'' algorithms constructed in this paper are implementable and efficient. These provide the first set of algorithms addressing these problems, and we have provided extensive numerical experiments in \S \ref{num_sec}. This includes orthogonal polynomials on the real line and unit circle (where we also discuss acceleration through extrapolation), as well as fractional diffusion for a two-dimensional quasicrystal. Since writing the initial version of this paper, our algorithms have also been implemented in the software package \texttt{SpecSolve} of \cite{colbrook2020computingSM}, which uses the machinery of high-order rational kernels to accelerate computation of the Radon--Nikodym derivative of regular enough measures. In the current paper, we also show that in the case where the measure is regular enough, a global collocation method and local Richardson extrapolation for the computation of the Radon--Nikodym derivative can both accelerate convergence.

\vspace{3mm}

Finally, some brief remarks are in order.
\begin{itemize}
	\item[(i)] The impossibility results hold in general, even when restricted to tridiagonal operators. Furthermore, many of the impossibility results hold for structured operators such as bounded discrete Schr\"odinger operators. Our results (constructive algorithms and impossibility results) also carry over to a large class of normal operators, including unitary operators or skew-adjoint operators, both of which are important in applications. For the sake of clarity, we have stuck to the self-adjoint case in the statement of theorems and proofs. Numerical examples for unitary operators are given in \S \ref{CMV_numerics}.
	\item[(ii)] The difficulty encountered when computing the singular continuous spectrum is partly due to the negative definition of the singular continuous part of a measure. It is the ``leftover'' part of the measure, the part that is not continuous with respect to Lebesgue measure and does not contain atoms.
The challenge of studying $\sigma_{\mathrm{sc}}$ analytically also reflects this difficulty - singular continuous spectra were once thought to be rather rare or exotic. However, they are quite generic; see, for example, \cite{simon1995operators} for a precise topological statement to this effect.
\item[(iii)] One might at first expect computational results to be independent of the function $f$ due to tridiagonalisation. However, the infinite-dimensional case is much more subtle than the finite-dimensional case. Using Householder transformations on a bounded sparse self-adjoint operator $T$ leads to a tridiagonal operator, but, in general, this operator is $T$ restricted to a \textit{strict} subspace of $l^2(\mathbb{N})$. Part of the operator may be lost in the strong operator limit. Instead, one must consider a sum of possibly infinitely many tridiagonal operators (see \cite[Ch. 2 \& 8]{anders_thesis}). Hence some spectral problems may have different classifications for different $f$.
\end{itemize}

\subsection{Open problems}

This results of this paper form part of a program \cite{colb1,colbrookinfinite,colbrookPSEUDO,colbrook4,colbrook3,colbrook2020computingSM,colbrook2020foundations} on determining the foundations of computation (boundaries of what is possible) for infinite-dimensional spectral problems. As such, there remain many interesting open problems related to this paper, such as:

\begin{itemize}
	\item \textbf{Computation of spectral measures for more general normal operators:} Proposition \ref{stone2} demonstrates how the techniques of this paper can be generalised to certain classes of normal operators whose spectrum does not necessarily lie along a curve. However, it is not obvious how to extend the techniques to completely general normal operators. For example, the method of integrating the resolvent along a contour cannot be easily extended to interior points of the spectrum.
	\item \textbf{Brown measure for non-normal operators:} There is an extension of spectral measures to certain non-normal operators known as the Brown measure \cite{brown1983lidskii,haagerup2007brown,haagerup2009invariant}. Computing spectra of non-normal operators is generally harder (in a sense made precise by the SCI hierarchy) than for normal operators (issues such as numerical stability are also present even in the finite-dimensional case). It would be interesting to see if this phenomenon is also present for computing the Brown measure. Some works on approximating the Brown measure from matrix elements include \cite{hansen2011,Arveson_cnum_lin94,bedos1996folner}.
\end{itemize}

\subsection{A motivating example}
\label{motiv}

As a motivating example, consider the case of a Jacobi operator with matrix
$$
J=\begin{pmatrix}
b_1 & a_1 &  & \\
a_1 & b_2 & a_2 & \\
& a_2 & b_3 & \ddots\\
 &  & \ddots & \ddots
\end{pmatrix},
$$
where $a_j,b_j\in\mathbb{R}$ and $a_j>0$. An enormous amount of work exists on the study of these operators, and the correspondence between bounded Jacobi matrices and probability measures with compact support \cite{teschl2000jacobi,deift1999orthogonal}. The entries in the matrix provide the coefficients in the recurrence relation for the associated orthonormal polynomials. To study the canonical measure $\mu_J$, one usually considers the principal resolvent function, which is defined on $\mathbb{C}\backslash\sigma(J)$ via
\begin{equation}
G(z):=\langle R(z,J)e_1,e_1\rangle=\int_{\mathbb{R}}\frac{d\mu_J(\lambda )}{\lambda -z},
\end{equation}
and then takes $z$ close to the real axis. The function $G$ is also known in the differential equations and Schr\"odinger communities as the Weyl $m$-function \cite{teschl2000jacobi,gesztesy1997m} and one can develop the discrete analogue of what is known as Weyl--Titchmarsh--Kodaira theory for Sturm--Liouville operators. Going back to the work of Stieltjes \cite{stieltjes1894recherches} (see also \cite{akhiezer1965classical,wall2018analytic}), there is a representation of $G$ as a continued fraction:
\begin{equation}
\label{ctd_fraction}
G(z):=\frac{1}{-z+b_1-\frac{a_1^2}{-z+b_2-...}}.
\end{equation}
One can also approximate $G$ via finite truncated matrices \cite{teschl2000jacobi}.

\begin{figure}
\centering
\includegraphics[width=0.48\textwidth,trim={32mm 92mm 32mm 92mm},clip]{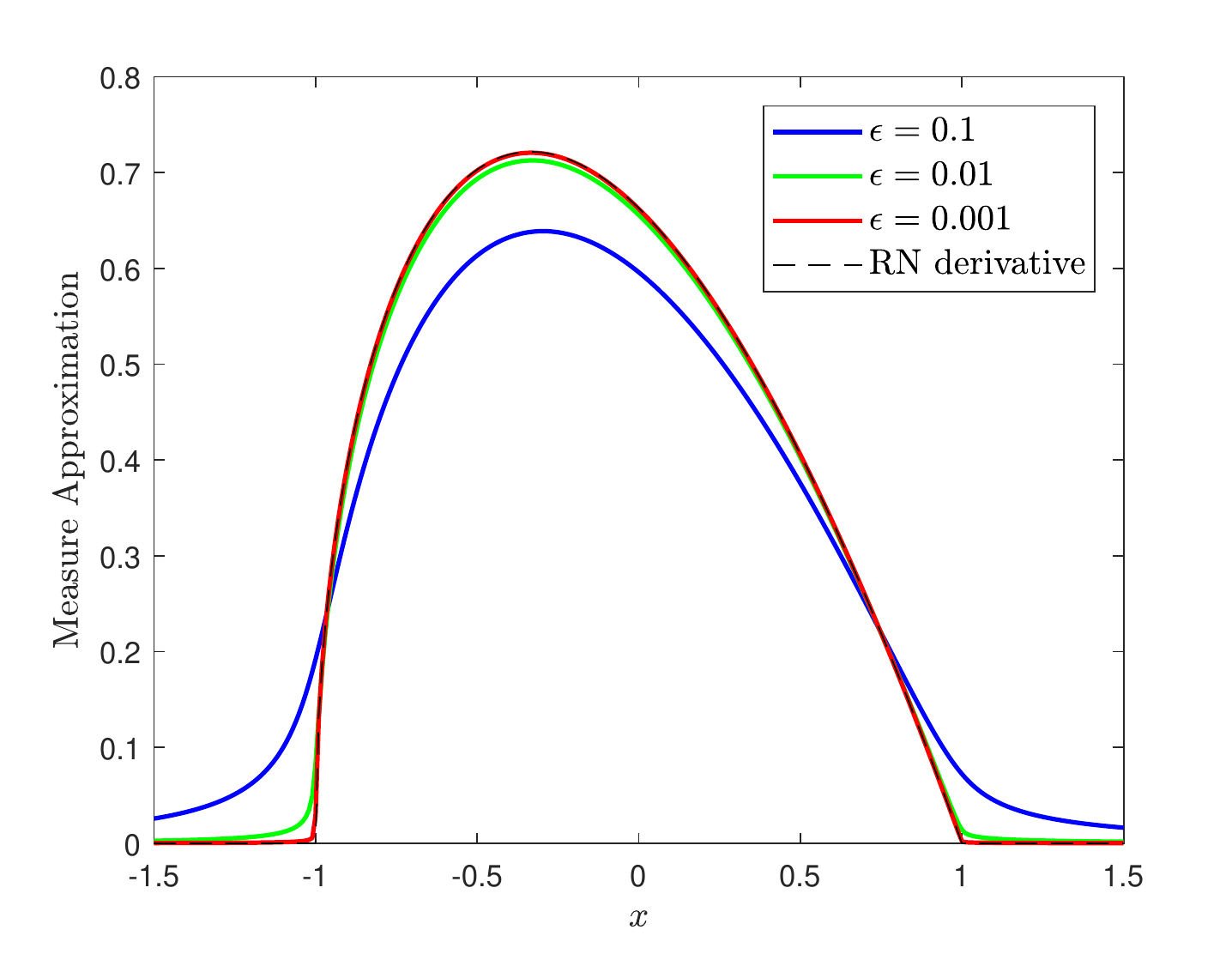}
\includegraphics[width=0.48\textwidth,trim={32mm 92mm 32mm 92mm},clip]{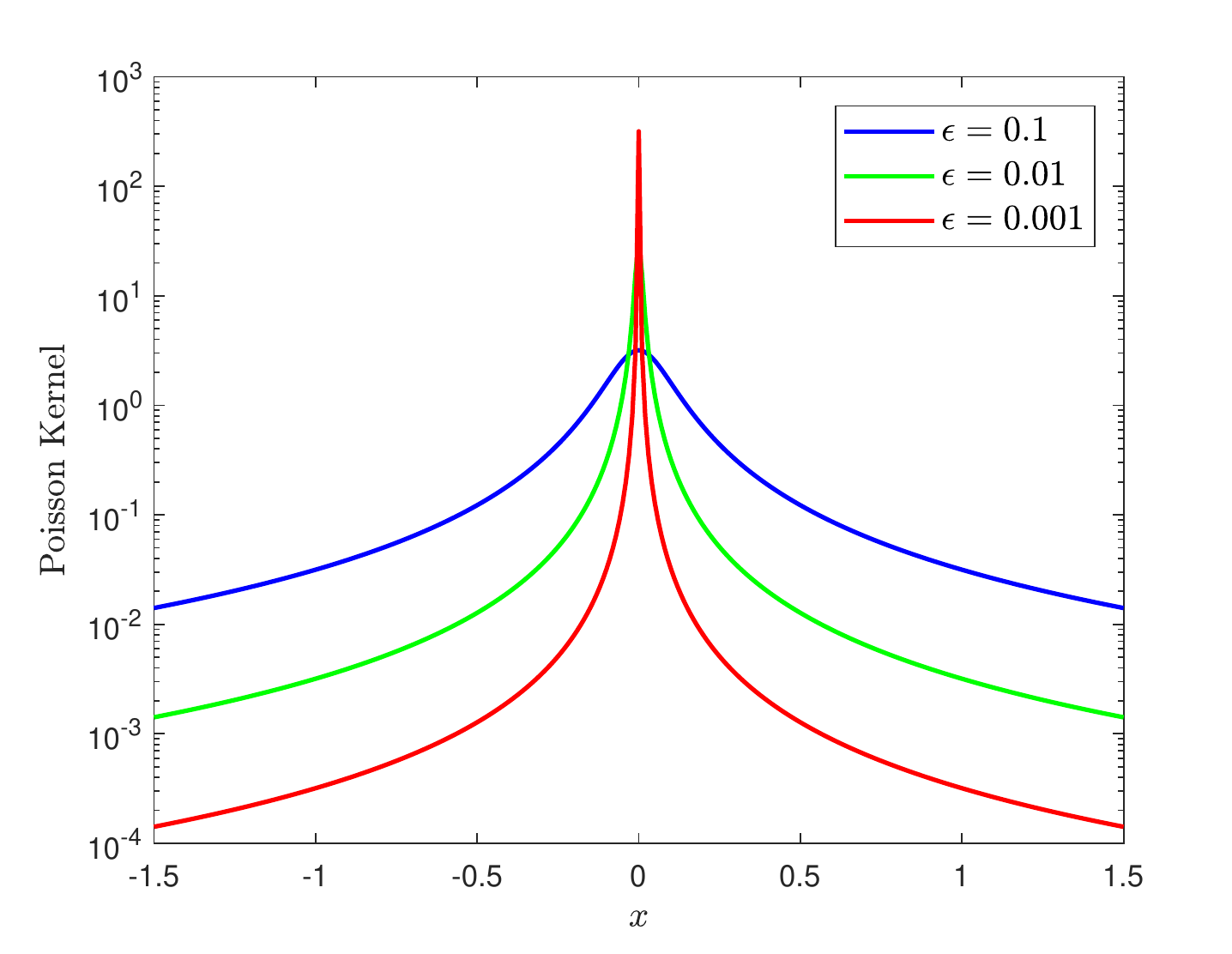} 
\caption{Smoothed approximations of the Radon--Nikodym derivative for the Jacobi operator associated to Jacobi polynomials with $\alpha=1$, $\beta=1/2$. Here the measure is absolutely continuous and supported on $[-1,1]$. Left: Computation of convolutions for different $\epsilon$ using the methods of this paper. Right: The associated Poisson kernel $\pi^{-1}\epsilon/(\epsilon^2+x^2)$ which approaches a Dirac delta distribution as $\epsilon\downarrow 0$.}
\label{idea}
\end{figure}

However, there are two major obstacles to overcome when using (\ref{ctd_fraction}) and its variants as a means to compute measures. First of all, this representation of the principal resolvent function is structurally dependent. For example, (\ref{ctd_fraction}) is valid for the restricted case of Jacobi operators and hence one is led to seeking different methods for different operators (such as tight-binding Hamiltonians on two-dimensional lattices which have a growing bandwidth when represented as operators on $l^2(\mathbb{N})$). Second, this would seem to give the wrong classification of the difficulty of the problem in the SCI hierarchy, giving rise to a tower of algorithms with two limits. One first takes a truncation parameter $n$ to infinity to compute $G(z)$ for $\mathrm{Im}(z)>0$, and then a second limit as $z$ approaches the real axis. One of the main messages of this paper is that both of these issues can be overcome. Measures can be computed in one limit via an algorithm $\Gamma_n$ and for a large class of operators. The only restriction is a known asymptotic decay rate of the off-diagonal entries. As a by-product, we compute the $m$-function of such operators with error control. Specific cases where this can be written explicitly do exist, such as periodic Jacobi matrices or perturbations of Toeplitz operators \cite{edleman} (see also \S \ref{prev_work}). However, there has been no general method proposed to compute the resolvent with error control. This consideration is crucial to allow the computation of measures in one limit.

To see how we might compute the measure using the resolvent, consider the Poisson kernels for the half-plane and the unit disk, defined respectively by
\begin{equation}
P_{H}(x,y)=\frac{1}{\pi}\frac{y}{x^2+y^2}\text{ and } P_{D}(x,y)=\frac{1}{2\pi}\frac{1-(x^2+y^2)}{(x-1)^2+y^2}=P_{D}(r,\theta)=\frac{1}{2\pi}\frac{1-r^2}{1-2r\cos(\theta)+r^2},
\end{equation}
where $(r,\theta)$ denote the usual polar coordinates. Let $T$ be a normal operator, then for $z\notin\sigma(T)$, we have from the functional calculus that
$$
R(z,T)=\int_{\sigma(T)}\frac{1}{\lambda -z}dE^T(\lambda ).
$$
For self-adjoint $T$, $z=u+iv\in\mathbb{C}\backslash{\mathbb{R}}$ ($u,v\in\mathbb{R}$) and $x\in l^2(\mathbb{N})$, we define
\begin{equation}
\label{Poiss_int1}
\begin{split}
K_H(z;T,x):&=\frac{1}{2\pi i}[R(z,T)-R(\overline{z},T)]x\\
&=\frac{1}{2\pi i}\int_{-\infty}^\infty\left(\frac{1}{\lambda -z}-\frac{1}{\lambda -\overline{z}}\right)dE^T(\lambda )x=\int_{-\infty}^\infty P_H(u-\lambda ,v)dE^T(\lambda )x.
\end{split}
\end{equation}
Similarly, if $T$ is unitary, $z=r\exp(i\psi)\in\mathbb{C}\backslash{\mathbb{T}}$ (with $z\neq 0$) and $x\in l^2(\mathbb{N})$, we define
\begin{equation}
\label{Poiss_int2}
K_D(z;T,x):=\frac{1}{2\pi i}[R(z,T)-R(1/\overline{z},T)]x=\frac{1}{2\pi i}\int_{\mathbb{T}}\left[\frac{1}{\lambda -z}-\frac{1}{\lambda -1/\overline{z}}\right]dE^T(\lambda )x.
\end{equation}
We change variables $\lambda =\exp(i\theta)$ and, with an abuse of notation, write $dE^T(\lambda )=i\exp(i\theta)dE^T(\theta)$. A simple calculation then gives
\begin{equation}
\label{Poiss_int3}
K_D(z;T,x)=\int_{0}^{2\pi} P_D(r,\psi-\theta)dE^T(\theta)x.
\end{equation}

Returning to our example, we see that the computation of the resolvent with error control allows the computation of $G(z)$ with error control through taking inner products. By considering $G(z)-G(\overline{z})$, this allows the computation of the convolution of the measure $\mu_J$ with the Poisson kernel $P_{H}$. In other words, we can compute a \textit{smoothed version} of the measure $\mu_J$ with error control. We can then take the smoothing parameter to zero to recover the measure (one can show that these smoothed approximations converge weakly in the sense of measures). Figure \ref{idea} demonstrates this for a typical example.

\subsection{Functional analytic setup}
\label{sec_not}

We consider the canonical\footnote{By a choice of basis our results extend to any separable Hilbert space.} separable Hilbert space $\mathcal{H} = l^2(\mathbb{N})$, the set of square summable sequences with canonical basis $\{e_n\}_{n\in\mathbb{N}}$. Let $\mathcal{C}(l^2(\mathbb{N}))$ be the set of closed densely defined linear operators $T$ such that $\mathrm{span}\{e_n:n\in\mathbb{N}\}$ forms a core of $T$ and $T^*$. The spectrum of $T\in\mathcal{C}(l^2(\mathbb{N}))$ will be denoted by $\sigma(T)$ and the point spectrum (the set of eigenvalues) by $\sigma_{\mathrm{p}}(T)$. The latter set is not always closed and in general the closure of a set $S$ will be denoted by $\overline{S}$. The resolvent operator $(T-zI)^{-1}$ defined on $\mathbb{C}\backslash\sigma(T)$ will be denoted by $R(z,T)$.

This paper focusses on the subclass $\Omega_{\mathrm{N}}\subset \mathcal{C}(l^2(\mathbb{N}))$ of normal operators, that is, operators for which $\mathcal{D}(T)=\mathcal{D}(T^*)$ and $\left\|Tx\right\|=\left\|T^*x\right\|$ for all $x\in\mathcal{D}(T)$. The subclasses $\subset\Omega_{\mathrm{N}}$ of self-adjoint and unitary operators will be denoted by $\Omega_{\mathrm{SA}}$ and $\Omega_{\mathrm{U}}$ respectively. For $T\in\Omega_{\mathrm{SA}}$ and $T\in\Omega_\mathrm{U}$, $\sigma(T)\subset\mathbb{R}$ and $\sigma(T)\subset\mathbb{T}$ respectively, where $\mathbb{T}$ denotes the unit circle. Given $T\in\Omega_{\mathrm{N}}$ and a Borel set $B$, $E^T_B$ will denote the projection $E^T(B)$. Given $x,y\in l^2(\mathbb{N})$, we can define a bounded (complex-valued) measure $\mu_{x,y}^T$ via the formula
\begin{equation}
\mu_{x,y}^T(B)=\langle E_{B}^T x,y\rangle.
\end{equation}
Via the Lebesgue decomposition theorem \cite{halmos2013measure}, the spectral measure $\mu^T_{x,y}$ can be decomposed into three parts
\begin{equation}
\label{spec_decomp_hard}
\mu_{x,y}^T=\mu_{x,y,\mathrm{ac}}^T+\mu_{x,y,\mathrm{sc}}^T+\mu_{x,y,\mathrm{pp}}^T,
\end{equation}
the absolutely continuous part of the measure (with respect to the Lebesgue measure), the singular continuous part (singular with respect to the Lebesgue measure and atomless) and the pure point part. When considering $\Omega_{\mathrm{SA}}$, we will consider Lebesgue measure on $\mathbb{R}$ and let
\begin{equation}
\label{RN_def}
\rho_{x,y}^T(\lambda)=\frac{d\mu_{x,y,\mathrm{ac}}^T}{dm}(\lambda),
\end{equation}
the Radon--Nikodym derivative of $\mu_{x,y,\mathrm{ac}}^T$ with respect to Lebesgue measure. Of course this can be extended to the unitary (and, more generally, normal) case. This naturally gives a decomposition of the Hilbert space $\mathcal{H}=l^2(\mathbb{N})$. For $\mathcal{I}=\mathrm{ac},\mathrm{sc}$ and $\mathrm{pp}$, we let $\mathcal{H}_{\mathcal{I}}$ consist of vectors $x$ whose measure $\mu_{x,x}^T$ is absolutely continuous, singular continuous and pure point respectively. This gives rise to the orthogonal (invariant subspace) decomposition
\begin{equation}
\label{hilbert_decomp}
\mathcal{H}=\mathcal{H}_{\mathrm{ac}}\oplus\mathcal{H}_{\mathrm{sc}}\oplus\mathcal{H}_{\mathrm{pp}},
\end{equation}
whose associated projections will be denoted by $P_{\mathrm{ac}}^{T}$, $P_{\mathrm{sc}}^{T}$ and $P_{\mathrm{pp}}^{T}$ respectively. These projections commute with $T$ and the projections obtained through the projection-valued measure. Of particular interest is the spectrum of $T$ restricted to each $\mathcal{H}_{\mathcal{I}}$, which will be denoted by $\sigma_{\mathcal{I}}(T)$. These different sets and subspaces often, but not always, characterise different physical properties in quantum mechanics (such as the famous RAGE theorem \cite{ruelle1969remark,amrein1973characterization,enss1978asymptotic}), where a system is modelled by some Hamiltonian $T\in \Omega_{\mathrm{SA}}$ \cite{cycon2009schrodinger,combes1993connections,geisel1991new,last1996quantum}. For example, pure point spectrum implies the absence of ballistic motion for many Schr\"odinger operators \cite{simon1990absence}.

\subsection{Algorithmic setup}
\label{sec_not2}

Given an operator $T\in\mathcal{C}(l^2(\mathbb{N}))$, we can view it as an infinite matrix
$$
T=\left( \begin{array}{cccc}
t_{11} & t_{12} & t_{13} &\hdots \\
t_{21} & t_{22} & t_{23} &\hdots\\
t_{31} & t_{32} & t_{33} &\hdots\\
\vdots & \vdots & \vdots& \ddots \end{array} \right)
$$
through the inner products\footnote{Our convention throughout will be that the inner product $\langle \cdot,\cdot\rangle$ is linear in the first component and conjugate-linear in the second.} $t_{ij}=\langle Te_j,e_i\rangle$. All of the algorithms constructed can also be adapted to operators on $l^2(\mathbb{Z})$, either through the use of a suitable re-ordering of the basis, or though considering truncations of matrices in two directions, which is useful numerically since it preserves bandwidth.
To be precise about the information needed to compute spectral properties, we define two classes of evaluation functions as
\begin{equation}\label{evals_def_nene}
\Lambda_1=\{\langle Te_j,e_i\rangle:i,j\in\mathbb{N}\},\quad\Lambda_2=\{\langle Te_j,e_i\rangle,\langle T^*e_j,T^*e_i\rangle:i,j\in\mathbb{N}\}.
\end{equation}
These can be understood as different sets of information our algorithms are allowed to access (see Appendix \ref{append1} for a precise meaning). All the results proven in this paper can be easily extended to the case of inexact input. This means replacing the evaluation functions by
$$
f_{i,j,m}^{(1)},f_{i,j,m}^{(2)}:\mathcal{C}(l^2(\mathbb{N}))\rightarrow\mathbb{Q}+i\mathbb{Q}
$$
such that $|f_{i,j,m}^{(1)}(T)-\langle Te_j,e_i\rangle|\leq 2^{-m}$ and $|f_{i,j,m}^{(2)}(T)-\langle T^*e_j,T^*e_i\rangle|\leq 2^{-m}$. Hence, the existence results carry over to algorithms that are only allowed to perform arithmetic operations over $\mathbb{Q}$. This could be useful for rigorous bounds using interval arithmetic and computer-assisted proofs (for those familiar with the term, our algorithms are Turing recursive), though for brevity, we stick to $\Lambda_1$ and $\Lambda_2$ throughout. For discrete operators, the above information is often given to us, for example, in tight-binding models in physics or as a discretisation of a PDE, and hence it is natural to seek to compute spectral properties from matrix values. The set $\Lambda_2$ is motivated via variational problems. For partial differential operators, such information is often given through inner products with a suitable basis, and in this case, the inexact input model is needed due to approximating the integrals (see Appendix \ref{append2}). For the classes considered in this paper, the evaluation sets $\Lambda_1$ and $\Lambda_2$ are in general different, yet the classifications in the SCI remain the same.

We will be concerned operators whose matrix representation has a known \textit{asymptotic rate of column/off-diagonal decay}. Namely, let $f:\mathbb{N}\rightarrow\mathbb{N}$ with $f(n)>n$ and let $\alpha=\{\alpha_n\}_{n\in\mathbb{N}},\beta=\{\beta_n\}_{n\in\mathbb{N}}$ be null sequences\footnote{We use the term ``null sequence'' for a sequence converging to zero.} of non-negative real numbers. We then define for $\mathrm{X}=\mathrm{SA}$ or $\mathrm{X}=\mathrm{U}$,
\begin{equation}\label{omega_def_ann}
\begin{split}
\Omega_{f,\alpha,\beta}^{\mathrm{X}}&=\{T\in\Omega_{\mathrm{X}}:\|(P_{f(n)}-I)TP_n\|=O(\alpha_n),\text{ as }n\rightarrow\infty\}\\
&\quad\quad\quad\quad\quad\times\{x\in l^2(\mathbb{N}):\|P_nx-x\|=O(\beta_n),\text{ as }n\rightarrow\infty\},
\end{split}
\end{equation}
where $P_n$ denotes the orthogonal projection onto $\mathrm{span}\{e_1,...,e_n\}$. In using the $O(\cdot)$ notation, the hidden constant is allowed to depend on the operator $T$ or the vector $x$. We will also use
\begin{equation}
\label{nnddnnd}
\Omega_{f,\alpha}^{\mathrm{X}}=\{T\in\Omega_{\mathrm{X}}:\|(P_{f(n)}-I)TP_n\|=O(\alpha_n),\text{ as }n\rightarrow\infty\}.
\end{equation}
When discussing $\Omega_{f,\alpha,\beta}^{\mathrm{SA}}$ and $\Omega_{f,\alpha}^{\mathrm{SA}}$ we will use the notation $\Omega_{f,\alpha,\beta}$ and $\Omega_{f,\alpha}$. The collection of vectors in $l^2(\mathbb{N})$ satisfying $\|P_nx-x\|=O(\beta_n)$ will be denoted by $V_{\beta}$. Finally, when $\alpha_n\equiv 0$, we will abuse notation slightly in requiring the stronger condition that
$$
\|(P_{f(n)}-I)TP_n\|=0.
$$
Thus $\Omega_{f,0}$ is the class of self-adjoint operators whose matrix sparsity structure is captured by the function $f$. For example, if $f(n)=n+1$, we recover the class of self-adjoint tridiagonal matrices, the most studied class of infinite-dimensional operators. When discussing classes that include vectors $x$, we extend $\Lambda_i$ to include pointwise evaluations of the coefficients of $x$. Other additions are sometimes needed, such as data regarding open sets as inputs for computations of measures, but this will always be made clear. When considering the general case of $\Omega_{f,\alpha}$, the function $f$ and sequence $\alpha$ can also be considered as inputs to the algorithm - in other words, the same algorithm works for each class.

\subsection{Connections with previous work}
\label{prev_work}

We have mentioned the literature on infinite-dimensional spectral problems and the SCI hierarchy. Computationally, our point of view in this paper is closest to the work of Olver, Townsend and Webb on practical infinite-dimensional linear algebra \cite{Olver_Townsend_Proceedings, Olver_SIAM_Rev, Olver_code1, Olver_code2,webb2017spectra}. Their work includes efficient codes, such as the infinite-dimensional QL (IQL) algorithm \cite{webb_thesis} (see also \cite{colbrookinfinite} for the IQR algorithm, which has its roots in the work of Deift, Li and Tomei \cite{deift1985toda}), as well as theoretical results. A PDE version of the FEAST algorithm based on contour integration of the resolvent has recently been proposed by Horning and Townsend in \cite{horning2019feast}, which computes discrete spectra. The set of algorithms we provide can be considered as a new member within the growing family of infinite-dimensional techniques.

A similar, though different, object studied in the mathematical physics literature, particularly when considering random Schr\"odinger operators, is the density of states \cite{kirsch1989random,carmona2012spectral,kirsch2006integrated}, which we mention for completeness and also to avoid potential confusion. This object is defined via the ``thermodynamic limit'', where instead of considering the infinite-dimensional operator $T$, one considers finite truncations, say $P_nTP_n$, and the limit $n\rightarrow\infty$ of the measure $\sum_{x_j\in\sigma(P_nTP_n)} \delta_{x_j}/n$. This measure is subtly different from the spectral measure of $T$ (for instance, on the discrete spectrum). The density of states is an important quantity in quantum mechanics, and there is a large literature on its computation. We refer the reader to the excellent review article \cite{lin2016approximating}, which discusses the most common methods. The idea of using the resolvent to approximate the density of states of finite matrices can be found in the method of \cite{haydock1972electronic}, which approximates the imaginary part of $\mathrm{Trace}\left[{R}(z,P_nTP_n)\right]$ for $\mathrm{Im}(z)>0$. Similarly, in the random matrix literature, the connection is made through the Stieltjes transformation (see, for example, \cite{bai2010spectral}). There are three immediate differences between our algorithms and those that compute the density of states. First, we seek to deal with the full, infinite-dimensional, operator directly to compute the spectral measure (and not the limit of increasing system sizes). Second, the object we are computing contains more refined spectral information of the operator and does not involve an averaging procedure. The density of states does not capture the full spectral information, such as the contribution of eigenvalues in the discrete spectrum, whereas the projection-valued spectral measure does. Third, there is a subtlety regarding the limits as $\mathrm{Im}(z)$ goes to zero and the truncation parameter goes to infinity (a similar trade-off also occurs in random matrix literature when theoretically analysing the density of states - see \cite{erdos2012universality}). In our case, appropriate \textit{rectangular} truncations of the infinite-dimensional operator are required to compute the resolvent with error control (see Theorem \ref{res_est1}). This approach differs from finite-dimensional techniques, which typically consider square truncations.

For the $C^*$-algebra viewpoint of the density of states, we refer the reader to the work of Areveson \cite{Arveson_cnum_lin94} and the references therein. Estimating the spectrum of $T$ via $\sigma(P_nTP_n)$ is known as the finite section method. This has often been viewed in connection with Toeplitz theory. See, for example, the work by B{\"o}ttcher \cite{Albrecht_Fields, Bottcher_pseu}, B{\"o}ttcher and Silberman \cite{Bottcher_book}, B{\"o}ttcher, Brunner, Iserles and N{\o}rsett \cite{Arieh2}, Brunner, Iserles and N{\o}rsett \cite{Arieh1}, Hagen, Roch and Silbermann \cite{Silbermann22}, Lindner \cite{lindner2006infinite}, Marletta \cite{Marletta_pollution}, Marletta and Scheichl \cite{Marletta_Spec_gaps}, and Seidel \cite{seidel2014fredholm}.

The study of spectral measures also has a rich history in the theory of orthogonal polynomials and quadrature rules for numerical integration going back to the work of Szeg{\H o} \cite{szeg1939orthogonal,MR838253}, briefly touched upon in \S \ref{motiv}. In certain cases, one can recover a distribution function for the associated measure of the Jacobi operator as a limit of functions constructed using Gaussian quadrature \cite[Ch. 2]{chihara2011introduction}. Our examples in \S \ref{JACOBI_SEC_1} can be considered as a computational realisation of Favard's theorem.

There are several results in the literature considering the computation of spectral density functions for Sturm--Liouville problems. In the case of Sturm--Liouville problems, the spectral density function corresponds to the multiplicative version of the spectral theorem. This is subtly different from the measures we compute, which arise from the projection-valued measure version of the spectral theorem. A common approach to approximate spectral density functions associated with
Sturm--Liouville operators on unbounded domains is to truncate the domain and use the Levitan--Levinson formula, as implemented in the software package SLEDGE \cite{pruess1993mathematical,fulton1994parallel,fulton1998computation}. This approach can be computationally expensive since the eigenvalues cluster as the domain size increases; often, hundreds of thousands of eigenvalues and eigenvectors need to be computed. More sophisticated methods avoiding domain truncation are considered for special cases in \cite{fulton2005computing,fulton2008new}, and an application in plasma physics can be found in \cite{MR3323553}. These make use of the additional structure present in Sturm--Liouville problems using results analogous to (\ref{ctd_fraction}) in the continuous case. Our results, particularly Theorem \ref{WowPDE}, hold for partial differential operators much more complicated than Sturm--Liouville operators (see Appendix \ref{append2}).

Finally, we wish to highlight the work of Webb and Olver \cite{webb2017spectra}, which is of particular relevance to the present study. There the authors studied, through connection coefficients, Jacobi operators that arise as compact perturbations of Toeplitz operators. Similar perturbations (where a stronger exponential decay of the perturbation is crucial for analyticity properties of the resolvent) were studied in the work of Bilman and Trogdon \cite{bilman2017numerical} in connection with the inverse scattering transform for the Toda lattice. (See also the work of Trogdon, Olver and Deconinck \cite{trogdon2012numerical} for computations of spectral measures for inverse scattering for the KdV equation.) The results proven in \cite{webb2017spectra} can be stated in terms of the SCI hierarchy:
\begin{itemize}
	\item If the perturbation is finite rank (and known), the computation of $\sigma_{\mathrm{pp}}$ lies in $\Delta_1^G$, and the computation of the $\mu_{\mathrm{ac}}$ lies in $\Delta_0^G$ (note that $\sigma_{\mathrm{ac}}$ is known analytically).
	\item If the perturbation is compact with a known rate of decay at infinity, then the computation of the full spectrum $\sigma$ lies in $\Delta_1^G$.
\end{itemize}
The current paper complements the work of \cite{webb2017spectra} by; considering operators much more general than tridiagonal compact perturbations of Toeplitz operators (we deal with arbitrary self-adjoint operators and assume we know $f$ such that (\ref{nnddnnd}) holds) and partial differential operators, allowing operators to be unbounded, building algorithms that are arithmetic and can cope with inexact input, and considering computation of a wider range of spectral information. At the price of this greater generality, the objects we study are generally not computable with error control (unless one has local regularity assumptions on the measure - see \cite[Ch. 4]{colbrook2020foundations}), and some lead to computational problems higher up in the SCI hierarchy, though still computationally useful as we shall demonstrate. Our methods are also entirely different and rely on estimating the resolvent operator with error control (Theorem \ref{res_est1}).

\subsection{Organisation of the paper}

The paper is organised as follows. In \S \ref{res_stones} we consider the computation of the resolvent with error control and generalisations of Stone's formula. The computation of measures, their various decompositions and projections are discussed in \S \ref{ev_meas_sec}. We then discuss the functional calculus and density of measures in \S \ref{applic_sec}. The computation of the different types of spectra as sets in the complex plane is discussed in \S \ref{heavy_spec}. We run extensive numerical tests in \S \ref{num_sec}, where we also introduce a new collocation method for the computation of the Radon--Nikodym derivative. We find that increased rates of convergence can also be obtained through iterations of Richardson extrapolation. A summary of the SCI hierarchy can be found in Appendix \ref{append1} and a proof of Theorem \ref{WowPDE} in Appendix \ref{append2}. Throughout, our theorems and proofs use the notation introduced in \S \ref{sec_not} and \S \ref{sec_not2}.

\section{Preliminary Results}
\label{res_stones}

The algorithms built in this paper rely on the computation of the action of the resolvent operator $R(z,T)=(T-z)^{-1}$ for $z\notin\sigma(T)$ with (asymptotic) error control. Given this, one can compute the action of the projections $E^T_{S}$ for a wide range of sets $S$ (Theorem \ref{meas_comp1} and its generalisations), and hence the measures $\mu_{x,y}^T$. This section discusses the computation of the resolvent with error control and generalisations of Stone's formula, which relate the resolvent to the projection-valued measures.

\subsection{Approximating the resolvent operator}
\label{approx_res_sect}

The key theorem for computing the action of the resolvent operator is the following, where we use $\sigma_1$ to denote the injection modulus of an operator defined as
$$
\sigma_1(T):=\min\{\|Tx\|:x\in\mathcal{D}(T),\|x\|=1\}.
$$
The proof of the following theorem boils down to a careful computation of a least-squares solution of a rectangular linear system.

\begin{theorem}
\label{res_est1}
Let $\alpha=\{\alpha_n\}_{n\in\mathbb{N}}$ and $\beta=\{\beta_n\}_{n\in\mathbb{N}}$ be null sequences, and $f:\mathbb{N}\rightarrow \mathbb{N}$ with $f(n)>n$. Define
$$
S_{f,\alpha,\beta}:=\left\{(T,x,z)\in\Omega^{\mathrm{N}}_{f,\alpha,\beta}\times\mathbb{C}:z\in\mathbb{C}\backslash \sigma(T)\right\}
$$
and for $(T,x)\in \Omega^{\mathrm{N}}_{f,\alpha,\beta}$ (defined in (\ref{omega_def_ann})), let $C_1(T,x)$, $C_2(T,x)\in \mathbb{R}_{\geq 0}$ be such that
\begin{enumerate}
	\item $ \|(I-P_{f(n)})TP_n\|\leq C_1\alpha_n$,
	\item $\|P_nx-x\|\leq C_2\beta_n$,
\end{enumerate}
where, for notational convenience, we drop the $(T,x)$ dependence in the notation for $C_1$ and $C_2$. Then there exists a sequence of arithmetical algorithms
$$
\Gamma_n:S_{f,\alpha,\beta}\rightarrow l^2(\mathbb{N}),
$$
each of which use the evaluation functions in $\Lambda_1$ (defined in (\ref{evals_def_nene})), such that each vector $\Gamma_n(T,x,z)$ has finite support with respect to the canonical basis for each $n$, and
$$\lim_{n\rightarrow\infty}\Gamma_n(T,x,z)= R(z,T)x\text{ in }l^2(\mathbb{N}).$$
Moreover, the following error bound holds
\begin{equation}
\label{res_bound}
\|\Gamma_{n}(T,x,z)-R(z,T)x\|\leq \frac{C_2\beta_{f(n)}+C_1\alpha_n\|\Gamma_{n}(T,x,z)\|+\|P_{f(n)}(T-zI)\Gamma_{n}(T,x,z)-P_{f(n)}x\|}{\mathrm{dist}(z,\sigma(T))}.
\end{equation}
If a bound on $C_1$ and $C_2$ are known, this error bound can be computed to arbitrary accuracy using finitely many arithmetic operations and comparisons.
\end{theorem}

\begin{proof}
Let $(T,x,z)\in S_{f,\alpha,\beta}$. We have that $n=\mathrm{rank} (P_n)=\mathrm{rank} ((T-zI)P_n)=\mathrm{rank} (P_{f(n)}(T-zI)P_n)$ for large $n$ since $\sigma_1(T-zI)>0$ and $\|(I-P_{f(n)})(T-zI)P_n\|\leq C_1\alpha_n\rightarrow0$ (recall that $z\notin\sigma(T)$). Hence we can define
$$
\widetilde\Gamma_{n}(T,x,z):=\begin{cases}
0 \quad\quad\quad  \text{ if } \sigma_1(P_n(T^*-\overline{z}I)P_{f(n)}(T-zI)P_n)\leq\frac{1}{n}\\
[P_n(T^*-\overline{z}I)P_{f(n)}(T-zI)P_n]^{-1}P_n(T^*-\overline{z}I)P_{f(n)}x \quad \text{ otherwise.}
\end{cases}
$$
Suppose that $n$ is large enough so that $\sigma_1(P_n(T^*-\overline{z}I)P_{f(n)}(T-zI)P_n)>1/n$. Then $\widetilde\Gamma_{n}(T,x,z)$ is a least-squares solution of the optimisation problem $\mathrm{argmin}_{y}\|P_{f(n)}(T-zI)P_ny-x\|$. The linear space $\mathrm{span}\{e_n:n\in\mathbb{N}\}$ forms a core of $T$ and hence also of $T-zI$. It follows by invertibility of $T-zI$ that given any $\epsilon>0$, there exists an $m=m(\epsilon)$ and a $y=y(\epsilon)$ with $P_my=y$ such that
$$
\|(T-zI)y-x\|\leq\epsilon.
$$
It follows that for all $n\geq m$,
\begin{align*}
\|(T-zI)\widetilde\Gamma_{n}(T,x,z)-x\|&\leq \|P_{f(n)}(T-zI)\widetilde\Gamma_{n}(T,x,z)-x\|+C_1\alpha_n\|\widetilde\Gamma_{n}(T,x,z)\|\\
&\leq \|P_{f(n)}(T-zI)y-x\|+C_1\alpha_n\|\widetilde\Gamma_{n}(T,x,z)\|\\
&\leq \|P_{f(n)}(T-zI)y-P_{f(n)}x\|+C_2\beta_{f(n)}+C_1\alpha_n\|\widetilde\Gamma_{n}(T,x,z)\|\\
&\leq \epsilon +C_2\beta_{f(n)}+C_1\alpha_n\|\widetilde\Gamma_{n}(T,x,z)\|.
\end{align*}
This implies that
\begin{align*}
\|\widetilde\Gamma_{n}(T,x,z)-R(z,T)x\|&\leq \|R(z,T)\|\|(T-zI)\widetilde\Gamma_{n}(T,x,z)-x\|\\
&\leq \|R(z,T)\|\left(\epsilon +C_2\beta_{f(n)}+C_1\alpha_n\|\widetilde\Gamma_{n}(T,x,z)\|\right).
\end{align*}
In particular, since $\alpha$ and $\beta$ are null, this implies that $\|\widetilde\Gamma_{n}(T,x,z)\|$ is uniformly bounded in $n$. Since $\epsilon>0$ was arbitrary, we also see that $\widetilde\Gamma_{n}(T,x,z)$ converges to $R(z,T)x$.

Define the matrices
$$
B_n=P_n(T^*-\overline{z}I)P_{f(n)}(T-zI)P_n,\quad C_n=P_n(T^*-\overline{z}I)P_{f(n)}.
$$
Given the evaluation functions in $\Lambda_1$, we can compute the entries of these matrices to any given accuracy and hence also to arbitrary accuracy in the operator norm using finitely many arithmetic operations and comparisons (using the error in the Frobenius norm to bound the error in the operator norm). Denote approximations of $B_n$ and $C_n$ by $\widetilde B_n$ and $\widetilde C_n$ respectively and assume that
$$
\|B_n-\widetilde B_n\|\leq u_n,\quad \|C_n-\widetilde C_n\| \leq v_n,
$$
for null sequences $\{u_n\},\{v_n\}$. Note that $\widetilde B_n^{-1}$ can be computed using finitely many arithmetic operations and comparisons. So long as $u_n$ is small enough, the resolvent identity implies that
$$
\|B_n^{-1}-\widetilde B_n^{-1}\|\leq \frac{\|\widetilde B_n^{-1}\|^2u_n}{1-u_n\|\widetilde B_n^{-1}\|}=:w_n.
$$
By taking $u_n$ and $v_n$ smaller if necessary (so that the algorithm is adaptive and it is straightforward to bound the norm of a finite matrix from above), we can ensure that $\|\widetilde B_n^{-1}\|v_n\leq n^{-1}$ and $(\|\widetilde C_n\|+v_n)w_n\leq n^{-1}$. From Proposition \ref{PCholesky} and a simple search routine, we can also compute $\sigma_1(P_n(T^*-\overline{z}I)P_{f(n)}(T-zI)P_n)$ to arbitrary accuracy using finitely many arithmetic operations and comparisons. Suppose this is done to an accuracy $1/n^2$ and denote the approximation via $\tau_n$. We then define
$$
\Gamma_{n}(T,x,z):=\begin{cases}
\quad\quad\quad 0 & \text{ if } \tau_n\leq\frac{1}{n}\\
\widetilde B_n^{-1}\widetilde C_n \widetilde x_n& \text{ otherwise,}
\end{cases}
$$
where $\widetilde x_n=P_{f(n)}x$. It follows that $\Gamma_{n}(T,x,z)$ can be computed using finitely many arithmetic operations and, for large $n$,
$$
\|\Gamma_{n}(T,x,z)-\widetilde\Gamma_{n}(T,x,z)\|\leq \left(\|\widetilde B_n^{-1}\|v_n+(\|\widetilde C_n\|+v_n)w_n\right)\|x\|\rightarrow 0,
$$
so that $\Gamma_n(T,x,z)$ converges to $R(z,T)x$. By construction, $\Gamma_n(T,x,z)$ has finite support with respect to the canonical basis.

Furthermore, the following error bound holds (which also holds if $\tau_n\leq1/n$)
\begin{align*}
\|\Gamma_{n}(T,x,z)-R(z,T)x\|&\leq \|R(z,T)\|\|(T-zI)\Gamma_{n}(T,x,z)-x\|\\
&\leq \frac{C_2\beta_{f(n)}+C_1\alpha_n\|\Gamma_{n}(T,x,z)\|+\|P_{f(n)}(T-zI)\Gamma_{n}(T,x,z)-P_{f(n)}x\|}{\mathrm{dist}(z,\sigma(T))},
\end{align*}
since $T$ is normal so that $\|R(z,T)\|=\mathrm{dist}(z,\sigma(T))^{-1}$. This bound converges to $0$ as $n\rightarrow\infty$. If the $C_1$ and $C_2$ are known it can be approximated to arbitrary accuracy using finitely many arithmetic operations and comparisons.
\end{proof}

\begin{remark}
If $T$ corresponds to a choice of basis in a space of functions (for example when using a spectral method), there is often a link between the regularity of the functions $x$ and the decay of the terms $\beta_n$. The bound (\ref{res_bound}) can then often be adapted to include such asymptotics, and hence indicate how large $n$ needs to be to gain a given approximation.
\end{remark}

Of course, a vast literature exists on computing $R(z,T)$, especially for infinite matrices with structure (such as being banded) and we refer the reader to \cite{lindner2006infinite,rabinovich2012limit,seidel2014fredholm,grochenig2013norm} for a small sample. Note that if $T$ is banded with bandwidth $m$, then we can take $f(n)=n+m$ and the above computation can be done in $O(nm^2)$ operations \cite{golub2012matrix}. The following corollary of Theorem \ref{res_est1} will be used repeatedly in the following proofs.

\begin{corollary}
\label{res_est2}
There exists a sequence of arithmetic algorithms
$$
\Gamma_n: \Omega_{f,\alpha,\beta}\times\mathbb{C}\backslash\mathbb{R}\rightarrow l^2(\mathbb{N})
$$
with the following properties:
\begin{enumerate}
	\item For all $(T,x)\in\Omega_{f,\alpha,\beta}$ and $z\in\mathbb{C}\backslash\mathbb{R}$, $\Gamma_n(T,x,z)$ has finite support with respect to the canonical basis and converges to $R(z,T)x$ in $l^2(\mathbb{N})$ as $n\rightarrow\infty$.
	\item For any $(T,x)\in\Omega_{f,\alpha,\beta}$, there exists a constant $C(T,x)$ such that for all $z\in\mathbb{C}\backslash\mathbb{R}$,
	$$
	\|\Gamma_n(T,x,z)-R(z,T)x\|\leq \frac{C(T,x)}{\left|\mathrm{Im}(z)\right|}\left[\alpha_n+\beta_n\right].
	$$
\end{enumerate}
\end{corollary}
\begin{proof}
Let $\Gamma_n(T,x,z)=\widehat \Gamma_{m(n,T,x,z)}(T,x,z)$, where $\widehat \Gamma_k$ are the algorithms from the statement of Theorem \ref{res_est1} and $m(n,T,x,z)$ is a subsequence diverging to infinity as $n\rightarrow\infty$. Clearly statement (1) holds so we must show how to choose the sequence $m(n,T,x,z)$ such that (2) holds (and hence our algorithms will be adaptive). From (\ref{res_bound}), it is enough to show that $m=m(n,T,x,z)$ can be chosen such that
$$
\beta_{f(m)}+\alpha_m\|\widehat\Gamma_{m}(T,x,z)\|+\|P_{f(m)}(T-zI)\widehat\Gamma_{m}(T,x,z)-P_{f(m)}x\|\lesssim \alpha_n+\beta_n.
$$
The left-hand side can be approximated to arbitrary accuracy using finitely many arithmetic operations and comparisons. Hence by repeatedly computing approximations to within $\alpha_n+\beta_n$, we can choose the minimal $m$ such that these approximate bounds are at most $2(\alpha_n+\beta_n)$.
\end{proof}

\subsection{Stone's formula and Poisson kernels}
\label{stone_sec}

Here we briefly discuss Stone's famous formula \cite{stone1932linear,coddington1955theory,reed1972methods}, which relates the convolution of spectral measures with Poisson kernels to the pointwise action of the projection-valued measures associated with an operator $T\in\Omega_{\mathrm{SA}}$ as $\epsilon\downarrow 0$ (see \S \ref{motiv}). Stone's formula can also be generalised to unitary operators and a much larger class of normal operators (see Proposition \ref{stone2}). We include a short (and standard) proof of Proposition \ref{stone1} for the benefit of the reader.

\begin{proposition}[Stone's formula]
\label{stone1}
The following boundary limits hold:
\begin{itemize}
	\item [(i)] Let $T\in\Omega_{\mathrm{SA}}$. Then for any $-\infty\leq a<b\leq\infty$ and $x\in l^2(\mathbb{N})$,
	$$
	\lim_{\epsilon\downarrow 0}\int_{a}^bK_H(u+i\epsilon;T,x)du=E^{T}_{(a,b)}x+\frac{1}{2}E^{T}_{\{a,b\}}x.
	$$
	\item [(ii)] Let $T\in\Omega_{\mathrm{U}}$. Then for any $0\leq a<b<2\pi$ and $x\in l^2(\mathbb{N})$,
	$$
	\lim_{\epsilon\downarrow 0}\int_{a}^bi\exp(i\psi)K_D((1-\epsilon)\exp(i\psi);T,x)d\psi=E^{T}_{(a,b)_{\mathbb{T}}}x+\frac{1}{2}E^{T}_{\{\exp(ia),\exp(ib)\}}x,
	$$
	where $(a,b)_{\mathbb{T}}$ denotes the image of $(a,b)$ under the map $\theta\rightarrow\exp(i\theta)$.
\end{itemize}
\end{proposition}

\begin{proof}
To prove (i), we can apply Fubini's theorem to interchange the order of integration and arrive at
$$
\int_{a}^bK_H(u+i\epsilon;T,x)du=\left[\int_{-\infty}^\infty\int_a^bP_H(u-\lambda ,\epsilon)du\ dE^T(\lambda )\right]x
$$
But
$$
\int_a^bP_H(u-\lambda ,\epsilon)du=\frac{1}{\pi}\left[\tan^{-1}\left(\frac{b-\lambda }{\epsilon}\right)-\tan^{-1}\left(\frac{a-\lambda }{\epsilon}\right)\right]
$$
is bounded and converges pointwise as $\epsilon\downarrow 0$ to $\chi_{(a,b)}(\lambda )+\chi_{\{a,b\}}(\lambda )/2$, where $\chi_S$ denotes the indicator function of a set $S$. Part (i) now follows from the dominated convergence theorem.

To prove (ii), we apply Fubini's theorem again, now noting that
\begin{equation}
\label{bd11}
\begin{split}
\int_a^bi\exp(i\psi)P_D((1-\epsilon),\psi-\theta)d\psi&=\frac{i\exp(i\theta)}{2\pi}\int_{a-\theta}^{b-\theta}\frac{(2\epsilon-\epsilon^2)\exp(i\psi)}{\epsilon^2+2(1-\epsilon)(1-\cos(\psi))}d\psi.
\end{split}
\end{equation}
We can split the interval into small intervals of width $\rho$ (where $0<\rho<1$) around each point where $\cos(\psi)=1$, and a finite union of intervals on which $1-\cos(\psi)$ is positive, bounded away from $0$. On these later intervals, the limit vanishes as $\epsilon\downarrow0$. Hence by periodicity and considering odd and even parts, we are left with considering
$$
I_1(\rho,\epsilon)=\int_0^\rho\frac{(2\epsilon-\epsilon^2)\cos(\psi)}{\epsilon^2+2(1-\epsilon)(1-\cos(\psi))}d\psi,\quad I_2(\rho,\epsilon)=\int_0^\rho\frac{(2\epsilon-\epsilon^2)\sin(\psi)}{\epsilon^2+2(1-\epsilon)(1-\cos(\psi))}d\psi.
$$
Explicit integration yields $I_2(\epsilon,\rho)=O(\epsilon\log(\epsilon))$ and hence the contribution vanishes in the limit. We also have
$$
I_1(\rho,\epsilon)=\frac{(\epsilon^2-2\epsilon)\rho+2(2+\epsilon^2-2\epsilon)\tan^{-1}\left(\frac{(2-\epsilon)\tan\left(\frac{\rho}{2}\right)}{\epsilon}\right)}{2(1-\epsilon)}.
$$
This converges to $\pi$ as $\epsilon\downarrow 0$. Considering the contributions of $I_1$ and $I_2$ in (\ref{bd11}), we see that (\ref{bd11}) converges pointwise as $\epsilon\downarrow 0$ to
$$
i\exp(i\theta)\left\{\chi_{(a,b)}(\theta)+[\chi_{\{a\}}(\theta)+\chi_{\{b\}}(\theta)]/2\right\}.
$$
Since the integral is also bounded, part (ii) now follows from the dominated convergence theorem and change of variables.
\end{proof}

This type of construction can be generalised to $T\in\Omega_{\mathrm{N}}$ whose spectrum lies on a regular enough curve. However, it is much more straightforward in the general case to use the analytic properties of the resolvent. The next proposition does this and also holds for operators whose spectrum does not necessarily lie along a curve.

\begin{figure}
\centering
\begin{tabular}{c c}
(a) & (b)\\
\includegraphics[width=0.6\textwidth,trim={0mm 0mm 0mm 2mm},clip]{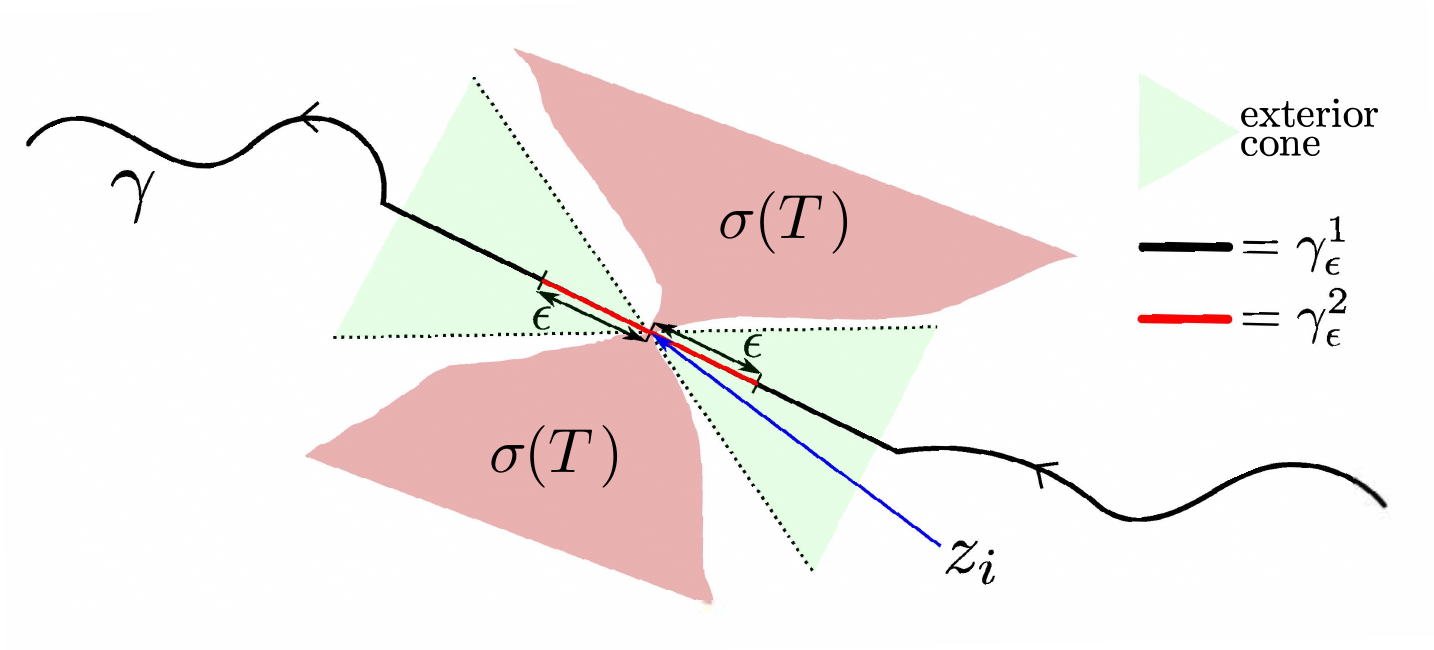} &
\includegraphics[width=0.38\textwidth,trim={0mm -10mm 0mm 0mm},clip]{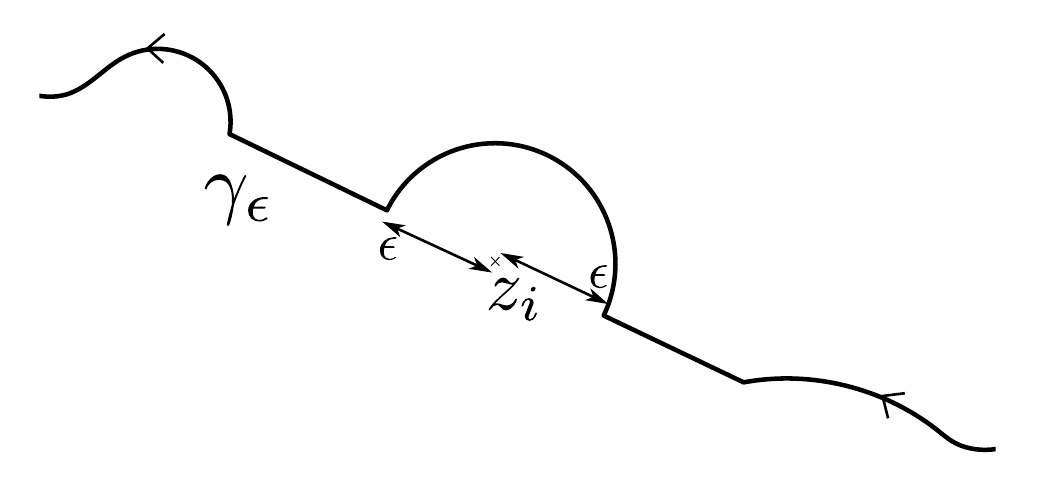}
\end{tabular}
\caption{Left: Exterior cone condition for Proposition \ref{stone2}. Right: Deformed contour $\gamma_{\epsilon}$ to compute $f_{\epsilon}(z_i)$.}
\label{gen_stone}
\end{figure}

\begin{proposition}[Generalised Stone's formula]
\label{stone2}
Let $T\in\Omega_{\mathrm{N}}$ and $\gamma$ be a rectifiable positively oriented Jordan curve with the following properties. The spectrum $\sigma(T)$ intersects $\gamma$ at finitely many points $z_1,...,z_m$ and in a neighbourhood of each of the $z_i$, $\gamma$ is formed of a line segment meeting $\sigma(T)$ only at $z_i$, at which point $\sigma(T)$ has a local exterior cone condition with respect to $\gamma$ (see Figure \ref{gen_stone}). Let $x\in l^2(\mathbb{N})$. Then we can define the Cauchy principal value integral of the resolvent $R(z,T)x$ along $\gamma$ and have
\begin{equation}
\label{stoneGEN}
\frac{-1}{2\pi i}\mathrm{PV}\int_{\gamma} R(z,T)xdz=E^T_{\sigma(T;\gamma)}x-\frac{1}{2}\left[\sum_{j=1}^mE^T_{\{z_j\}}x\right],
\end{equation}
where $\sigma(T;\gamma)$ is the closure of the intersection of $\sigma(T)$ with the interior of $\gamma$.
\end{proposition}

\begin{proof}
We will argue for the case $m=1$, and the general case follows in exactly the same manner. Let $\epsilon>0$ be small so that in a neighbourhood of the $\epsilon-$ball around $z_1$, $\gamma$ is given by a straight line. We then decompose $\gamma$ into two disjoint parts
$$
\gamma=\gamma_\epsilon^1\cup \gamma_\epsilon^2,
$$
where $\gamma_\epsilon^2$ denotes the line segment of $\gamma$ at most $\epsilon$ away from $z_1$ (as shown in Figure \ref{gen_stone}). We set
$$
F_\epsilon(x,T)=\int_{\gamma_\epsilon^1} R(z,T)xdz=\left[\int_{\sigma(T)}\int_{\gamma_\epsilon^1}\frac{1}{\lambda -z}dz dE^T(\lambda ) \right]x.
$$
We then consider the inner integral
$$
f_\epsilon(\lambda )=\int_{\gamma_\epsilon^1}\frac{1}{\lambda -z}dz.
$$
If $\lambda $ is inside $\gamma$ then $\lim_{\epsilon\downarrow 0}f_\epsilon(\lambda )=-2\pi i$ via Cauchy's residue theorem. Similarly, if $\lambda $ is outside $\gamma$ then $\lim_{\epsilon\downarrow 0}f_\epsilon(\lambda )=0$. To calculate $f_{\epsilon}(z_1)$, consider the contour integral along $\gamma_{\epsilon}$ in Figure \ref{gen_stone}. We see that
$$
f_{\epsilon}(z_1)-i\pi=-2i\pi
$$
and hence $f_{\epsilon}(z_1)=-i\pi$. We would like to apply the dominated convergence theorem. Clearly, away from $z_1$, $f_{\epsilon}$ is bounded as $\epsilon\downarrow 0$. Now let $0<\delta<\epsilon$ then
$$
f_{\delta}(\lambda )-f_{\epsilon}(\lambda )=\int_\delta^\epsilon\frac{1}{\frac{\lambda -z_1}{w}-s}+\frac{1}{\frac{\lambda -z_1}{w}+s}ds=\log\left(\frac{\epsilon+\frac{\lambda -z_1}{w}}{-\epsilon+\frac{\lambda -z_1}{w}}\right)-\log\left(\frac{\delta+\frac{\lambda -z_1}{w}}{-\delta+\frac{\lambda -z_1}{w}}\right)
$$
for some $w\in\mathbb{T}$. Taking the pointwise limit $\delta\downarrow0$, we see that $f_{\epsilon}(\lambda )$ is bounded for $\lambda \in\sigma(T)$ in a neighbourhood of $z_1$ as $\epsilon\downarrow 0$ if the same holds for
$$
g_\epsilon(\lambda )=\log\left(\frac{\epsilon+\frac{\lambda -z_1}{w}}{-\epsilon+\frac{\lambda -z_1}{w}}\right).
$$
By rotating and translating, we can assume that $w=1$ and $z_1=0$ without loss of generality. Let $\lambda_1=\mathrm{Re}(\lambda )$ and $\lambda_2=\mathrm{Im}(\lambda )$. Using the cone condition gives $\alpha\left|\lambda_1\right|\leq\left|\lambda_2\right|$ for some $\alpha>0$. Assume $\lambda_1\neq 0$ then
$$
\left|\frac{\epsilon+\lambda}{-\epsilon+\lambda}\right|^2=\frac{(\epsilon+\lambda_1)^2+\lambda_2^2}{(\epsilon-\lambda_1)^2+\lambda_2^2}=1+\frac{4x}{(x-1)^2+y^2},
$$
where $x=\epsilon/\lambda_1$ and $y=\lambda_2/\lambda_1$. Note that $y^2\geq\alpha^2$ and without loss of generality we take $y\geq\alpha$. Define
$$
h(x,y)=\frac{4x}{(x-1)^2+y^2}
$$
Note that $h(x,y)\rightarrow 0$ as $\left|x\right|^2+\left|y\right|^2\rightarrow\infty$. We must show that $h(x,y)$ is bounded above $-1$ for $y\geq\alpha$. It is enough to consider points where $\partial h/\partial x=0$ which occur when $x_{\pm}=\pm\sqrt{1+y^2}.$ We have
$$
h(x_{\pm},y)=\frac{\pm 2}{\sqrt{1+y^2}\mp 1}\geq\frac{-2}{\sqrt{1+\alpha^2}+1}>-1,
$$
and hence we have proved the required boundedness. We then define
$$
\mathrm{PV}\int_{\gamma} R(z,T)xdz=\lim_{\epsilon\downarrow 0}F_\epsilon(x,T).
$$
The relation (\ref{stoneGEN}) now follows from the dominated convergence theorem.
\end{proof}

\section{Computation of Measures}
\label{ev_meas_sec}

For the sake of brevity, the analysis in the rest of this paper will consider the self-adjoint case $T\in\Omega_{\mathrm{SA}}$, which is the case most encountered in applications. However, the algorithms we build are based on Theorem \ref{res_est1} (and Corollary \ref{res_est2}) and the link with Poisson kernels/Cauchy transforms. Given the relation (\ref{Poiss_int3}) and Proposition \ref{stone2}, many of the results can be straightforwardly extended to the unitary case and more general cases where conditions similar to that of Proposition \ref{stone2} hold. We consider examples of unitary operators in \S \ref{CMV_numerics}.

\subsection{Full spectral measure}
\label{full_meas_pls}

We start by considering the computation of $E_U^Tx$, where $U\subset\mathbb{R}$ is a non-trivial open set. In other words, $U$ is not the whole of $\mathbb{R}$ or the empty set. The collection of these subsets will be denoted by $\mathcal{U}$. To be precise, we assume that we have access to a finite or countable collection $a_m(U),b_m(U)\in\mathbb{R}\cup\{\pm\infty\}$ such that $U$ can be written as a disjoint union
\begin{equation}
\label{open_union}
U=\bigcup_{m}\left(a_m(U),b_m(U)\right).
\end{equation}
With an abuse of notation, we add this information as evaluation functions to $\Lambda_i$ (defined in (\ref{evals_def_nene})).

\begin{theorem}[Computation of measures on open sets]
\label{meas_comp1}
Given the set-up in \S \ref{sec_not}, \S \ref{sec_not2} and the previous paragraph, consider the map
\begin{align*}
\Xi_{\mathrm{meas}}:\Omega_{f,\alpha,\beta}\times \mathcal{U}&\rightarrow l^2(\mathbb{N})\\
(T,x,U)&\rightarrow E_U^Tx.
\end{align*}
Then $\{\Xi_{\mathrm{meas}},\Omega_{f,\alpha,\beta}\times \mathcal{U}, \Lambda_1\}\in\Delta_2^A$. In other words, we can construct a convergent sequence of arithmetic algorithms for the problem.
\end{theorem}

\begin{remark}
\label{TKIP}
Essentially, this theorem tells us that if we can compute the action of the resolvent operator with asymptotic error control near the real axis, then we can compute the spectral measures of open sets in one limit. In the unitary case, this can easily be extended to relatively open sets of $\mathbb{T}$ if we can evaluate the resolvent near the unit circle. For any $U\in\mathcal{U}$, the approximation of $E_U^Tx$ has finite support, and hence we can take inner products to compute $\mu_{x,y}^T(U)$.
\end{remark}

\begin{remark}
One may wonder whether it is possible to upgrade the convergence of the algorithm in Theorem \ref{meas_comp1} from $\Delta_2$ to $\Delta_1$. In other words, whether it is possible to compute the measure with error control. However, this is difficult because the measure may be singular. Theorem \ref{meas_err_impossible} shows this is impossible even for singleton sets and discrete Schr\"odinger operators acting on $l^2(\mathbb{N})$.
\end{remark}

\begin{proof}[Proof of Theorem \ref{meas_comp1}]
Let $T\in \Omega_{\mathrm{SA}}$ and $z_1,z_2\in\mathbb{C}\backslash\mathbb{R}$. By the resolvent identity and self-adjointness of $T$,
$$
\|R(z_1,T)-R(z_1,T)\|\leq \left|\mathrm{Im}(z_1)\right|^{-1}\left|\mathrm{Im}(z_2)\right|^{-1}\left|z_1-z_2\right|.
$$
Hence, for $z=u+i\epsilon$ with $\epsilon>0$, the vector-valued function $K_H(u+i\epsilon;T,x)$ (considered with argument $u$) is Lipschitz continuous with Lipschitz constant bounded by $\epsilon^{-2}\|x\|/\pi$. Now consider the class $\Omega_{f,\alpha,\beta}\times \mathcal{U}$ and let $(T,x,U)\in\Omega_{f,\alpha,\beta}\times \mathcal{U}$. From Corollary \ref{res_est2}, we can construct a sequence of arithmetic algorithms, $\widehat\Gamma_n$, such that
$$
\|\widehat\Gamma_n(T,u,z)-K_H(u+i\epsilon;T,x)\|\leq\frac{C(T,x)}{\epsilon}\left(\alpha_n+\beta_n\right)
$$
for all $(T,x)\in\Omega_{f,\alpha,\beta}$. It follows from standard quadrature rules and taking subsequences if necessary (using that $\{\alpha_n\}$ and $\{\beta_n\}$ are null), that for $-\infty<a<b<\infty$, the integral
\begin{equation}
\label{stone22}\int_{a}^b K_H\left(u+\frac{i}{n};T,x\right)du  
\end{equation} 
can be approximated to an accuracy $\widehat C(T,x)/n$ using finitely many arithmetic operations and comparisons and the relevant set of evaluation functions $\Lambda_1$ (the constant $C$ now becomes $\widehat C$ due to not knowing the exact value of $\|x\|$). 

Recall that we assumed the disjoint union
$$
U=\bigcup_{m}(a_m,b_m)
$$
where $a_m,b_m\in\mathbb{R}\cup\{\pm\infty\}$ and the union is at most countable. Without loss of generality, we assume that the union is over $m\in\mathbb{N}$. We then let $a_{m,n},b_{m,n}\in\mathbb{Q}$ be such that $a_{m,n}\downarrow a_m$ and $b_{m,n}\uparrow b_m$ as $n\rightarrow \infty$ with $a_{m,n}<b_{m,n}$ and hence $(a_{m,n},b_{m,n})\subset (a_m,b_m)$. Let
\begin{equation*}
\label{Un_def}
U_n=\bigcup_{m=1}^n (a_{m,n},b_{m,n}),
\end{equation*}
then the proof of Stone's formula in Proposition \ref{stone1} (essentially an application of the dominated convergence theorem) can be easily adapted to show that
\begin{equation*}
\label{stone23}
\lim_{n\rightarrow \infty}\int_{U_n} K_H\left(u+\frac{i}{n};T,x\right)du=E^{T}_{U}x.  
\end{equation*} 
Note that we do not have to worry about contributions from endpoints of the intervals $(a_m,b_m)$ since we approximate strictly from within the open set $U$. To finish the proof, we simply let $\Gamma_{n}(T,x,U)$ be an approximation of the integral
$$
\int_{U_n} K_H\left(u+\frac{i}{n};T,x\right)du
$$
with accuracy $\widehat C(T,x)/n$. By the above remarks, such an approximation can be computed using finitely many arithmetic operations and comparisons from the relevant set of evaluation functions $\Lambda_1$.
\end{proof}

This theorem can clearly be extended to cover the more general case of Proposition \ref{stone2} if $\gamma$ is regular enough to allow approximation of
$$
\mathrm{PV}\int_{\gamma} R(z,T)xdz,
$$
given the ability to compute $R(z,T)x$ with asymptotic error control. Note that when it comes to numerically computing the integrals in Propositions \ref{stone1} and \ref{stone2}, it is advantageous to deform the contour so that most of the contour lies far from the spectrum so that the resolvent has a smaller Lipschitz constant. The proof can also be adapted to compute $E_{I}x$, where $I=[a,b]$ is a closed interval, by considering intervals shrinking to $[a,b]$ ($a,b$ finite). A special case of this is the computation of the spectral measure of singleton sets. However, for these it much easier to directly use the formulae
$$
E^{T}_{\{u\}}x=\lim_{\epsilon\downarrow 0}\epsilon \pi K_H(u+i\epsilon;T,x),\quad E^{T}_{\{\exp(i\theta)\}}x=\lim_{\epsilon\downarrow 0}\epsilon\pi i\exp(i\theta) K_D((1-\epsilon)\exp(i\theta);T,x),
$$
for $T\in\Omega_{\mathrm{SA}}$ and $T\in\Omega_{\mathrm{U}}$ respectively.

\subsection{Measure decompositions and projections}
\label{measu_deomps}

Recall from \S \ref{sec_not} that $P_{\mathcal{I}}^T$ denotes the orthogonal projection onto the space $\mathcal{H}^T_{\mathcal{I}}$, where $\mathcal{I}$ denotes a generic type ($\mathrm{ac},\mathrm{sc},\mathrm{pp},\mathrm{c}$ or $\mathrm{s}$). We have included the continuous and singular parts denoted by $\mathrm{c}$ or $\mathrm{s}$ which correspond to $\mathcal{H}_{\mathrm{ac}}\oplus\mathcal{H}_{\mathrm{sc}}$ and $\mathcal{H}_{\mathrm{sc}}\oplus\mathcal{H}_{\mathrm{pp}}$ respectively. These are often encountered in mathematical physics. As in \S \ref{full_meas_pls}, we assume the decomposition in (\ref{open_union}) and add the $\{a_m,b_m\}$ as evaluation functions to $\Lambda_i$ (defined in (\ref{evals_def_nene})). In this section, we prove the following theorem.

\begin{theorem}
\label{spec_decomp_comp}
Given the set-up in \S \ref{sec_not}, \S \ref{sec_not2} and \S \ref{full_meas_pls}, consider the map
\begin{align*}
\Xi_{\mathcal{I}}:\Omega_{f,\alpha,\beta}\times V_{\beta}\times \mathcal{U}&\rightarrow\mathbb{C}\\
(T,x,y,U)&\rightarrow \langle P_{\mathcal{I}}^TE_U^Tx,y\rangle=\mu^T_{x,y,\mathcal{I}}(U),
\end{align*}
for $\mathcal{I}=\mathrm{ac},\mathrm{sc},\mathrm{pp},\mathrm{c}$ or $\mathrm{s}$. Then for $i=1,2$
$$
\Delta^G_2 \not\owns  \{\Xi_{\mathcal{I}},\Omega_{f,\alpha,\beta}\times V_{\beta}\times \mathcal{U},\Lambda_i\} \in \Delta^A_3.
$$
\end{theorem}

To prove this theorem, it is enough, by the polarisation identity, to consider $x=y$ (note that all the projections commute). We will split the proof into two parts: the $\Delta^A_3$ inclusion, for which it is enough to consider $\Lambda_1$, and the $\Delta^G_2$ exclusion, for which it is enough to consider $\Lambda_2$.

\subsubsection{Proof of inclusion in Theorem \ref{spec_decomp_comp}}

\begin{proof}[Proof of inclusion in Theorem \ref{spec_decomp_comp}]
Since $P^T_{\mathrm{pp}}=I-P^T_{\mathrm{c}}$, $P^T_{\mathrm{ac}}=I-P^T_{\mathrm{s}}$ and $P^T_{\mathrm{sc}}=P^T_{\mathrm{s}}-P^T_{\mathrm{pp}}$, it is enough, by Theorem \ref{meas_comp1} and Remark \ref{TKIP}, to consider only $\mathcal{I}=\mathrm{c}$ and $\mathcal{I}=\mathrm{s}$.

\textbf{Step 1}: We first deal with $\mathcal{I}=\mathrm{c}$, where we shall use a similar argument to the proof of Theorem \ref{F_calc} (which is more general than what we need). We recall the RAGE theorem \cite{ruelle1969remark,amrein1973characterization,enss1978asymptotic} as follows. Let $Q_n$ denote the orthogonal projection onto vectors in $l^2(\mathbb{N})$ with support outside the subset $\{1,...,n\}\subset \mathbb{N}$. Then for any $x\in l^2(\mathbb{N})$,
\begin{equation}
\label{Ruelle2}
\begin{split}
\langle P_{\mathrm{c}}^TE_U^Tx,x\rangle=\|P_{\mathrm{c}}^TE_U^Tx\|^2&=\lim_{n\rightarrow\infty}\lim_{t\rightarrow\infty}\frac{1}{t}\int_0^t\left\|Q_ne^{-iTs}E_U^Tx\right\|^2 ds\\
&=\lim_{n\rightarrow\infty}\lim_{t\rightarrow\infty}\frac{1}{t}\int_0^t\left\|Q_ne^{-iTs}\chi_{U}(T)x\right\|^2 ds.
\end{split}
\end{equation}
The proof of Theorem \ref{F_calc} is easily adapted to show that there exists arithmetic algorithms $\widetilde\Gamma_{n,m}$ using $\Lambda_1$ such that
$$
\|Q_ne^{-iTs}\chi_{U}(T)x-\widetilde\Gamma_{n,m}(T,x,U,s)\|\leq\frac{C(T,x,U)}{m}
$$
for all $(T,x,U,s)\in\Omega_{f,\alpha,\beta}\times\mathcal{U}\times \mathbb{R}$. Note that this bound can be made independent of $s$ (as we have written above) by sufficiently approximating the function $\lambda \rightarrow\exp(-i\lambda s)\chi_{U}(\lambda )$ (it has known total variation for a given $s$ and uniform bound). We now define
$$
\Gamma_{n,m}(T,x,U)=\frac{1}{m^2}\sum_{j=1}^{m^2}\|\widetilde\Gamma_{m,n}(T,x,U,j/m)\|^2.
$$
Using the fact that for $a,b\in l^2(\mathbb{N})$,
\begin{equation}
\label{useful_square}
\left|\langle a,a \rangle -\langle b,b \rangle\right|\leq \|a-b\|\left(2\|a\|+\|a-b\|\right),
\end{equation}
it follows that
$$
\left|\|Q_ne^{-iTs}\chi_{U}(T)x\|^2-\|\widetilde\Gamma_{n,m}(T,x,U,s)\|^2\right|\leq\frac{C(T,x,U)}{m}\left(2\|x\|+\frac{C(T,x,U)}{m}\right).
$$

Hence
\begin{align*}
\left|\Gamma_{n,m}(T,x,U)-\frac{1}{m}\int_0^m\left\|Q_ne^{-iTs}\chi_{U}(T)x\right\|^2 ds\right|&\leq \frac{1}{m^2}\sum_{j=1}^{m^2}\frac{C(T,x,U)}{m}\left(2\|x\|+\frac{C(T,x,U)}{m}\right)\\
&+\frac{1}{m^2}\sum_{j=1}^{m^2}\left|g_n(j/m)-m\int_{\frac{j-1}{m}}^{\frac{j}{m}}g_n(s)ds\right|,
\end{align*}
where $g_n(s)=\|Q_ne^{-iTs}\chi_{U}(T)x\|^2.$ Clearly the first term converges to $0$ as $m\rightarrow\infty$, so we only need to consider the second. Using (\ref{useful_square}), it follows that for any $\epsilon>0$
$$
\left|g_n(s)-g_n(s+\epsilon)\right|\leq 4\|Q_ne^{-iTs}(e^{-iT\epsilon}-I)\chi_{U}(T)x\|\|x\|\leq 4\|x\|\|(e^{-iT\epsilon}-I)\chi_{U}(T)x\|.
$$
But $e^{-iT\epsilon}-I$ converges strongly to $0$ as $\epsilon\downarrow0$ and hence the quantity
$$
\left|g_n(j/m)-m\int_{\frac{j-1}{m}}^{\frac{j}{m}}g_n(s)ds\right|\rightarrow 0
$$
uniformly in $j$ as $m\rightarrow\infty$. It follows that
$$
\lim_{m\rightarrow\infty}\Gamma_{n,m}(T,x,U)=\lim_{t\rightarrow\infty}\frac{1}{t}\int_0^t\left\|Q_ne^{-iTs}E_U^Tx\right\|^2 ds
$$
and hence
$$
\lim_{n\rightarrow\infty}\lim_{m\rightarrow\infty}\Gamma_{n,m}(T,x,U)=\langle P_{\mathrm{c}}^TE_U^Tx,x\rangle.
$$

\textbf{Step 2}: Next we deal with the case $\mathcal{I}=\mathrm{s}$. Note that for $z\in\mathbb{C}\backslash\mathbb{R}$, $\langle R(z,T)x,x \rangle$ is simply the Stieltjes transform (also called the Borel transform) of the positive measure $\mu_{x,x}^T$
$$
\langle R(z,T)x,x \rangle=\int_{\mathbb{R}}\frac{1}{\lambda -z}d\mu_{x,x}^T(\lambda ).
$$
The Hilbert transform of $\mu^T_{x,x}$ is given by the limit
$$
H_{\mu_{x,x}^T}(t)=\frac{1}{\pi}\lim_{\epsilon\downarrow0}\mathrm{Re}\left(\langle R(t+i\epsilon,T)x,x \rangle\right),
$$
with the limit existing (Lebesgue) almost everywhere. This object was studied in \cite{poltoratski2010hilbert,poltoratski1996distributions}, where we shall use the result (since the measure is positive) that for any bounded continuous function $f$,\footnote{Note that this is stronger than weak$^*$ convergence which in this case means restricting to continuous functions vanishing at infinity. That the result holds for arbitrary bounded continuous functions is due to the tightness condition that the result holds for the function identically equal to $1$.}
\begin{equation}
\label{Hilbert_limit}
\lim_{\theta\rightarrow\infty}\frac{\pi \theta}{2}\int_{\mathbb{R}}f(t)\chi_{\{w:|H_{\mu_{x,x}^T}(w)|\geq \theta\}}(t)dt=\int_{\mathbb{R}}f(t)d\mu_{x,x,\mathrm{s}}^T(t).
\end{equation}

Now let $(T,x,U)\in\Omega_{f,\alpha,\beta}\times\mathcal{U}$ with
$$
U=\bigcup_{m}(a_m,b_m),
$$
where $a_m,b_m\in\mathbb{R}\cup\{\pm\infty\}$ and the disjoint union is at most countable as in (\ref{open_union}). Without loss of generality, we assume that the union is over $m\in\mathbb{N}$. Due to the possibility of point spectra at the endpoints $a_m,b_m$, we cannot simply replace $f$ by $\chi_{U}$ in the above limit (\ref{Hilbert_limit}). However, this can be overcome in the following manner. 

Let $\partial U$ denote the boundary of $U$ defined by $\overline{U}\backslash U$ and let $\nu$ denote the measure $\mu_{x,x}^T|_{\partial U}$. Let $\{f_{l}\}_{l\in\mathbb{N}}$ denote a pointwise increasing sequence of continuous functions, converging everywhere up to $\chi_{U}$, such that the support of each $f_l$ is contained in
$$
[-l,l]\bigcap\left(\bigcup_{m=1}^{l}\left(a_m+1/\sqrt{l},b_m-1/\sqrt{l}\right)\right).
$$
Such a sequence exists (and can easily be explicitly constructed) precisely because $U$ is open. We first claim that
\begin{equation}
\label{Hilbert_limit2}
\lim_{l\rightarrow\infty}\frac{\pi l}{2}\int_{\mathbb{R}}f_l(t)\chi_{\{w:|H_{\mu_{x,x}^T}(w)|\geq l\}}(t)dt=\mu_{x,x,\mathrm{s}}^T(U).
\end{equation}
To see this note that for any $k\in\mathbb{N}$, the following inequalities hold
\begin{equation*}
\begin{split}
\liminf_{l\rightarrow\infty}\frac{\pi l}{2}\int_{\mathbb{R}}f_l(t)\chi_{\{w:|H_{\mu_{x,x}^T}(w)|\geq l\}}(t)dt&\geq \liminf_{l\rightarrow\infty}\frac{\pi l}{2}\int_{\mathbb{R}}f_k(t)\chi_{\{w:|H_{\mu_{x,x}^T}(w)|\geq l\}}(t)dt\\
&=\int_{\mathbb{R}}f_k(t)d\mu_{x,x,\mathrm{s}}^T(t),
\end{split}
\end{equation*}
where the last equality is due to (\ref{Hilbert_limit}). Taking $k\rightarrow\infty$ yields
\begin{equation}
\label{Hilbert_limit3}
\liminf_{l\rightarrow\infty}\frac{\pi l}{2}\int_{\mathbb{R}}f_l(t)\chi_{\{w:|H_{\mu_{x,x}^T}(w)|\geq l\}}(t)dt\geq\mu_{x,x,\mathrm{s}}^T(U),
\end{equation}
so we are left with proving a similar bound for the limit supremum. Note that any point in the support of $f_l$ is of distance at least $1/\sqrt{l}$ from $\partial U$. It follows that there exists a constant $C$ independent of $t$ such that for any $t\in \mathrm{supp}(f_l)$,
$$
\left|H_{\nu}(t)\right|\leq C\sqrt{l}
$$
Now let $\epsilon\in(0,1)$. Then, for large $l$, $l-C\sqrt{l}\geq (1-\epsilon)l$ and hence
\begin{equation}
\label{set_inclu}
\mathrm{supp}(f_l)\cap\{w:|H_{\mu_{x,x}^T}(w)|\geq l\}\subset \mathrm{supp}(f_l)\cap\{w:|H_{\mu_{x,x}^T-\nu}(w)|\geq (1-\epsilon)l\}.
\end{equation}
Now let $f$ be any bounded continuous function such that $f\geq \chi_U$. Then using (\ref{set_inclu}),
\begin{equation*}
\begin{split}
\limsup_{l\rightarrow\infty}\frac{\pi l}{2}\int_{\mathbb{R}}f_l(t)\chi_{\{w:|H_{\mu_{x,x}^T}(w)|\geq l\}}(t)dt&\leq \limsup_{l\rightarrow\infty}\frac{1}{1-\epsilon}\frac{\pi (1-\epsilon)l}{2}\int_{\mathbb{R}}f_l(t)\chi_{\{w:|H_{\mu_{x,x}^T-\nu}(w)|\geq (1-\epsilon)l\}}(t)dt\\
&\leq \limsup_{l\rightarrow\infty}\frac{1}{1-\epsilon}\frac{\pi (1-\epsilon)l}{2}\int_{\mathbb{R}}f(t)\chi_{\{w:|H_{\mu_{x,x}^T-\nu}(w)|\geq (1-\epsilon)l\}}(t)dt\\
&=\frac{1}{1-\epsilon}\int_{\mathbb{R}}f(t)d([\mu_{x,x}^T-\nu]_{\mathrm{s}})(t).
\end{split}
\end{equation*}
Now we let $f\downarrow\chi_{\overline{U}}$, with pointwise convergence everywhere. This is possible since the complement of $\overline{U}$ is open. By the dominated convergence theorem, and since $\epsilon$ was arbitrary, this yields
$$
\limsup_{l\rightarrow\infty}\frac{\pi l}{2}\int_{\mathbb{R}}f_l(t)\chi_{\{w:|H_{\mu_{x,x}^T}(w)|\geq l\}}(t)dt\leq[\mu_{x,x}^T-\nu]_{\mathrm{s}}(\overline{U})=\mu_{x,x,{\mathrm{s}}}^T(U),
$$
where the last equality follows from the definition of $\nu$. The claim (\ref{Hilbert_limit2}) now follows.

Let $\chi_n$ be a sequence of non-negative continuous piecewise affine functions on $\mathbb{R}$, bounded by $1$ and such that $\chi_n(t)=0$ if $t\leq n-1$ and $\chi_n(t)=1$ if $t \geq n+1$. Consider the integrals
$$
I(n,m)=\frac{\pi n}{2}\int_{\mathbb{R}}f_n(t)\chi_n(\left|F_m(t)\right|)dt,
$$
where $F_m(t)$ is an approximation of
$$
\frac{1}{\pi}\mathrm{Re}\left(\left\langle R\left(t+\frac{i}{m},T\right)x,x \right\rangle\right)
$$
to pointwise accuracy $O(m^{-1})$ over $t\in[-n,n]$. Note that a suitable piecewise affine function $f_n$ can be constructed using $\Lambda_1$, as can suitable $\chi_n$, and a suitable approximation function $F_m$ can be pointwise evaluated using $\Lambda_1$ (again by Corollary \ref{res_est2}). To see this, recall the definition of $\Lambda_1$ in (\ref{evals_def_nene}) and that we added $\{a_m,b_m\}$ from (\ref{open_union}) to $\Lambda_1$. To define $f_n$, we can define the function by suitable piecewise affine functions on each interval $[-n,n]\cap\left(a_m+1/\sqrt{n},b_m-1/\sqrt{n}\right)$. It follows that there exists arithmetic algorithms $\Gamma_{n,m}(T,x,U)$ using $\Lambda_1$ such that
$$
\left|I(n,m)-\Gamma_{n,m}(T,x,U)\right|\leq\frac{C(T,x,U)}{m}.
$$
The dominated convergence theorem implies that
$$
\lim_{m\rightarrow\infty}\Gamma_{n,m}(T,x,U)=\lim_{m\rightarrow\infty}I(n,m)=\frac{\pi n}{2}\int_{\mathbb{R}}f_n(t)\chi_n(|H_{\mu_{x,x}^T}(t)|)dt.
$$
Note that continuity of $\chi_n$ is needed to gain convergence almost everywhere and prevent possible oscillations about the level set $\{H_{\mu_{x,x}^T}(t)=n\}$. We also have
$$
\chi_{\{w:|H_{\mu_{x,x}^T}(w)|\geq n+1\}}(t)\leq\chi_n(|H_{\mu_{x,x}^T}(t)|)\leq\chi_{\{w:|H_{\mu_{x,x}^T}(w)|\geq n-1\}}(t)
$$
The same arguments used to prove (\ref{Hilbert_limit2}), therefore show that
$$
\lim_{n\rightarrow\infty}\frac{\pi n}{2}\int_{\mathbb{R}}f_n(t)\chi_n(|H_{\mu_{x,x}^T}(t)|)dt=\mu_{x,x,\mathrm{s}}^T(U).
$$
Hence,
$$
\lim_{n\rightarrow\infty}\lim_{m\rightarrow\infty}\Gamma_{n,m}(T,x,U)=\mu_{x,x,\mathrm{s}}^T(U),
$$
completing the proof of inclusion in Theorem \ref{spec_decomp_comp}.
\end{proof}

\subsubsection{Proof of exclusion in Theorem \ref{spec_decomp_comp}}
\label{measu_deomps2}

To prove the exclusion, we need two results which will also be used in \S \ref{heavy_spec}. First, we consider a result connected to Anderson localisation (Theorem \ref{andlocal}) and, second, we consider a result concerning sparse potentials of discrete Schr\"odinger operators (Theorem \ref{oracle_lemma}). The free Hamiltonian $H_0$ acts on $l^2(\mathbb{N})$ via the tridiagonal matrix representation
$$
H_0=\begin{pmatrix}
2 & -1 &  & \\
-1 & 2 & -1 & \\
 & -1 & 2 & \ddots\\
 &  & \ddots & \ddots
\end{pmatrix}
$$
We define a Schr{\"o}dinger operator acting on $l^2(\mathbb{N})$ to be an operator of the form
\begin{equation*}
H_v=H_0+v,
\end{equation*}
where $v$ is a bounded (real-valued) multiplication operator with matrix $\mathrm{diag}(v(1),v(2),...)$.

Since Anderson's introduction of his famous model 60 years ago \cite{anderson1958absence}, there has been a considerable amount of work by both physicists and mathematicians aiming to understand the suppression of electron transport due to disorder (Anderson localisation). A full discussion of Anderson localisation is beyond the scope of this paper, and we refer the reader to \cite{carmona2012spectral,cycon2009schrodinger,kirsch2007invitation} for broader surveys. When considering Anderson localisation, we will assume that $v=v_{\omega}=\{v(1),v(2),...\}$ is a collection of independent identically distributed random variables. Following \cite{graf1994anderson}, we assume that the single-site probability distribution has a density $\rho\in L^1(\mathbb{R})\cap L^\infty(\mathbb{R})$ with $\|\rho\|_1=1$ (with respect to the standard Lebesgue measure). For such a potential, a measure of disorder is given by the quantity $\|\rho\|_{\infty}^{-1}$. The first result we need is the following theorem which follows straightforwardly from the technique of \cite{graf1994anderson}, and hence we have not provided a proof which would be almost verbatim to \cite{graf1994anderson}.

\begin{theorem}[Anderson Localisation for Perturbed Operator \cite{graf1994anderson}]
\label{andlocal}
There exists a constant $C>0$ such that if $\|\rho\|_{\infty}\leq C$ and $\rho$ has compact support, then the operator $H_{v_\omega}+A$ has only pure point spectrum with probability 1 for any fixed self-adjoint operator $A$ of the form
\begin{equation}
\label{A_perturbation}
A=\sum_{j=1}^M\alpha_j\left|x_{m_j}\right\rangle\left\langle x_{n_j}\right|.
\end{equation}
In other words, the operator $A$'s matrix with respect to the canonical basis has only finitely many non-zeros.
\end{theorem}

The second result we need is the following from \cite{krutikov2001schrodinger}.
\begin{theorem}[Krutikov and Remling \cite{krutikov2001schrodinger}]
\label{oracle_lemma}
Consider discrete Schr\"odinger operators acting on $l^2(\mathbb{N})$. Let $v$ be a (real-valued and bounded) potential of the following form:
$$
v(n)=\sum_{j=1}^\infty g_j\delta_{n,m_j},\quad m_{j-1}/m_{j}\rightarrow0.
$$
Then $[0,4]\subset \sigma_{\mathrm{ess}}(H_v)$ and the following dichotomy holds:
\begin{itemize}
	\item[(a)] If $\sum_{j\in\mathbb{N}}g_j^2<\infty$ then $H_v$ is purely absolutely continuous on $(0,4)$.
	\item[(b)] If $\sum_{j\in\mathbb{N}}g_j^2=\infty$ then $H_v$ is purely singular continuous on $(0,4)$.
\end{itemize}
\end{theorem}

\begin{proof}[Proof of exclusion in Theorem \ref{spec_decomp_comp}]
Since $P^T_{\mathrm{pp}}=I-P^T_{\mathrm{c}}$, $P^T_{\mathrm{ac}}=I-P^T_{\mathrm{s}}$ and $P^T_{\mathrm{sc}}=P^T_{\mathrm{s}}-P^T_{\mathrm{pp}}$, it is enough, by Theorem \ref{meas_comp1} and Remark \ref{TKIP}, to consider $\mathcal{I}=\mathrm{pp}$, $\mathrm{ac}$ and $\mathrm{sc}$. We restrict the proof to considering bounded Schr\"odinger operators $H_v$ acting on $l^2(\mathbb{N})$, which are clearly a subclass of $\Omega_{f,0}$ for $f(n)=n+1$. Note that since the evaluation functions in $\Lambda_2$ can be recovered from those in $\Lambda_1$ in this special case, we can assume that we are dealing with $\Lambda_1$. We also set $x=e_1$, with the crucial properties that this vector is cyclic and hence $\mu^{H_v}_{e_1,e_1}$ has the same support as $\sigma(H_v)$, and that $x\in {V}_{0}$. Throughout, we also take $U=(0,4)$.

\textbf{Step 1}: We begin with $P^T_{\mathrm{pp}}$. Suppose for a contradiction that there does exist a sequence of general algorithms $\Gamma_n$ such that
$$
\lim_{n\rightarrow\infty}\Gamma_n(H_v)=\langle P^{H_v}_{\mathrm{pp}}E_{(0,4)}^{H_v}e_1,e_1\rangle.
$$
We take a general algorithm, denoted $\widehat\Gamma_n$, from Theorem \ref{meas_comp1} which has
$$
\lim_{n\rightarrow\infty}\widehat\Gamma_n(H_v)=\mu_{e_1,e_1}^{H_v}((0,4)).
$$
Since $e_1$ is cyclic, this limit is non-zero if $(0,4)\cap\sigma(H_v)\neq\emptyset$. We therefore define
$$
\widetilde\Gamma_n(H_v)=\begin{dcases}
0\quad\text{if }\widehat\Gamma_n(H_v)=0\\
\frac{\Gamma_n(H_v)}{\widehat\Gamma_n(H_v)}\quad\text{otherwise}.
\end{dcases}
$$

We will use Theorem \ref{andlocal} and the following well-known facts:
\begin{enumerate}
	\item If for any $l\in\mathbb{N}$ there exists $m_l$ such that $v({m_l+1})=v({m_l+2})=...=v({m_l+l})=0$, then $(0,4)\subset \sigma(H_v)$.
	\item If there exists $N\in\mathbb{N}$ such that $v(n)$ is $0$ for $n\geq N$, then $\sigma_{\mathrm{pp}}(H_v)\cap(0,4)=\emptyset$ \cite{remling1998absolutely}, but $[0,4]\subset \sigma(H_v)$ (the potential acts as a compact perturbation so the essential spectrum is $[0,4]$).
	\item If we are in the setting of Theorem \ref{andlocal}, then the spectrum of $H_{v_\omega}+A$ is pure point almost surely. Moreover, if $\rho=\chi_{[-c,c]}/(2c)$ for some constant $c$, then $[-c,4+c]\subset\sigma_{\mathrm{pp}}(H_{v_\omega}+A)$ almost surely.
\end{enumerate}

The strategy will be to construct a potential $v$ such that $(0,4)\subset \sigma(H_v)$, yet $\widetilde\Gamma_n(H_v)$ does not converge. This is a contradiction since by our assumptions, for such a $v$ we must have
$$
\widetilde\Gamma_n(H_v)\rightarrow \frac{\langle P^{H_v}_{\mathrm{pp}}E_{(0,4)}^{H_v}e_1,e_1\rangle}{\mu_{e_1,e_1}^{H_v}((0,4))}.
$$
To do this, choose $\rho=\chi_{[-c,c]}/(2c)$ for some constant $c$ such that the conditions of Theorem \ref{andlocal} hold and define the potential $v$ inductively as follows.

Let $v_1$ be a potential of the form $v_{\omega}$ (with the density $\rho$) such that $\sigma(H_{v_1})$ is pure point. Such a $v_1$ exists by Theorem \ref{andlocal} and we have $\langle P^{H_{v_1}}_{\mathrm{pp}}E_{(0,4)}^{H_{v_1}}e_1,e_1\rangle=\mu_{e_1,e_1}^{H_{v_1}}((0,4))$. Hence for large enough $n$ it must hold that $\widetilde\Gamma_n(H_{v_1})>3/4$. Fix $n=n_1$ such that this holds. Then $\Gamma_{n_1}(H_{v_1})$ only depends on $\{v_1(j):j\leq N_1\}$ for some integer $N_1$ by (i) of Definition \ref{Gen_alg}. Define the potential $v_2$ by $v_2(j)=v_1(j)$ for all $j\leq N_1$ and $v_2(j)=0$ otherwise. Then by fact (2) above, $\langle P^{H_{v_2}}_{\mathrm{pp}}E_{(0,4)}^{H_{v_2}}e_1,e_1\rangle=0$ but $\mu_{e_1,e_1}^{H_{v_2}}((0,4))\neq0$, and hence $\widetilde\Gamma_n(H_{v_2})<1/4$ for large $n$, say for $n=n_2>n_1$. But then $\Gamma_{n_2}(H_{v_2})$ only depends on $\{v_2(j):j\leq N_2\}$ for some integer $N_2$. 

We repeat this process inductively switching between potentials which induce $\widetilde\Gamma_{n_k}(H_{v_{k}})<1/4$ for $k$ even and potentials which induce $\widetilde\Gamma_{n_k}(H_{v_{k}})>3/4$ for $k$ odd. Explicitly, if $k$ is even then define a potential $v_{k+1}$ by $v_{k+1}(j)=v_k(j)$ for all $j\leq N_k$ and $v_{k+1}(j)=v_{\omega}(j)$ (with the density $\rho$) otherwise such that the spectrum of $H_{v_{k}}$ is pure point. Such a $\omega$ exists from Theorem \ref{andlocal} applied with the perturbation $A$ to match the potential for $j\leq N_k$. If $k$ is odd then we define $v_{k+1}$ by $v_{k+1}(j)=v_k(j)$ for all $j\leq N_k$ and $v_{k+1}(j)=0$ otherwise. We can then choose $n_{k+1}$ such that the above inequalities hold and $N_{k+1}$ such that $\Gamma_{n_{k+1}}(H_{v_{k+1}})$ only depends on $\{v_{k+1}(j):j\leq N_{k+1}\}$. We also ensure that $N_{k+1}\geq N_k+k$.

Finally set $v(j)=v_{k}(j)$ for $j\leq N_{k}$. It is clear from (iii) of Definition \ref{Gen_alg}, that $\widetilde\Gamma_{n_k}(H_v)=\widetilde\Gamma_{n_k}(H_{v_{k}})$ and this implies that $\widetilde\Gamma_{n_k}(H_v)$ cannot converge. However, since $N_{k+1}\geq N_k+k$, for any $k$ odd we have $v({N_k+1})=v({N_k+2})=...=v({N_k+k})=0$. Fact (1) implies that $(0,4)\subset \sigma(H_v)$, hence $\mu_{e_1,e_1}^{H_v}((0,4))\neq 0$ and therefore $\widetilde\Gamma_n(H_v)$ converges. This provides the required contradiction.

\textbf{Step 2}: Next we deal with $\mathcal{I}=\mathrm{ac}$. To prove that one limit will not suffice, our strategy will be to reduce a certain decision problem to the computation of $\Xi_{\mathrm{ac}}$. Let $(\mathcal{M}',d')$ be the discrete space $\{0,1\}$, let $\Omega'$ denote the collection of all infinite sequence $\{a_{j}\}_{j\in\mathbb{N}}$ with entries $a_{j}\in\{0,1\}$ and consider the problem function
\begin{equation*}
\Xi'(\{a_{j}\}):\text{ `Does }\{a_{j}\}\text{ have infinitely many non-zero entries?'},
\end{equation*}
which maps to $(\mathcal{M}',d')$. In Appendix \ref{append1}, it is shown that $\mathrm{SCI}(\Xi',\Omega')_{G} = 2$ (where the evaluation functions consist of component-wise evaluation of the array $\{a_{j}\}$). Suppose for a contradiction that $\Gamma_{n}$ is a height one tower of general algorithms such that
$$
\lim_{n\rightarrow\infty}\Gamma_n(H_v)=\langle P^{H_v}_{\mathrm{ac}}E_{(0,4)}^{H_v}e_1,e_1\rangle.
$$
We will gain a contradiction by using the supposed tower to solve $\{\Xi',\Omega'\}$. 

Given $\{a_{j}\}\in\Omega'$, consider the operator $H_v$, where the potential is of the following form:
\begin{equation}
\label{factorial_pot}
v(m)=\sum_{k=1}^\infty a_{k}\delta_{m,k!}.
\end{equation}
Then by Theorem \ref{oracle_lemma}, $\langle P^{H_v}_{\mathrm{ac}}E_{(0,4)}^{H_v}e_1,e_1\rangle=\mu^{H_v}_{e_1,e_1}((0,4))$ if $\sum_{k}a_{k}<\infty$ (that is, if $\Xi'(\{a_j\})=0$) and $\langle P^{H_v}_{\mathrm{ac}}E_{(0,4)}^{H_v}e_1,e_1\rangle=0$ otherwise. Note that in either case we have $\mu^{H_v}_{e_1,e_1}((0,4))\neq 0$. We follow Step 1 and take a general algorithm, denoted $\widehat\Gamma_n$, from Theorem \ref{meas_comp1} which has
$$
\lim_{n\rightarrow\infty}\widehat\Gamma_n(H_v)=\mu_{e_1,e_1}^{H_v}((0,4)).
$$
Since $e_1$ is cyclic, this limit is non-zero for $H_v$, where $v$ is of the form (\ref{factorial_pot}). We therefore define
$$
\widetilde\Gamma_n(H_v)=\begin{dcases}
0\quad\text{if }\widehat\Gamma_n(H_v)=0\\
\frac{\Gamma_n(H_v)}{\widehat\Gamma_n(H_v)}\quad\text{otherwise}.
\end{dcases}
$$
It follows that
$$
\lim_{n\rightarrow\infty}\widetilde\Gamma_n(H_v)=\begin{dcases}
1\quad\text{if }\Xi'(\{a_j\})=0\\
0\quad\text{otherwise}.
\end{dcases}
$$

Given $N$, we can evaluate any matrix value of $H$ using only finitely many evaluations of $\{a_{j}\}$ and hence the evaluation functions $\Lambda_1$ can be computed using component-wise evaluations of the sequence $\{a_j\}$. We now set
$$
\overline{\Gamma}_{n}(\{a_{j}\})=\begin{dcases}
0\quad\text{if }\Gamma_n(H_v)>\frac{1}{2}\\
1\quad\text{otherwise}.
\end{dcases}
$$
The above comments show that each of these is a general algorithm and it is clear that it converges to $\Xi'(\{a_j\})$ as $n\rightarrow\infty$, the required contradiction.

\textbf{Step 3}: Finally, we must deal with $\mathcal{I}=\mathrm{sc}$. The argument is the same as Step 2, replacing $\langle P^{H_v}_{\mathrm{ac}}E_{(0,4)}^{H_v}e_1,e_1\rangle$ with $\langle P^{H_v}_{\mathrm{sc}}E_{(0,4)}^{H_v}e_1,e_1\rangle$ and the resulting $\widetilde\Gamma_n(H_v)$ with $1-\widetilde\Gamma_n(H_v)$.
\end{proof}

\section{Two Important Applications}
\label{applic_sec}

Two important applications of our techniques are the computation of the functional calculus and of the Radon-Nikodym derivative of $\mu_{x,y,\mathrm{ac}}^T$ with respect to Lebesgue measure, denoted by $\rho_{x,y}^T$. Both of these have applications throughout mathematics and the physical sciences, some of which are explored numerically in \S \ref{num_sec}. For example, suppose that we wish to solve the Schr\"odinger equation
$$
\frac{du}{dt}=-iHu,\quad u_{t=0}=u_0,
$$
where $H$ is some self-adjoint Hamiltonian. We can express the solution at time $t$ as
$$
u(t)=\exp(-iHt)u_0=\left[\int_{\mathbb{R}}\exp(-i\lambda t)dE^H(\lambda)\right]u_0.
$$
For example, the quantity
$$
L(t)=\langle u(t),u_0\rangle=\int_{\mathbb{R}}\exp(-i\lambda t)d\mu_{u_0,u_0}^H(\lambda),
$$
known as the autocorrelation function \cite{tannor2007introduction}, is simply the Fourier transform of the spectral measure $d\mu_{u_0,u_0}^H$. In particular, if $d\mu_{u_0,u_0}^H$ is absolutely continuous, then $\rho_{u_0,u_0}^H$ and $L$ form a Fourier transform pair. The computation of evolution generated by an operator is in some sense dual to the computation of the spectral measure. This interpretation of a time evolution can be adapted to describe many signals generated by PDEs~\cite{simon1982schrodinger,hundertmark2013operator,dell2019second} and stochastic processes~\cite{kallianpur1971spectral,girardin2003semigroup}~\cite[Ch.~7]{rosenblatt1991stochastic}. In this section, we show how to compute the functional calculus and $\rho_{x,y}^T$.

\subsection{Computation of the functional calculus}
\label{applic_FCCC}

Recall that given a possibly unbounded complex-valued Borel function $F$, defined on $\sigma(T)$, and $T\in\Omega_{\mathrm{N}}$, $F(T)$ is defined by
$$
F(T)=\int_{{\sigma(T)}}F(\lambda )dE^T(\lambda ).
$$
$F(T)$ is a densely defined closed normal operator with dense domain given by
$$
\mathcal{D}(F(T))=\left\{x\in l^2(\mathbb{N}):\int_{\sigma(T)}\left|F(\lambda )\right|^2d\mu_{x,x}^T(\lambda )<\infty\right\}.
$$
For simplicity, we will only deal with the case that $F$ is a bounded continuous function on $\mathbb{R}$, that is, $F\in C_b(\mathbb{R})$. In this case $\mathcal{D}(F(T))$ is the whole of $l^2(\mathbb{N})$ (the variations $|\mu^T_{x,y}|$ are finite) and we can use standard properties of the Poisson kernel. We assume that given $F\in C_b(\mathbb{R})$, we have access to piecewise constant functions $F_n$ supported in $[-n,n]$ such that $\|F-F_n\|_{L^\infty([-n,n])}\leq n^{-1}$. Clearly other suitable data also suffices and, as usual, we abuse notation slightly by adding this information to the evaluation sets $\Lambda_i$ (recall that $\Lambda_i$ are defined in (\ref{evals_def_nene})).

\begin{theorem}[Computation of the functional calculus]
\label{F_calc}
Given the set-up in \S \ref{sec_not} and \S \ref{sec_not2}, consider the map
\begin{align*}
\Xi_{\mathrm{fun}}:\Omega_{f,\alpha,\beta}\times C_b(\mathbb{R})&\rightarrow l^2(\mathbb{N})\\
(T,x,F)&\rightarrow F(T)x.
\end{align*}
Then $\{\Xi_{\mathrm{fun}},\Omega_{f,\alpha,\beta}\times C_b(\mathbb{R}), \Lambda_1\}\in\Delta_2^A$.
\end{theorem}

\begin{proof}
Let $(T,x,F)\in\Omega_{f,\alpha,\beta}\times C_b(\mathbb{R})$ then by Fubini's theorem,
$$
\int_{-n}^n K_H(u+i/n;T,x)F_n(u)du=\left[\int_{-\infty}^\infty\int_{-n}^nP_H(u-\lambda ,1/n)F_n(u)du\ dE^T(\lambda )\right]x.
$$
The inner integral is bounded since $F$ is bounded and the Poisson kernel integrates to $1$ along the real line. It also converges to $F(\lambda )$ everywhere. Hence by the dominated convergence theorem
$$
\lim_{n\rightarrow\infty}\int_{-n}^n K_H(u+i/n;T,x)F_n(u)du=F(T)x.
$$
We now use the same arguments used to prove Theorem \ref{meas_comp1}. Using Corollary \ref{res_est2}, together with $\|K_H(u+i/n;T,x)\|_{L^\infty(\mathbb{R})}\leq nC_1$ and the fact that $K_H(u+i/n;T,x)$ is Lipschitz continuous with Lipschitz constant $n^2C_2$ for some (possibly unknown) constants $C_1$ and $C_2$, we can approximate this integral with an error that vanishes in the limit $n\rightarrow\infty$. 
\end{proof}

If $\sigma(T)$ is bounded, then, with slightly more information available to our algorithms, a simpler proof holds using the Stone--Weierstrass theorem. Suppose that given $x$, the vectors $T^nx$ can be computed to arbitrary precision. There exists a sequence of polynomials $p_m(z)$ converging uniformly to $F(z)$ on $\sigma(T)$. Assuming such a sequence can be explicitly constructed (for example using Bernstein or Chebyshev polynomials), we can take $p_m(T)x$ as approximations of $F(T)x$. If we can bound $\|p_m(z)-F(z)\|_{\sigma(T)}\leq \epsilon_m$ with $\epsilon_m$ null, then the vector $F(T)x$ can be computed \textit{with error control}. However, computing $T^nx$ for large $n$ (even if $x=e_1$) may be computationally expensive as was found in the example in \S \ref{penrose_numerics}. We will also see in \S \ref{penrose_numerics} that if $\sigma(T)$ is bounded and $F$ is analytic in an open neighbourhood of $\sigma(T)$, then $F(T)x$ can be computed with error control by deforming the integration contour away from the spectrum. Such a deformation is useful since we the resolvent does not blow up along such a contour, and we can bound its Lipschitz constant.

\subsection{Computation of the Radon--Nikodym derivative}
\label{app1_RN}

Recall the definition of the Radon--Nikodym derivative in (\ref{RN_def}) and note that $\rho_{x,y}^T\in L^1(\mathbb{R})$ for $T\in\Omega_{\mathrm{SA}}$. We consider its computation in the $L^1$ sense in the following theorem, where, as before, we assume (\ref{open_union}), adding the approximations of $U$ to our evaluation set $\Lambda_1$ (defined in (\ref{evals_def_nene})), along with component-wise evaluations of a given vector $y$. However, we must consider the computation away from the singular part of the spectrum.

\begin{theorem}[Computation of the Radon--Nikodym derivative]
\label{meas_comp3}
Given the set-up in \S \ref{sec_not} and \S \ref{sec_not2}, consider the map
\begin{align*}
\Xi_{\mathrm{RN}}:\Omega_{f,\alpha,\beta}\times l^2(\mathbb{N})\times \mathcal{U}&\rightarrow L^1(\mathbb{R})\\
(T,x,y,U)&\rightarrow \rho_{x,y}^T|_{U}.
\end{align*}
We restrict this map to the quadruples $(T,x,y,U)$ such that $U$ is strictly separated from $\mathrm{supp}(\mu_{x,y,\mathrm{sc}}^T)\cup\mathrm{supp}(\mu_{x,y,\mathrm{pp}}^T)$ and denote this subclass by $\widetilde \Omega_{f,\alpha,\beta}$. Then $\{\Xi_{\mathrm{RN}},\widetilde \Omega_{f,\alpha,\beta}, \Lambda_1\}\in\Delta_2^A$. Furthermore, each output $\Gamma_n(T,x,y,U)$ of the algorithms constructed in the proof consists of a piecewise affine function, supported in $U$ with rational knots and taking (complex) rational values at these knots.
\end{theorem}

\begin{remark}
Essentially, this theorem tells us that if we can compute the action of the resolvent operator with asymptotic error control, then we can compute the Radon--Nikodym derivative of the absolutely continuous part of the measures on open sets which are a positive distance away from the singular support of the measure. The assumption that $U$ is separated from $\mathrm{supp}(\mu_{x,y,\mathrm{sc}}^T)\cup\mathrm{supp}(\mu_{x,y,\mathrm{pp}}^T)$ may seem unnatural but is needed to gain $L^1$ convergence of the approximation. However, without it, the proof still gives almost everywhere pointwise convergence.
\end{remark}

\begin{proof}
Let $(T,x,y,U)\in\widetilde \Omega_{f,\alpha,\beta}$. For $u\in U$ we decompose as follows
\begin{equation}
\label{RN_deriv}
\langle K_H(u+i\epsilon;T,x),y \rangle=\frac{1}{\pi}\int_{\mathbb{R}}\frac{\epsilon}{(\lambda -u)^2+\epsilon^2}\rho_{x,y}^T(\lambda )d\lambda +\frac{1}{\pi}\int_{\mathbb{R}\backslash U}\frac{\epsilon}{(\lambda -u)^2+\epsilon^2}\left\{d\mu_{{x,y},\mathrm{sc}}^T(\lambda )+d\mu_{{x,y},\mathrm{pp}}^T(\lambda )\right\}.
\end{equation}
The first term converges to $\rho_{x,y}^T|_{U}$ in $L^1(U)$ as $\epsilon\downarrow 0$ since $\rho_{x,y}^T|_U\in L^1(U)$. Since we assumed that $U$ is separated from $\mathrm{supp}(\mu_{{x,y},\mathrm{sc}}^T)\cup\mathrm{supp}(\mu_{{x,y},\mathrm{pp}}^T)$, it follows that the second term of (\ref{RN_deriv}) converges to $0$ in $L^1(U)$ as $\epsilon\downarrow 0$. Hence we are done if we can approximate $\langle K_H(u+i/n;T,x),y \rangle$ in $L^1(U)$ with an error converging to zero as $n\rightarrow\infty$.

Recall that $K_H(u+i/n;T,x)$ is Lipschitz continuous with Lipschitz constant at most $n^2\|x\|/\pi$. By assumption, and using Corollary \ref{res_est2}, we can approximate $K_H(u+i/n;T,x)$ to asymptotic precision with vectors of finite support. Hence the inner product
$$f_n(u):=\langle K_H(u+i/n;T,x),y \rangle$$
can be approximated to asymptotic precision (now with a possibly unknown constant also depending on $\|y\|$) and $f_n$ is Lipschitz continuous with Lipshitz constant at most $n^2\|x\|\|y\|/\pi$.

Recall that $U$ can be written as the disjoint union
$$
U=\bigcup_{m}(a_m,b_m),
$$
where $a_m,b_m\in\mathbb{R}\cup\{\pm\infty\}$ and the union is at most countable. Without loss of generality, we assume that the union is over $m\in\mathbb{N}$. Given an interval $(a_m,b_m)$, let $a_m<z_{m,1,n}< z_{m,2,n}<...<z_{m,r_m,n}<b_m$ such that $z_{m,j,n}\in\mathbb{Q}$ and $\left|z_{m,j,n}-z_{m,j+1,n}\right|\leq (b_m-a_m)^{-1}n^{-3}m^{-2}$ and $\left|a_m-z_{m,1,n}\right|,\left|b_m-z_{m,r_m,n}\right|\leq n^{-1}$. Let $f_{m,n}$ be a piecewise affine interpolant with knots $z_{m,1,n},...,z_{m,r_m,n}$ supported on $(z_{m,1,n},z_{m,r_m,n})$ with the property that $\left|f_{m,n}(z_{m,j,n})-f_n(z_{m,j,n})\right|<C(b_m-a_m)^{-1}n^{-1}m^{-2}$. Here $C$ is some unknown constant which occurs from the asymptotic approximation of $f_n$ that arises from Corollary \ref{res_est2} and we can always compute such $f_{m,n}$ in finitely many arithmetic operations and comparisons.

Let $\Gamma_n(T,x,y,U)$ be the function that agrees with $f_{m,n}$ on $(a_m,b_m)$ for $m\leq n$ and is zero elsewhere. Clearly the nodes of $\Gamma_n(T,x,y,U)$ can be computed using finitely many arithmetic operations and comparisons and the relevant set of evaluation functions $\Lambda_1$. A simple application of the triangle inequality implies that
\begin{equation*}
\begin{split}
\int_{U}\left|\Gamma_n(T,U,x,y)(u)-\rho_{x,y}^T(u)\right|du&\leq \sum_{m>n}\int_{(a_m,b_m)}\left|\rho_{x,y}^T(u)\right|du+\sum_{m\leq n}\int_{(a_m,b_m)\backslash (z_{m,1,n},z_{m,r_m,n})}\left|\rho_{x,y}^T(u)\right|du\\
&+\sum_{m\leq n}\int_{(z_{m,1,n},z_{m,r_m,n})}\left|\rho_{x,y}^T(u)-f_n(u)\right|du+\frac{\widetilde C(x,y,T)}{n}\sum_{m\leq n}\frac{1}{m^2},
\end{split}
\end{equation*}
where the last term arises due to the piecewise linear interpolant. The bound clearly converges to zero as required.
\end{proof}

\section{Computing Spectra as Sets}
\label{heavy_spec}

We now turn to computing the different types of spectra as sets in the complex plane. Specifically, define the problem functions $\Xi_{\mathcal{I}}^{\mathbb{C}}(T)=\sigma_{\mathcal{I}}(T)$ for $\mathcal{I}=\mathrm{ac},\mathrm{sc}$ or $\mathrm{pp}$. Note also that $\sigma_{\mathrm{pp}}(T)=\overline{\sigma_{\mathrm{p}}(T)}$, the closure of the set of eigenvalues. Recalling the definition of a computational problem in Appendix \ref{append1}, we compute these quantities in a metric space $\mathcal{M}$ with metric $d_{\mathcal{M}}$. Since we wish to include unbounded operators, we use the \emph{Attouch--Wets metric} defined by
\begin{equation}\label{eq:Attouch-Wets}
d_{\mathrm{AW}}(C_1,C_2)=\sum_{n=1}^{\infty} 2^{-n}\min\left\{{1,\underset{\left|x\right|\leq n}{\sup}\left|\mathrm{dist}(x,C_1)-\mathrm{dist}(x,C_2)\right|}\right\},
\end{equation}
for $C_1,C_2\in\mathrm{Cl}(\mathbb{C}),$ where $\mathrm{Cl}(\mathbb{C})$ denotes the set of closed non-empty subsets of $\mathbb{C}$. When considering bounded $T$, whose spectrum is necessarily a compact subset of $\mathbb{C}$, we let $(\mathcal{M},d_{\mathcal{M}})$ be the set of all non-empty compact subsets of $\mathbb{C}$ provided with the \textit{Hausdorff metric} $d_{\mathcal{M}}=d_{\mathrm{H}}$:
\begin{equation}\label{Hausdorff}
d_{\mathrm{H}}(X,Y) = \max\left\{\sup_{x \in X} \inf_{y \in Y} d(x,y), \sup_{y \in Y} \inf_{x \in X} d(x,y) \right\},
\end{equation}
where $d(x,y) = |x-y|$ is the usual Euclidean distance. Note that for compact sets (and hence for bounded operators), the topological notions of convergence according to $d_\mathrm{H}$ and $d_\mathrm{AW}$ coincide. To allow the possibility that the different spectral sets $\sigma_{\mathcal{I}}(T)$ are empty, we add the empty set to our metric space as an isolated point (the space remains metrisable).\footnote{This simply means that $F_n\rightarrow\emptyset$ if and only if $F_n=\emptyset$ eventually.}

The main theorem of this section is the following:

\begin{theorem}
\label{spec_decomp_as_sets}
Given the above setup and that in \S \ref{sec_not} and \S \ref{sec_not2}, for $i=1,2$ it holds that
$$
\Delta^G_2 \not\owns  \{\Xi_{\mathrm{ac}}^{\mathbb{C}},\Omega_{f,\alpha},\Lambda_i\} \in \Delta^A_3,\quad \Delta^G_2 \not\owns  \{\Xi_{\mathrm{sc}}^{\mathbb{C}},\Omega_{f,\alpha},\Lambda_i\} \in \Delta^A_4,\quad \Delta^G_2 \not\owns  \{\Xi_{\mathrm{pp}}^{\mathbb{C}},\Omega_{f,\alpha},\Lambda_i\} \in \Delta^A_3.
$$
If $f(n)-n\geq \sqrt{2n}+\frac{1}{2}$, then $\{\Xi_{\mathrm{sc}}^{\mathbb{C}},\Omega_{f,0},\Lambda_i\}\not\in\Delta^G_3$ also holds.
\end{theorem}

In order to prove Theorem \ref{spec_decomp_as_sets}, we only need to prove the lower bounds for $\Lambda_2$ and the upper bounds for $\Lambda_1$ (recall that $\Lambda_i$ are defined in (\ref{evals_def_nene})). These results show that despite the results of \S\S \ref{res_stones}-- \ref{applic_sec}, in general it is very hard to compute the decomposition of the spectrum in the sense of (\ref{hilbert_decomp}). We also answer the question posed in \S \ref{stone_sec} and prove that the spectral measures, while computable in one limit, cannot, in general, be computed with error control (see Theorem \ref{meas_err_impossible}), unless one has additional regularity assumptions such as in \cite{colbrook2020computingSM} (computation with error control is made precise in \cite[Ch. 4]{colbrook2020foundations}).

\subsection{Proof for point spectra}

\begin{proof}[Proof that $\{\Xi_{\mathrm{pp}}^{\mathbb{C}},\Omega_{f,\alpha},\Lambda_2\}\notin \Delta_2^G$ ] To prove this, it is enough to consider bounded Schr\"odinger operators acting on $l^2(\mathbb{N})$ (defined in \S \ref{measu_deomps2}), which are clearly a subclass of $\Omega_{f,0}$ for $f(n)=n+1$. Note that since the evaluation functions in $\Lambda_2$ can be recovered from those in $\Lambda_1$ in this special case, we can assume that we are dealing with $\Lambda_1$. Suppose for a contradiction that there does exist a sequence of general algorithms, $\Gamma_n$, with
$$
\lim_{n\rightarrow\infty}\Gamma_n(H_v)=\Xi_{\mathrm{pp}}^{\mathbb{C}}(H_v).
$$
We will construct a potential $v$ such that $\Gamma_n(H_v)$ does not converge. To do this, choose $\rho=\chi_{[-c,c]}/(2c)$ for some constant $c$ such that the conditions of Theorem \ref{andlocal} hold. We will use Theorem \ref{andlocal} and the following well-known facts:
\begin{enumerate}
	\item If there exists $N\in\mathbb{N}$ such that $v(n)$ is $0$ for $n\geq N$, then $\sigma_{\mathrm{pp}}(H_v)\cap(0,4)=\emptyset$ \cite{remling1998absolutely}, but $[0,4]\subset \sigma(H_v)$ (the potential acts as a compact perturbation so the essential spectrum is $[0,4]$).
	\item If we are in the setting of Theorem \ref{andlocal}, then the spectrum of $H_{v_\omega}+A$ is pure point almost surely. Moreover, if $\rho=\chi_{[-c,c]}/(2c)$ for some constant $c$, then $[-c,4+c]\subset\sigma_{\mathrm{pp}}(H_{v_\omega}+A)$ almost surely.
\end{enumerate}

We will define the potential $v$ inductively as follows. Let $v_1$ be a potential of the form $v_{\omega}$ (with density $\rho$) such that $[-c,4+c]\subset \sigma(H_{v_1})$ and $\sigma(H_{v_1})$ is pure point. Such a $v_1$ exists by Theorem \ref{andlocal} and fact (2) above. Then for large enough $n$ there exists $z_n\in\Gamma_n(H_{v_1})$ such that $\left|z_n-2\right|\leq 1$. Fix $n=n_1$ such that this holds. Then $\Gamma_{n_1}(H_{v_1})$ only depends on $\{v_1(j):j\leq N_1\}$ for some integer $N_1$ by (i) of Definition \ref{Gen_alg}. Define the potential $v_2$ by $v_2(j)=v_1(j)$ for all $j\leq N_1$ and $v_2(j)=0$ otherwise. Then by fact (1) above, $\Gamma_n(H_{v_2})\cap [1/2,7/2]=\emptyset$ for large $n$, say for $n_2$. But then $\Gamma_{n_2}(H_{v_2})$ only depends on $\{v_2(j):j\leq N_2\}$ for some integer $N_2$. 

We repeat this process inductively switching between potentials which induce $\Gamma_{n_k}(H_{v_k})\cap [1/2,7/2]=\emptyset$ for $k$ even and potentials which induce $\Gamma_{n_k}(H_{v_k})\cap [1,3]\neq\emptyset$ for $k$ odd. Explicitly, if $k$ is even then define a potential $v_{k+1}$ by $v_{k+1}(j)=v_k(j)$ for all $j\leq N_k$ and $v_{k+1}(j)=v_{\omega}(j)$ (with the density $\rho$) otherwise such that $[-c,4+c]\subset \sigma(H_{v_{k+1}})$ and $\sigma(H_{v_{k+1}})$ is pure point. Such a $\omega$ exists from Theorem \ref{andlocal} and fact (2) above applied with the perturbation $A$ to match the potential for $j\leq N_k$. If $k$ is odd then we define $v_{k+1}$ by $v_{k+1}(j)=v_k(j)$ for all $j\leq N_k$ and $v_{k+1}(j)=0$ otherwise. We can then choose $n_{k+1}$ such that the above intersections hold and $N_{k+1}$ such that $\Gamma_{n_{k+1}}(H_{v_{k+1}})$ only depends on $\{v_{k+1}(j):j\leq N_{k+1}\}$. Finally set $v(j)=v_{k}(j)$ for $j\leq N_{k}$. It is clear from (iii) of Definition \ref{Gen_alg}, that $\Gamma_{n_k}(H_{v})=\Gamma_{n_k}(H_{v_k})$. But then this implies that $\Gamma_{n_k}(H_{v})$ cannot converge, the required contradiction.
\end{proof}

A similar argument gives the following theorem, where $\mathbb{V}$ is used to denote bounded real-valued potentials on $\mathbb{N}$ and $\Lambda_3$ denotes the pointwise evaluations of such potentials.

\begin{theorem}[Impossibility of computing spectral measures with error control]
\label{meas_err_impossible}
Consider the problem function
\begin{equation*}
\begin{split}
\widehat\Xi:\mathbb{V}\times \mathbb{N}&\rightarrow\mathbb{R}_{\geq 0}\\
(v,j)&\rightarrow \langle E_{\{1\}}^{H_v}e_j, e_j\rangle
\end{split}.
\end{equation*}
Then $\{\widehat\Xi,\mathbb{V}\times \mathbb{N},\Lambda_3\}\in\Delta_2^A$ but $\{\widehat\Xi,\mathbb{V}\times \mathbb{N},\Lambda_3\}\notin\Delta_1^G$. In other words, $\widehat\Xi$ can be computed in one limit, but it cannot be computed with error control.
\end{theorem}

\begin{proof}
The result $\{\widehat\Xi,\mathbb{V}\times \mathbb{N},\Lambda_3\}\in\Delta_2^A$ follows directly from the remarks after Theorem \ref{meas_comp1} and Theorem \ref{res_est1}. Suppose for a contradiction that $\{\widehat\Xi,\mathbb{V}\times \mathbb{N},\Lambda_3\}\in\Delta_1^G$ and that $\Gamma_n$ is a sequence of general algorithms solving the problem with error control. It follows that for each $j\in\mathbb{N}$, there exists a sequence of general algorithms $\Gamma_n^j$ such that
\begin{equation*}
\lim_{n\rightarrow\infty}\Gamma_n^j(v)=\begin{cases}
1,\quad &\text{if }\widehat\Xi(v,j)>0\\
0,\quad&\text{otherwise}
\end{cases}.
\end{equation*}
Informally, these are described as follows. Fix $j$ and consider the lower bound on $\widehat\Xi(v,j)$ computed by $\{\Gamma_m(v,j):m\leq n\}$. If this is greater than $0$ then set $\Gamma_n^j(v)=1$, otherwise set $\Gamma_n^j(v)=0$. It follows that $\Gamma_n^j(v)$ also converges from below. It holds that $1\in\sigma_{\mathrm{p}}(H_v)$ if and only if $\widehat\Xi(v,j)>0$ for some $j\in\mathbb{N}$. Now define
$$
\widehat\Gamma_{n}(v)=\sup_{j\leq n}\Gamma^j_n(v).
$$
It is clear that this is a general algorithm using $\Lambda_3$. Furthermore,
\begin{equation*}
\lim_{n\rightarrow\infty}\widehat\Gamma_n(v)=\begin{cases}
1,\quad &\text{if }1\in\sigma_{\mathrm{p}}(H_v)\\
0,\quad&\text{otherwise}
\end{cases},
\end{equation*}
with convergence from below.

Now we may choose a $v$ such that $1\in\sigma_{\mathrm{p}}(H_v)$ (this can be achieved for example by taking a potential which induces pure point spectrum and shifting the operator accordingly). It follows that for large $n$, we have $\widehat\Gamma_n(v)=1$. But the computation of $\widehat\Gamma_n(v)$ is only dependent on $v(j)$ for $j<N$ for some $N\in\mathbb{N}$. Define $v_0\in\mathbb{V}$ by $v_0(j)=v(j)$ if $j<N$ and $v_0(j)=0$ otherwise. It follows that $\widehat\Gamma_n(v_0)=1$. But since the potential has compact support, $1\notin\sigma_{\mathrm{p}}(H_{v_0})$ and hence $\widehat\Gamma_n(v_0)=0$, the required contradiction.
\end{proof}

We now prove that $\Xi_{\mathrm{pp}}^\mathbb{C}$ can be computed using a height two arithmetical tower. The first step is the following technical lemma, whose proof will also be used later when considering $\Xi_{\mathrm{ac}}^\mathbb{C}$.

\begin{lemma}
\label{pps_tech}
Let $a<b$ with $a,b\in\mathbb{R}$ and consider the decision problem
\begin{align*}
\Xi_{a,b,\mathrm{pp}}:\Omega_{f,\alpha}&\rightarrow\{0,1\}\\
T&\rightarrow\begin{dcases}
1\quad\text{if }\sigma_{\mathrm{pp}}(T)\cap[a,b]\neq\emptyset\\
0\quad\text{otherwise}.
\end{dcases}
\end{align*}
Then there exists a height two arithmetical tower $\Gamma_{n_2,n_1}$ (with evaluation functions $\Lambda_1$) for $\Xi_{a,b,\mathrm{pp}}$. Furthermore, the final limit is from below in the sense that $\Gamma_{n_2}(T):=\lim_{n_1\rightarrow\infty}\Gamma_{n_2,n_1}(T)\leq \Xi_{a,b,\mathrm{pp}}(T)$.
\end{lemma}

\begin{proof}
Step 1 of the proof of Theorem \ref{spec_decomp_comp} yields a height two arithmetical tower $\widehat \Gamma_{n_2,n_1}^j(T)$ for the computation of $\mu_{e_j,e_j,\mathrm{c}}^T((a,b))$. Note that the final limit is from above and using the fact that $\mu_{e_j,e_j,\mathrm{c}}^T(\{a,b\})=0$, we obtain a height two tower for $\mu_{e_j,e_j,\mathrm{c}}^T([a,b])$. We can then use the height one tower for $\mu_{e_j,e_j}^T([a,b])$ constructed in \S \ref{stone_sec}, denoted by $\widetilde \Gamma_{n_1}^j(T)$, and define
$$
a_{j,n_2,n_1}(T)=\widetilde \Gamma_{n_1}^j(T)-\widehat \Gamma_{n_2,n_1}^j(T).
$$
This provides a height two arithmetical tower for $\mu_{e_j,e_j,\mathrm{pp}}^T([a,b])$ with the final limit from below. Without loss of generality (by taking successive maxima), we can assume that these towers are non-decreasing in $n_2$. Now set
$$
\Upsilon_{n_2,n_1}(T)=\max_{1\leq j\leq n_2} a_{j,n_2,n_1}(T).
$$
Then it is clear that the limit $\lim_{n_1\rightarrow\infty}\Upsilon_{n_2,n_1}(T)=\Upsilon _{n_2}(T)$ exists. Furthermore, the monotonicity of $a_{j,n_2,n_1}(T)$ in $n_2$ implies that
$$
\lim_{n_2\rightarrow\infty}\Upsilon_{n_2}(T)=\sup_{n\in\mathbb{N}}\mu_{e_n,e_n,\mathrm{pp}}^T([a,b]),
$$
with monotonic convergence from below. This limiting value is zero if $\Xi_{a,b,\mathrm{pp}}(T)=0$, otherwise it is a positive finite number.

To convert this to a height two tower for the decision problem $\Xi_{a,b,\mathrm{pp}}$, that maps to the discrete space $\{0,1\}$, we use the following trick. Consider the intervals
$
J_1^{n_2}=[0,1/n_2],
$ and 
$
J_2^{n_2}=[2/n_2,\infty).
$
Let $k(n_2,n_1)\leq n_1$ be maximal such that $\Upsilon_{n_2,n_1}(T)\in J_1^{n_2}\cup J_2^{n_2}$. If no such $k$ exists or $\Upsilon_{n_2,k}(T)\in J_1^{n_2}$ then set ${\Gamma}_{n_2,n_1}(T)=0$. Otherwise set ${\Gamma}_{n_2,n_1}(T)=1$. These can be computed using finitely many arithmetic operations and comparisons using $\Lambda_1$. The point of the intervals $J_1^{n_2}$ and $J_2^{n_2}$ is that we can show $\lim_{n_1\rightarrow\infty}{\Gamma}_{n_2,n_1}(T)={\Gamma}_{n_2}(T)$ exists. This is because $\lim_{n_1\rightarrow\infty}\Upsilon_{n_2,n_1}(T)=\Upsilon_{n_2}(T)$ exists and hence we cannot oscillate infinitely often between the separated intervals $J_1^{n_2}$ and $J_2^{n_2}$. Now suppose that $\Xi_{a,b,\mathrm{pp}}(T)=0$, then $\lim_{n_1\rightarrow\infty}\widehat\Gamma_{n_2,n_1}(T)=0$ and hence $\lim_{n_1\rightarrow\infty}\Gamma_{n_2,n_1}(T)=0$ for all $n_2$. Now suppose that $\Xi_{a,b,\mathrm{pp}}(T)=1$, then for large enough $n_2$ we must have that $\Upsilon_{n_2}(T)>2/n_2$ and hence $\Gamma_{n_2}(T)=1$. Together, these prove the convergence and that $\Gamma_{n_2}(T)\leq \Xi_{a,b,\mathrm{pp}}(T)$.
\end{proof}

\begin{proof}[Proof that $\{\Xi_{\mathrm{pp}}^{\mathbb{C}},\Omega_{f,\alpha},\Lambda_1\}\in \Delta_3^A$]
To construct a height two arithmetical tower for $\Xi_{\mathrm{pp}}^{\mathbb{C}}$, we will use Lemma \ref{pps_tech} repeatedly. Let $\widehat\Gamma_{n_2,n_1}(\cdot,I)$ denote the height two tower constructed in the proof of Lemma \ref{pps_tech} for the closed interval $I$ ($I=[a,b]$), where without loss of generality by taking successive maxima in $n_2$, we can assume that this tower is non-decreasing in $n_2$ (this is where we use convergence from below in the final limit in the statement of the lemma). For a given $n_1$ and $n_2$, we construct $\Gamma_{n_2,n_1}(T)$ as follows (we will use some basic terminology from graph theory).

Define the intervals $I_{n_2,n_1,j}^{0}=[j,j+1]$ for $j=-n_2,...,n_2-1$ so that these cover the interval $[-n_2,n_2]$. Now suppose that $I_{n_2,n_1,j}^{k}$ are defined for $j=1,...,r_k(n_2,n_1,T)$. Compute each $\widehat\Gamma_{n_2,n_1}(T,{I_{n_2,n_1,j}^{k}})$ and if this is 1, bisect $I_{n_2,n_1,j}^{k}$ via its midpoint into two equal halves consisting of closed intervals. We then take all these bisected intervals and label them as $I_{n_2,n_1,j}^{k+1}$ for $j=1,...,r_{k+1}(n_2,n_1,k,T)$. This is repeated until we have no further bisections or the intervals $I_{n_2,n_1,j}^{n_2}$ have been computed. By adding the interval $[-n_2,n_2]$ as a root with children $I_{n_2,n_1,j}^{0}$, this creates a finite tree structure where a non-root interval $I$ is a parent of two intervals precisely if those two intervals are formed from its bisection and $\widehat\Gamma_{n_2,n_1}(T,I)=1$. We then prune this tree by discarding all leaves $I$ which have $\widehat\Gamma_{n_2,n_1}(T,I)=0$ to form the tree $\mathcal{T}_{n_2,n_1}(T)$. Finally, we let $\Gamma_{n_2,n_1}(T)$ be the union of all the leaves of $\mathcal{T}_{n_2,n_1}(T)$. Clearly this can be computed using finitely many arithmetic operations and comparisons using $\Lambda_1$. The construction is shown visually in Figure \ref{tree_construct}.

In the above construction, the number of intervals considered (including those not in the tree $\mathcal{T}_{n_2,n_1}(T)$) for a fixed $n_2$ is $n_22^{n_2+1}+1$ and hence independent of $n_1$. It follows that $\mathcal{T}_{n_2,n_1}(T)$ and $\Gamma_{n_2,n_1}(T)$ are constant for large $n_1$ (due to the convergence of the $\widehat\Gamma_{n_2,n_1}(T,I)$ in $\{0,1\}$). We denote these limiting values by $\mathcal{T}_{n_2}(T)$ and $\Gamma_{n_2}(T)$ respectively and also denote the corresponding intervals in the construction at the $m-$th level of this limit by $I_{n_2,j}^{m}$. Note also that if $\Xi_{\mathrm{pp}}^{\mathbb{C}}(T)=\emptyset$ then $\Gamma_{n_2}(T)=\emptyset$.

Now suppose that $z\in\Xi_{\mathrm{pp}}^{\mathbb{C}}(T)$, then there exists a sequence of nested intervals $I_m=I_{n_2,a_{m,n_2}}^m$ containing $z$ for $m=0,...,n_2$, where these intervals are independent of $n_2$. Fix $m$, then for large $n_2$ we must have that $\widehat\Gamma_{n_2}(T,I_{j})=1$ for $j=1,...,m$. It follows that $I_m$ has a descendent interval $I_{n_2,m}$ contained in $\Gamma_{n_2}(T)$ and hence we must have
$$
\mathrm{dist}(z,\Gamma_{n_2}(T))\leq 2^{-m}.
$$
Since $m$ was arbitrary it follows that $\mathrm{dist}(z,\Gamma_{n_2}(T))$ converges to $0$ as $n_2\rightarrow\infty$.

Conversely, suppose that $z_{m_j}\in\Gamma_{m_j}(T)$ with $m_j\rightarrow\infty$, then we must show that all limit points of $\{z_{m_j}\}$ lie in $\Xi_{\mathrm{pp}}^{\mathbb{C}}(T)$. Suppose this were false, then by taking a subsequence if necessary, we can assume that $z_{m_j}\rightarrow z$ and $\mathrm{dist}(z_{m_j},\Xi_{\mathrm{pp}}^{\mathbb{C}}(T))\geq\delta$ for some $\delta>0$. We claim that it is sufficient to prove that the maximum length of the leaves of $\mathcal{T}_{n_2}(T)$ intersecting a fixed compact subset of $\mathbb{R}$, converges to zero as $n_2\rightarrow\infty$. Suppose this has been shown, then $z_{m_j}\in I_{m_j}$ for some leaf $I_{m_j}$ of $\mathcal{T}_{m_j}(T)$.  It follows that $I_{m_j}\cap \Xi_{\mathrm{pp}}^{\mathbb{C}}(T)\neq\emptyset$ and $\left|I_{m_j}\right|\rightarrow 0$. But this contradicts $z_{m_j}$ being positively separated from $\Xi_{\mathrm{pp}}^{\mathbb{C}}(T)$.

We are thus left with proving the claim regarding the lengths of leaves. Suppose this were false, then there exists a compact set $K\subset\mathbb{R}$ and leaves $I_{j}$ in $\mathcal{T}_{b_j}(T)$ such that the lengths of $I_j$ do not converge to zero and $I_j$ intersect $K$. By taking subsequences if necessary, we can assume that the lengths of each $I_j$ are constant. Then by the compactness of $K$ and taking subsequences if necessary again, we can assume that each of the $I_j$ are equal to a common interval $I$. It follows that $\widehat \Gamma_{b_j}(T,I)=1$ but that $\widehat \Gamma_{b_j}(T,I_1)=\widehat \Gamma_{b_j}(T,I_2)=0$ since $I$ is a leaf, where $I_1$ and $I_2$ form the bisection of $I$. Taking $b_j\rightarrow\infty$, this implies that $I\cap \Xi_{\mathrm{pp}}^{\mathbb{C}}(T)\neq\emptyset$ but $I_1\cap \Xi_{\mathrm{pp}}^{\mathbb{C}}(T)=I_2\cap \Xi_{\mathrm{pp}}^{\mathbb{C}}(T)=\emptyset$ which is absurd. Hence we have shown the required contradiction, and we have finished the proof.
\end{proof}

\begin{figure}
\centering
\includegraphics[width=0.8\textwidth,trim={0mm 0mm 0mm 0mm},clip]{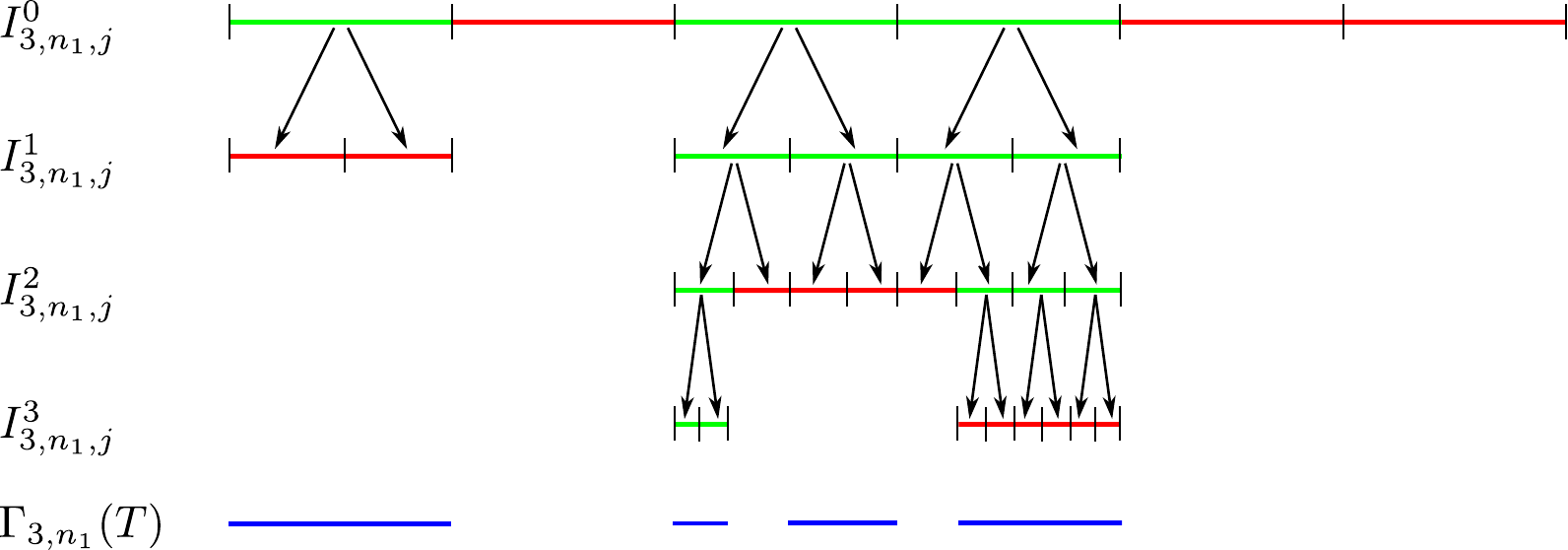}
\caption{Example of tree structure used to compute the point spectrum for $n_2=3$. Each tested interval is shown in green ($\widehat\Gamma_{n_2,n_1}(T,I)=1$) or red ($\widehat\Gamma_{n_2,n_1}(T,I)=0$). The arrows show the bisections and the final output is shown in blue.}
\label{tree_construct}
\end{figure}

\subsection{Proof for absolutely continuous spectra}

To prove the lower bound (that one limit will not suffice), our strategy will be to reduce a certain decision problem to the computation of $\Xi_{\mathrm{ac}}^{\mathbb{C}}$. Let $(\mathcal{M}',d')$ be the discrete space $\{0,1\}$, let $\Omega'$ denote the collection of all infinite sequence $\{a_{j}\}_{j\in\mathbb{N}}$ with entries $a_{j}\in\{0,1\}$ and consider the problem function
\begin{equation*}
\Xi'(\{a_{j}\}):\text{ `Does }\{a_{j}\}\text{ have infinitely many non-zero entries?'},
\end{equation*}
that maps to $(\mathcal{M}',d')$. In Appendix \ref{append1}, it is shown that $\mathrm{SCI}(\Xi',\Omega')_{G} = 2$ (where the evaluation functions consist of component-wise evaluation of the array $\{a_{j}\}$).

\begin{proof}[Proof that $\{\Xi_{\mathrm{ac}}^{\mathbb{C}},\Omega_{f,\alpha},\Lambda_2\}\notin \Delta_2^G$] We are done if we prove the result for $f(n)=n+1$ and $\alpha=0$. In this case $\Lambda_1$ and $\Lambda_2$ are equivalent so we can restrict the argument to $\Lambda_1$. Suppose for a contradiction that $\Gamma_{n}$ is a height one tower of general algorithms solving $\{\Xi_{\mathrm{ac}}^{\mathbb{C}},\Omega_{f,0},\Lambda_2\}$. We will gain a contradiction by using the supposed tower to solve $\{\Xi',\Omega'\}$. 

Given $\{a_{j}\}\in\Omega'$, consider the operator $H=H_0+v$, where the potential is of the following form:
$$
v(m)=\sum_{k=1}^\infty a_{k}\delta_{m,k!}.
$$
Then by Theorem \ref{oracle_lemma}, $[0,4]\subset\sigma_{\mathrm{ac}}(H)$ if $\sum_{k}a_{k}<\infty$ (that is, if $\Xi'(\{a_j\})=0$) and $\sigma_{\mathrm{ac}}(H)\cap(0,4)=\emptyset$ otherwise. Given $N$ we can evaluate any matrix value of $H$ using only finitely many evaluations of $\{a_{j}\}$ and hence the evaluation functions $\Lambda_1$ can be computed using component-wise evaluations of the sequence $\{a_j\}$. We now set
$$
\widehat\Gamma_{n}(\{a_{j}\})=\begin{dcases}
0\quad\text{if }\mathrm{dist}(2,\Gamma_n(H))<1\\
1\quad\text{otherwise}.
\end{dcases}
$$
The above comments show that each of these is a general algorithm, and it is clear that it converges to $\Xi'(\{a_j\})$ as $n\rightarrow\infty$, the required contradiction.
\end{proof}

To construct the height two (arithmetical) tower for $\Xi_{\mathrm{ac}}^{\mathbb{C}}$, we require the following lemma.

\begin{lemma}
\label{acs_tech}
Let $a<b$ with $a,b\in\mathbb{R}$ and consider the decision problem
\begin{align*}
\Xi_{a,b,\mathrm{ac}}:\Omega_{f,\alpha}&\rightarrow\{0,1\}\\
T&\rightarrow\begin{dcases}
1\quad\text{if }\sigma_{\mathrm{ac}}(T)\cap[a,b]\neq\emptyset\\
0\quad\text{otherwise}.
\end{dcases}
\end{align*}
Then there exists a height two arithmetical tower $\Gamma_{n_2,n_1}$ (with evaluation functions $\Lambda_1$) for $\Xi_{a,b,\mathrm{ac}}$. Furthermore, the final limit is from below with $\Gamma_{n_2}(T):=\lim_{n_1\rightarrow\infty}\Gamma_{n_2,n_1}(T)\leq \Xi_{a,b,\mathrm{ac}}(T)$.
\end{lemma}

\begin{proof}
Fix such an $a$ and $b$ and let $\chi_n$ be a sequence of non-negative, continuous piecewise affine functions on $\mathbb{R}$, bounded by $1$ and of compact support such that $\chi_n$ converge pointwise monotonically up to the constant function $1$. Define also the function
$$
\upsilon_{m,n}(u,T)=\langle K_{H}(u+i/n,T,e_m),e_m\rangle
$$ 
and set
$$
a_{m,n_2,n_1}(T)=\int_{a}^b \upsilon_{m,n_1}(u,T)\chi_{n_2}(\left|\upsilon_{m,n_1}(u,T)\right|)du.
$$
Since each $\chi_n$ is continuous and has compact support (which implies that the integrand is bounded for fixed $n_2$), and since $\upsilon_{m,n}(u,T)$ converges almost everywhere to $\rho^T_{e_{m},e_m}(u)$ (the Radon--Nikodym derivative of the absolutely continuous part of the measure $\mu_{e_m,e_m}^T$), it follows by the dominated convergence theorem that
$$
\lim_{n_1\rightarrow\infty}a_{m,n_2,n_1}(T)=:a_{m,n_2}(T)=\int_{a}^b \rho^T_{e_{m},e_m}(u)\chi_{n_2}(\rho^T_{e_{m},e_m}(u))du.
$$
We now use the fact that the $\chi_n$ are increasing and the dominated convergence theorem again (with bounding integrable function $\rho^T_{e_{m},e_m}$) to deduce that
$$
\lim_{n_2\rightarrow\infty}a_{m,n_2}(T)=\mu_{e_m,e_m,\mathrm{ac}}^T([a,b]),
$$
with monotonic convergence from below.

Using Corollary \ref{res_est2} (and the now standard argument of Lipschitz continuity of the resolvent), we can compute approximations of $a_{m,n_2,n_1}(T)$ to accuracy $1/n_1$ in finitely many arithmetic operations and comparisons. Call these approximations $\widetilde a_{m,n_2,n_1}(T)$ and set
$$
\Upsilon_{n_2,n_1}(T)=\max_{1\leq j\leq n_2} \widetilde a_{j,n_2,n_1}(T).
$$
The proof now follows that of Lemma \ref{pps_tech} exactly.
\end{proof}

\begin{proof}[Proof that $\{\Xi_{\mathrm{ac}}^{\mathbb{C}},\Omega_{f,\alpha},\Lambda_1\}\in \Delta_3^A$]
This is exactly the same construction as in the above proof of the inclusion $\{\Xi_{\mathrm{pp}}^{\mathbb{C}},\Omega_{f,\alpha},\Lambda_1\}\in \Delta_3^A$. We simply replace the tower constructed in the proof of Lemma \ref{pps_tech} by the tower constructed in the proof of Lemma \ref{acs_tech}.
\end{proof}

\subsection{Proof for singular continuous spectra}

We first prove the lower bound for the singular continuous spectrum via Theorem \ref{oracle_lemma}. Note that the impossibility result $\{\Xi_{\mathrm{sc}}^{\mathbb{C}},\Omega_{f,\alpha},\Lambda_2\}\notin \Delta_2^G$ follows from the same argument that was used to show $\{\Xi_{\mathrm{ac}}^{\mathbb{C}},\Omega_{f,\alpha},\Lambda_2\}\notin \Delta_2^G$. To show that two limits will not suffice for $f(n)-n\geq \sqrt{2n}+1/2$, our strategy will be to again reduce a certain decision problem to the computation of $\Xi_{\mathrm{sc}}^{\mathbb{C}}$. Let $(\mathcal{M}',d')$ be the discrete space $\{0,1\}$, let $\Omega'$ denote the collection of all infinite matrices $\{a_{i,j}\}_{i,j\in\mathbb{N}}$ with entries $a_{i,j}\in\{0,1\}$ and consider the problem function
\begin{equation*}
\Xi'(\{a_{i,j}\}):\text{ `Does }\{a_{i,j}\}\text{ have a column containing infinitely many non-zero entries?'},
\end{equation*}
that maps to $(\mathcal{M}',d')$. In \cite{ben2015can}, a Baire category argument was used to prove that $\mathrm{SCI}(\Xi',\Omega')_{G} = 3$ (where the evaluation functions consist in component-wise evaluation of the array $\{a_{i,j}\}$).

\begin{proof}[Proof that $\{\Xi_{\mathrm{sc}}^{\mathbb{C}},\Omega_{f,\alpha},\Lambda_2\}\notin \Delta_3^G$ if $f(n)-n\geq \sqrt{2n}+1/2 $] Assume that the function $f$ satisfies $f(n)-n\geq \sqrt{2n}+1/2$. The proof will use a direct sum construction. Given $\{a_{i,j}\}\in\Omega'$, consider the operators $H_j=H_0+v_{(j)}$, where the potential is of the following form:
$$
v_{(j)}(n)=\sum_{k=1}^\infty a_{k,j}\delta_{n,k!}.
$$
Using Theorem \ref{oracle_lemma}, $[0,4]\subset\sigma_{\mathrm{sc}}(H_j)$ if $\sum_{k}a_{k,j}=\infty$ (that is, if the $j$-th column has infinitely many $1$s) and $\sigma_{\mathrm{sc}}(H_j)\cap(0,4)=\emptyset$ otherwise. Now consider an effective bijection (with effective inverse) between the canonical bases of $l^2(\mathbb{N})$ and $\oplus_{j=1}^\infty l^2(\mathbb{N})$:
$$
\phi:\{e_n:n\in\mathbb{N}\}\rightarrow \{e_{\textbf{k}}:\textbf{k}\in \mathbb{N}^{\mathbb{N}},\|\textbf{k}\|_{0}=1\}.
$$
Set $H(\{a_{i,j}\})=\bigoplus_{j=1}^\infty H_j$. Then through $\phi$, we view $H=H(\{a_{i,j}\})$ as a self-adjoint operator acting on $l^2(\mathbb{N})$. Explicitly, we consider the matrix
$$
H_{m,n}=\langle H e_{\phi(n)},e_{\phi(m)}\rangle.
$$
We choose the following bijection (where $m$ lists the canonical basis in each Hilbert space):
\begin{center}
\begin{tikzpicture}
\matrix[matrix of math nodes,inner sep=1pt,row sep=1em,column sep=1em] (M)
{ & j=1 & j=2 & j=3 & \cdots\\
    m=1 & \phi(1) & \phi(3) & \phi(6) & \cdots \\
    m=2 & \phi(2) & \phi(5) &   &  \\
    m=3 & \phi(4) &   &   &  \\
    \cdots & \cdots &  &  \\
}
;
\draw[->] (M-2-2.south west) -- (M-2-2.north east);
\draw[->] (M-3-2.south west) -- (M-2-3.north east);
\draw[->] (M-4-2.south west) -- (M-2-4.north east);
\draw[->] (M-5-2.south west) -- (M-2-5.north east);
\end{tikzpicture}
\end{center}
A straightforward computation shows that $H\in\Omega_{f,0}$. We also observe that if $\Xi'(\{a_{i,j}\})=1$ then $[0,4]\subset\sigma_{\mathrm{sc}}(H)$, otherwise $\sigma_{\mathrm{sc}}(H)\cap(0,4)=\emptyset$.

Suppose for a contradiction that $\Gamma_{n_2,n_1}$ is a height two tower of general algorithms that solves $\{\Xi_{\mathrm{sc}}^{\mathbb{C}},\Omega_{f,0},\Lambda_1\}$. We will gain a contradiction by using the supposed height two tower to solve $\{\Xi',\Omega'\}$. We now set
$$
\widehat\Gamma_{n_2,n_1}(\{a_{i,j}\})=1-\min\{1,\mathrm{dist}(3,\Gamma_{n_2,n_1}(H(\{a_{i,j}\})))\},
$$
where we use the convention $\mathrm{dist}(3,\emptyset)=1$. The comments above show that each of these is a general algorithm. Furthermore, the convergence of $\Gamma_{n_2,n_1}$ implies that
$$
\lim_{n_2\rightarrow\infty}\lim_{n_1\rightarrow\infty}\widehat\Gamma_{n_2,n_1}(\{a_{i,j}\})=1-\min\{1,\mathrm{dist}(3,\sigma_{\mathrm{sc}}(H(\{a_{i,j}\})))\}=\Xi'(\{a_{i,j}\}).
$$
We are not quite done since the convergence here takes place on the interval $[0,1]$ with the usual metric as opposed to $\{0,1\}$ with the discrete metric. To get round this, we use the following, now familiar, trick.

Consider the intervals
$
J_1=[0,1/2],
$ and 
$
J_2=[3/4,1].
$
Let $k(n_2,n_1)\leq n_1$ be maximal such that $\widehat\Gamma_{n_2,k}(\{a_{i,j}\})\in J_1\cup J_2$. If no such $k$ exists or $\widehat\Gamma_{n_2,k}(\{a_{i,j}\})\in J_1$ then set ${\Gamma}_{n_2,n_1}'(\{a_{i,j}\})=0$. Otherwise set ${\Gamma}_{n_2,n_1}'(\{a_{i,j}\})=1$. As in the proof of Lemma \ref{pps_tech}, we can show $\lim_{n_1\rightarrow\infty}{\Gamma}_{n_2,n_1}'(\{a_{i,j}\})={\Gamma}_{n_2}'(\{a_{i,j}\})$ exists. If $\Xi'(\{a_{i,j}\})=0$, then for large $n_2$, $\lim_{n_1\rightarrow\infty}\widehat\Gamma_{n_2,k}(\{a_{i,j}\})<1/2$ and hence $\lim_{n_2\rightarrow\infty}{\Gamma}_{n_2}'(\{a_{i,j}\})=0$. Similarly, if $\Xi'(\{a_{i,j}\})=1$, then for large $n_2$ we must have $\lim_{n_1\rightarrow\infty}\widehat\Gamma_{n_2,k}(\{a_{i,j}\})>3/4$ and hence $\lim_{n_2\rightarrow\infty}{\Gamma}_{n_2}'(\{a_{i,j}\})=1$. Hence ${\Gamma}_{n_2,n_1}'$ is a height two tower of general algorithms solving $\{\Xi',\Omega'\}$, a contradiction.
\end{proof}

Finally, we will use the following lemma to prove that the singular continuous spectrum can be computed in three limits.

\begin{lemma}
\label{scs_tech}
Let $a<b$ with $a,b\in\mathbb{R}$ and consider the decision problem
\begin{align*}
\Xi_{a,b,\mathrm{sc}}:\Omega_{f,\alpha}&\rightarrow\{0,1\}\\
T&\rightarrow\begin{dcases}
1\quad\text{if }\sigma_{\mathrm{sc}}(T)\cap[a,b]\neq\emptyset\\
0\quad\text{otherwise}.
\end{dcases}
\end{align*}
Then there exists a height three arithmetical tower $\Gamma_{n_3,n_2,n_1}$ (with evaluation functions $\Lambda_1$) for $\Xi_{a,b,\mathrm{sc}}$. Furthermore, the final limit is from below with $\Gamma_{n_3}(T):=\lim_{n_2\rightarrow\infty}\lim_{n_1\rightarrow\infty}\Gamma_{n_3,n_2,n_1}(T)\leq \Xi_{a,b,\mathrm{sc}}(T)$.
\end{lemma}

Once this is proven, we can use the same construction that was used to prove $\{\Xi_{\mathrm{pp}}^{\mathbb{C}},\Omega_{f,\alpha},\Lambda_1\}\in \Delta_3^A$ and $\{\Xi_{\mathrm{ac}}^{\mathbb{C}},\Omega_{f,\alpha},\Lambda_1\}\in \Delta_3^A$ to show that $\{\Xi_{\mathrm{sc}}^{\mathbb{C}},\Omega_{f,\alpha},\Lambda_1\}\in \Delta_4^A$, but with an additional limit. Namely, we replace $(n_2,n_1)$ by $(n_3,n_2)$ in the proof and use the tower constructed in the proof of Lemma \ref{acs_tech} instead of $\widehat \Gamma_{n_2,n_1}(T,I)$ for an interval $I$. We still gain the required convergence, since the only change is an additional limit in the finite number of decision problems that decide the appropriate tree.

\begin{proof}[Proof of Lemma \ref{scs_tech}]
Note that we can write
$$
\mu_{e_m,e_m,\mathrm{sc}}^T([a,b])=\mu_{e_m,e_m}^T([a,b])-\mu_{e_m,e_m,\mathrm{pp}}^T([a,b])-\mu_{e_m,e_m,\mathrm{ac}}^T([a,b]).
$$
From this and the proofs of Lemmas \ref{pps_tech} and \ref{acs_tech}, it is clear that we can construct a height two arithmetical tower, $a_{m,n_2,n_1}(T)$, for $\mu_{e_m,e_m,\mathrm{sc}}^T([a,b])$, where the final limit is from above. Now set
$$
\Upsilon_{n_3,n_2,n_1}(T)=\max_{1\leq j\leq n_3} a_{j,n_2,n_1}.
$$
We see that each successive limit converges, with the second from above and the final from below.  By taking successive maxima, minima of our base algorithms, we can assume that the second and final limits are monotonic and that $\Upsilon_{n_3,n_2,n_1}(T)$ is monotonic in both $n_2$ and $n_3$. Define $\Upsilon_{n_3,n_2}(T)=\lim_{n_1\rightarrow\infty}\Upsilon_{n_3,n_2,n_1}(T)$, $\Upsilon_{n_3}(T)=\lim_{n_2\rightarrow\infty}\Upsilon_{n_3,n_2}(T)$ and $\Upsilon(T)=\lim_{n_3\rightarrow\infty}\Upsilon_{n_3}(T)$. Then $\Upsilon(T)$ is zero if $\Xi_{a,b,\mathrm{sc}}(T)=0$, otherwise it is a positive finite number.

With a slight change to the previous argument (the monotonicity in $n_2$ and $n_3$ is crucial for this to work), consider the intervals
$
J_1^{m}=[0,1/m],
$ and 
$
J_2^{m}=[2/m,\infty).
$
Let $k(m,n,n_1)\leq n_1$ be maximal such that $\Upsilon_{m,n,n_1}(T)\in J_1^{m}\cup J_2^{m}$. If no such $k$ exists or $\Upsilon_{m,n,k}(T)\in J_1^{m}$ then set ${\widehat\Gamma}_{m,n,n_1}(T)=0$. Otherwise set ${\widehat\Gamma}_{m,n,n_1}(T)=1$. We then define
$$
\Gamma_{n_3,n_2,n_1}=\max_{1\leq m\leq n_3}\min_{1\leq n\leq n_2}{\widehat\Gamma}_{m,n,n_1}(T).
$$
These can be computed using finitely many arithmetic operations and comparisons using $\Lambda_1$, and, as before, the first limit exists with
$$
\Gamma_{n_3,n_2}(T)=\lim_{n_1\rightarrow\infty}\Gamma_{n_3,n_2,n_1}(T)=\max_{1\leq m\leq n_3}\min_{1\leq n\leq n_2}{\widehat\Gamma}_{m,n}(T).
$$
Note also that the second and third sequential limits exist through the use of maxima and minima.

Now suppose that $\Xi_{a,b,\mathrm{sc}}(T)=0$ and fix $n_3$. Then for large $n_2$, we must have that $\Upsilon_{m,n_2}(T)<1/(2n_3)$ for all $m\leq n_3$ due to the monotonic convergence of $\Upsilon_{p}$ as $p\rightarrow\infty$. It follows in this case that
$$
\lim_{n_2\rightarrow\infty }\Gamma_{n_3,n_2}(T)=0,\quad\text{for all }n_3.
$$

Now suppose that $\Xi_{a,b,\mathrm{sc}}(T)=1$. It follows in this case that there exists $M\in\mathbb{N}$ such that if $m\geq M$ then $\Upsilon_m(T)>3/m$. Due to the monotonic convergence of $\Upsilon_{m,p}$ as $p\rightarrow\infty$ it follows that for all $p$ we must have $\Upsilon_{m,p}>3/m$ and hence there exists $N(m,p)\in\mathbb{N}$ such that if $n_1\geq N(m,p)$ then we must have $\Upsilon_{m,p,n_1}\geq 2/m$. It follows that if $n_3\geq M$ then we must have $\widehat\Gamma_{n_3,p}(T)=1$ for all $p$ and hence that
$$
\lim_{n_3\rightarrow\infty}\Gamma_{n_3}(T)=1.
$$
The conclusion of the lemma now follows.
\end{proof}

\section{Numerical Examples}
\label{num_sec}

We now demonstrate the applicability of the new algorithms. In particular, these are the first algorithms that compute their respective spectral properties for the whole class $\Omega_{f,\alpha}$, and even for the restricted case of tridiagonal self-adjoint matrices. The algorithms are also implicitly parallelisable, allowing large scale computations. We focus on discrete operators and the numerical implementation of the algorithms for ODEs, PDEs and integral operators will be the focus of a future software package.\footnote{Since writing the initial version of this paper, our algorithms have been implemented in the software package \texttt{SpecSolve} of \cite{colbrook2020computingSM}, which treats general discrete, differential and integral operators.}

\subsection{Jacobi matrices and orthogonal polynomials}
\label{JACOBI_SEC_1}

For our first set of examples, we consider the natural link between the spectral measures of Jacobi operators and orthogonal polynomials on $\mathbb{R}$. Let $J$ be a Jacobi matrix
$$
J=\begin{pmatrix}
b_1 & a_1 &  &\\
a_1 & b_2 & a_2 & \\
 & a_2 & b_3 & \ddots\\
 &  & \ddots & \ddots
\end{pmatrix}
$$
with $a_j,b_j\in\mathbb{R}$ and $a_j>0$. In this case, under suitable conditions, the probability measure $\mu_J:=\mu_{e_1,e_1}^J$ is the probability measure associated with the orthonormal polynomials defined by
\begin{equation*}
\begin{split}
&xP_k(x)=a_{k+1}P_{k+1}(x)+b_{k+1}P_k(x)+a_kP_{k-1}(x),\\
&P_{-1}(x)=0,\quad P_0(x)=1,
\end{split}
\end{equation*}
and the spectral measure that appears in the multiplicative version of the spectral theorem (see, for example, \cite{teschl2000jacobi,deift1999orthogonal,stone1932linear}). Classically, one usually first considers the measure and then constructs the orthogonal polynomials and the corresponding $J$. In some sense, the algorithms constructed in this paper, and in particular \S \ref{applic_sec}, compute the inverse problem (note that $J\in\Omega_{n\rightarrow n+1,0}$). In other words, we compute the measure $\mu_J$ given the recurrence coefficients defining the orthogonal polynomials. This is a very difficult problem to study theoretically, and, until now, there has been no numerical method able to tackle the problem in this generality (see \S \ref{prev_work} for comments on methods that tackle compact perturbations of Toeplitz operators). To verify our method, we will consider problems where the measure $\mu_J$ is known analytically.

We begin with the well-known class of Jacobi polynomials defined for $\alpha,\beta>-1$ which have
$$
a_k=2\sqrt{\frac{k(k+\alpha)(k+\beta)(k+\alpha+\beta)}{(2k+\alpha+\beta-1)(2k+\alpha+\beta)^2(2k+\alpha+\beta+1)}},\quad b_k=\frac{\beta^2-\alpha^2}{(2k+\alpha+\beta)(2k-2+\alpha+\beta)},
$$
and measure on the interval $[-1,1]$ given by
\begin{equation}
d\mu_J=\frac{(1-x)^\alpha(1+x)^\beta}{N(\alpha,\beta)}dx=f_{\alpha,\beta}(x)dx,
\end{equation}
where $N(\alpha,\beta)$ is a normalising constant, ensuring the measure is a probability measure. To assess the convergence of the algorithm in \S \ref{app1_RN} that approximates the Radon--Nikodym derivative $f_{\alpha,\beta}$, in this section we will plot various errors as a function of $\epsilon$, the distance from the points at which we compute the resolvents to the real axis.

\begin{figure}
\centering
\includegraphics[width=0.49\textwidth,clip]{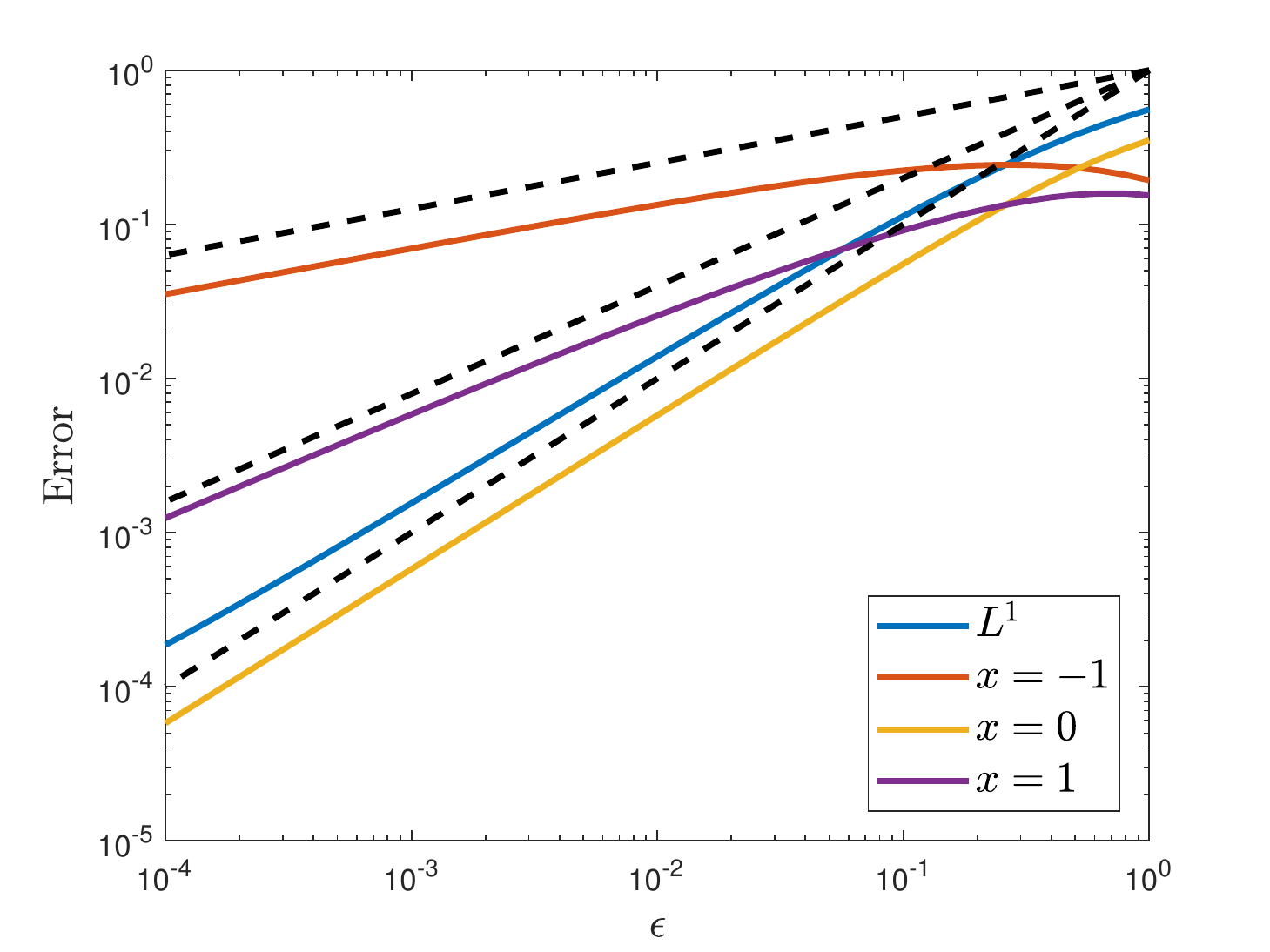}
\includegraphics[width=0.49\textwidth,clip]{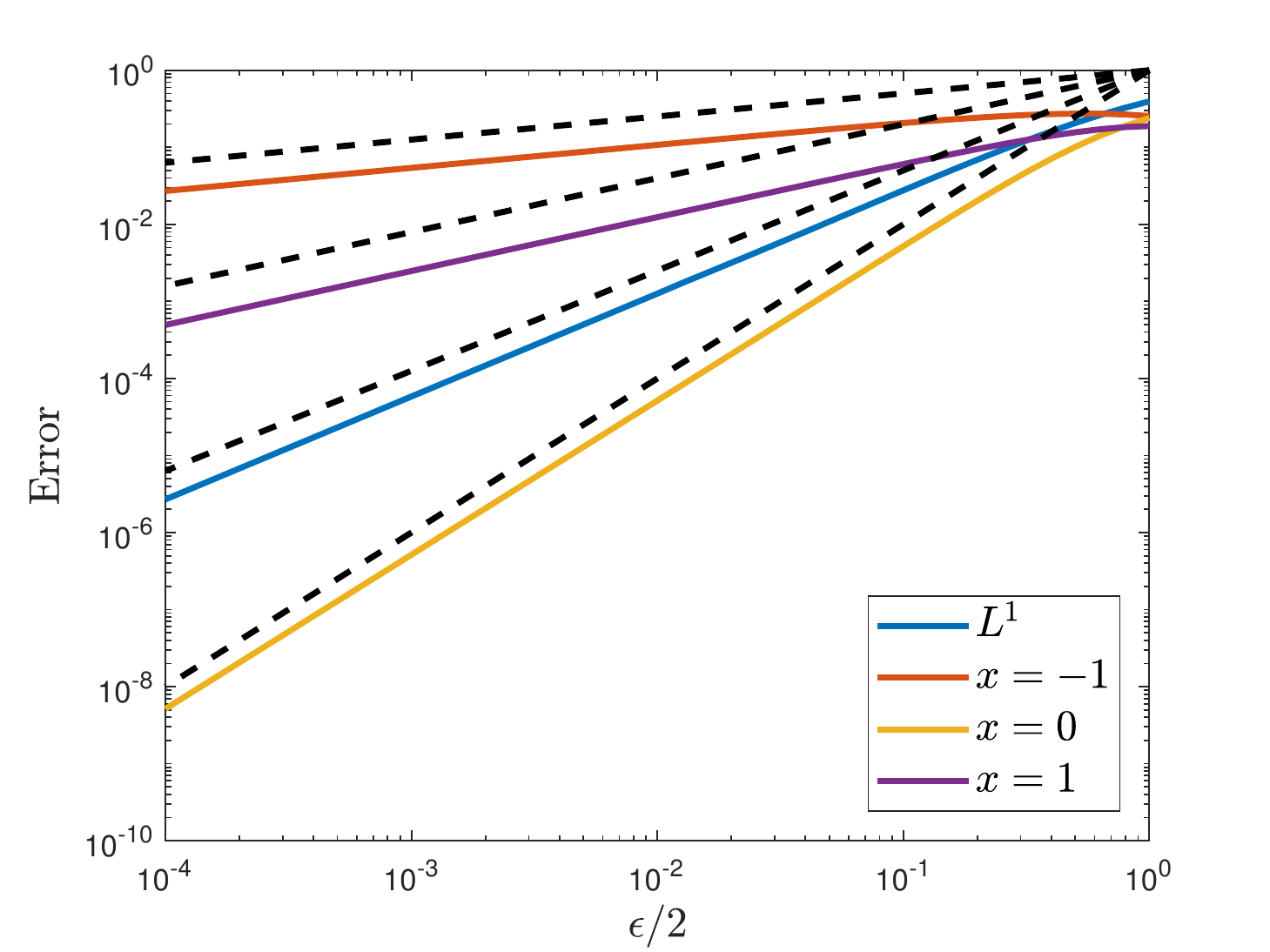}
\caption{Results for Jacobi polynomials with $\alpha=0.7$ and $\beta=0.3$. Left: Convergence in $L^1$ and at the points $\pm 1$ and $0$. The rates $O(\epsilon)$, $O(\epsilon^{0.7})$ and $O(\epsilon^{0.3})$ are also shown as dashed lines. Right: Convergence with Richardson extrapolation. The rates $O(\epsilon^2)$, $O(\epsilon^{1.3})$, $O(\epsilon^{0.7})$ and $O(\epsilon^{0.3})$ are also shown. As discussed in the text, the rates reflect the local smoothness properties of the density $f_{\alpha,\beta}$.}
\label{Jacobi1}
\end{figure}

\begin{figure}
\centering
\includegraphics[width=0.49\textwidth,clip]{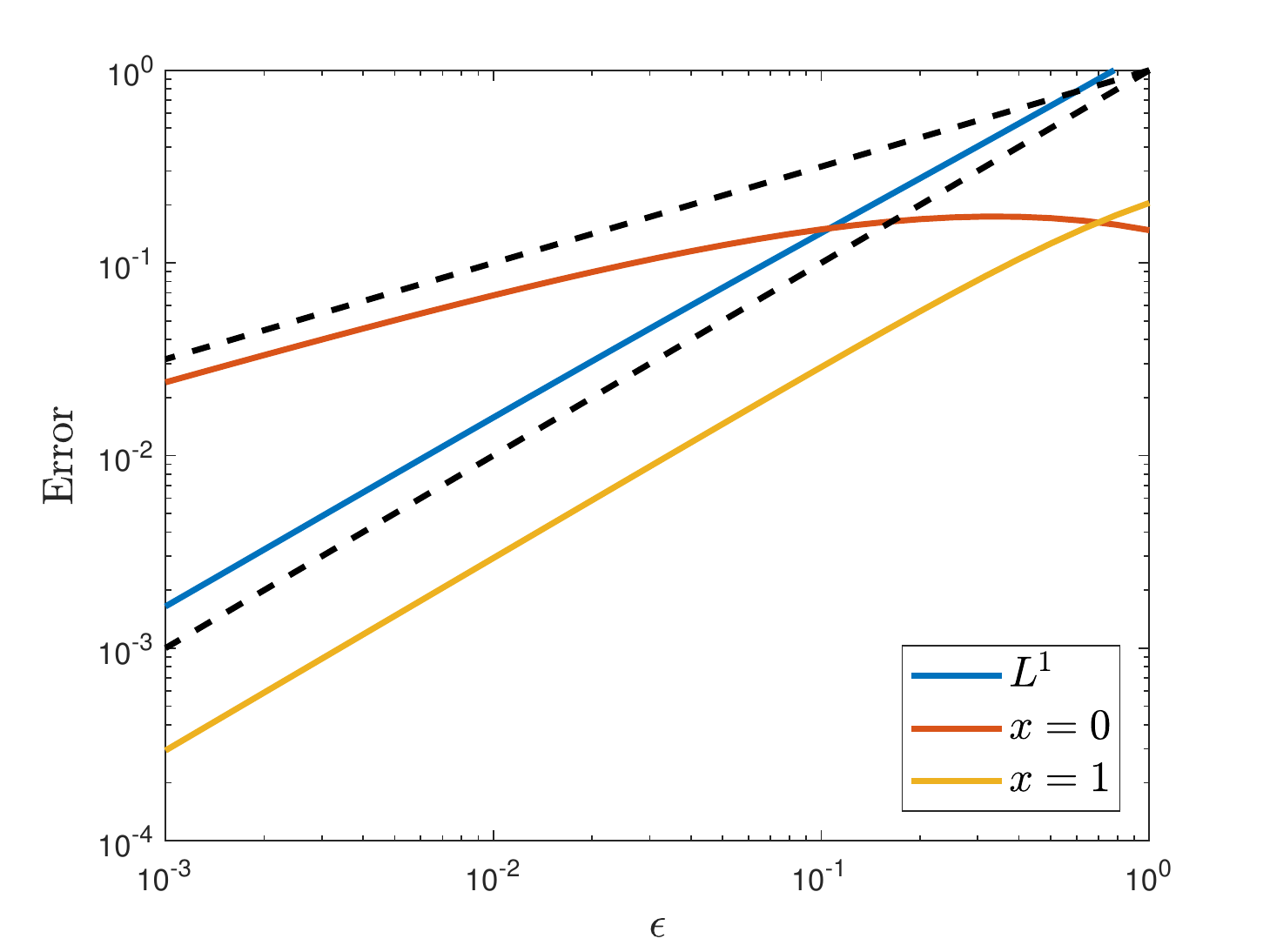}
\includegraphics[width=0.49\textwidth,clip]{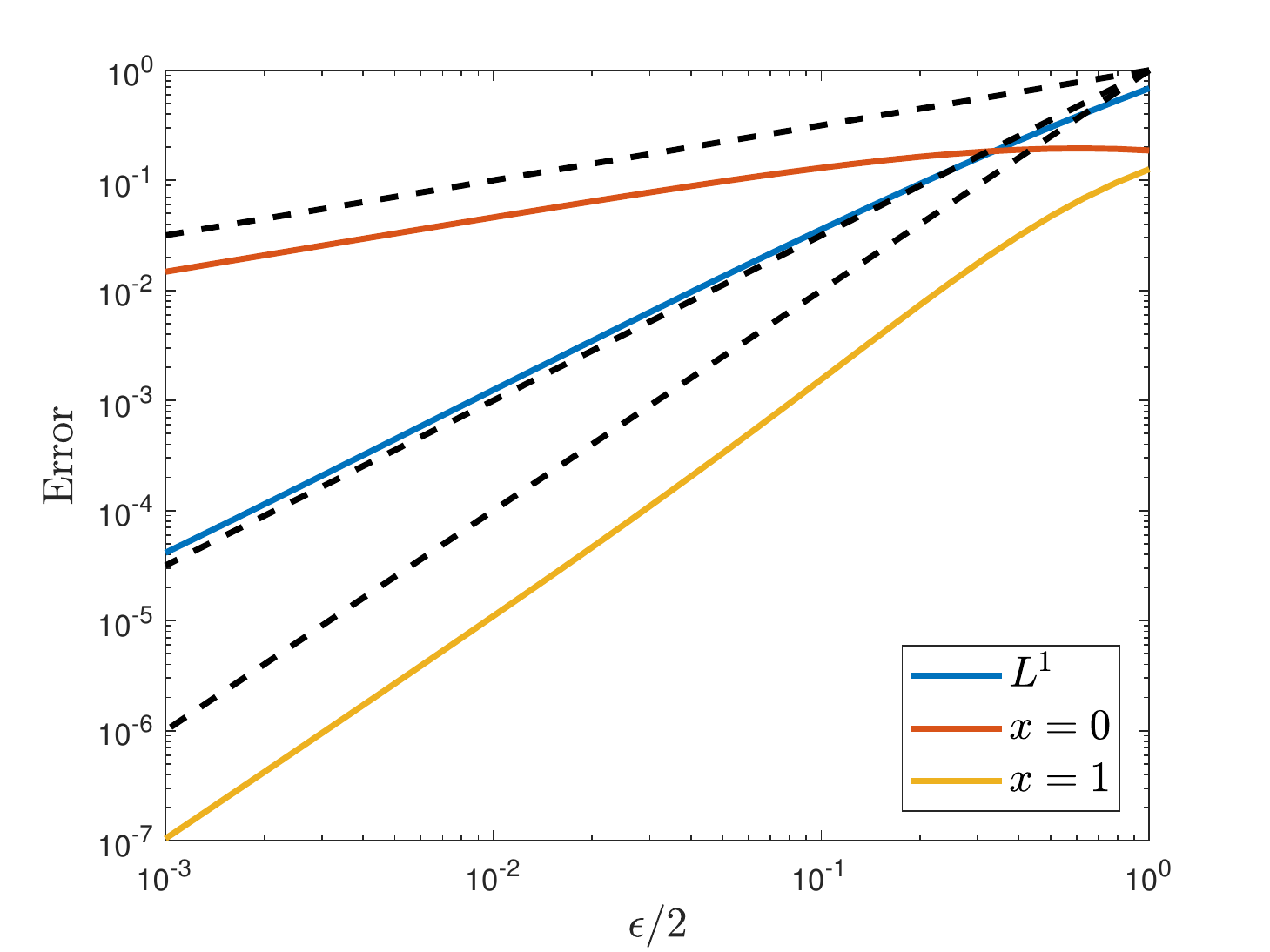} 
\caption{Results for Laguerre polynomials with $\alpha=0.5$. Left: Convergence in $L^1$ and at the points $0$ and $1$. The rates $O(\epsilon)$ and $O(\epsilon^{0.5})$ are also shown as dashed lines. Right: Convergence with Richardson extrapolation. The rates $O(\epsilon^2)$, $O(\epsilon^{1.5})$ and $O(\epsilon^{0.5})$ are also shown.}
\label{Lag1}
\end{figure}

Figure \ref{Jacobi1} (left) shows a typical error plot for $\alpha=0.7$ and $\beta=0.3$. We plot both the $L^1$ error (computed using a large number of discrete points), and the pointwise errors at $-1,0$ and $1$. The procedure of Theorem \ref{res_est1} is used to determine adaptively how large our (rectangular) matrix truncations should be for a given $\epsilon$. We see that both the $L^1$ error and error at $0$ appear to decrease as $O(\epsilon)$,\footnote{This can be proven, though there is an additional $\log(\epsilon^{-1})$ factor. We shall omit such terms in the ensuing discussion.} whereas the errors at $-1$ and $+1$ decrease as $O(\epsilon^{\beta})$ and $O(\epsilon^{\alpha})$ respectively, shown in the plot. This suggests using Richardson extrapolation to accelerate convergence \cite{richardson1927viii}. This is shown in the right of Figure \ref{Jacobi1}, where the extrapolation was computed at distances $\epsilon$ and $\epsilon/2$ from the real axis. Now the error at $0$ decreases as $O(\epsilon^{2})$, whereas the $L^1$ error appears to decrease as $O(\epsilon^{1.3})$. We found similar results for different $\alpha$ and $\beta$. In general, interior points decrease at the rate $O(\epsilon)$ and then $O(\epsilon^{2})$ after extrapolation. The left end point error decreases as $O(\epsilon^{\min\{1,\beta\}})$ and then $O(\epsilon^{\min\{2,\beta\}})$ after extrapolation. The right end point error decreases as $O(\epsilon^{\min\{1,\alpha\}})$ and then $O(\epsilon^{\min\{2,\alpha\}})$ after extrapolation. Finally, the $L^1$ error decreases as $O(\epsilon^{\min\{1,1+\alpha,1+\beta\}})$ and then $O(\epsilon^{\min\{2,1+\alpha,1+\beta\}})$ after extrapolation. These rates for pointwise and $L^1$ errors reflect the local H\"older exponent and integrability of the density and its first derivative respectively, and can be proven using a Taylor series argument for general measures.\footnote{These results can be proven using the theory of high-order rational kernels developed in \cite{colbrook2020computingSM}.} Moreover, we found that increased rates of convergence could be obtained (and again proven) locally near smoother parts of the measure by using further iterates of extrapolation. Note also that we took a uniform value $\epsilon$ over the whole interval. However, $\epsilon$ could just have easily been a function of the position $x$, allowing it to be smaller for points where the resolvent is estimated more accurately for a given matrix size.

To demonstrate the algorithm on unbounded operators, we next consider the class of generalised Laguerre polynomials for $\alpha>-1$ which have
$$
a_k=\sqrt{k(k+\alpha)},\quad b_k=2k+\alpha-1,
$$
and measure on the interval $[0,\infty)$ given by
\begin{equation}
d\mu_J=\frac{x^{\alpha}e^{-x}}{\Gamma(\alpha+1)}dx.
\end{equation}
Results are shown in Figure \ref{Lag1} for $\alpha=0.5$, where we have plotted the (relative) $L^1$ error over the interval $[0,1]$, as well as pointwise errors at $0$ and $1$. Similar conclusions can be drawn as before. Pointwise errors are also shown for this example and the Jacobi operator, but now using the $10$th iterate of Richardson extrapolation, in Figure \ref{Jacobi10}. The errors decay at the expected rates (also shown), with $O(\epsilon^{10})$ convergence at smooth parts of the measure. Near singular points (namely at $x=-0.99$ and $x=0.01$ for the Jacobi and Laguerre cases, respectively), the prefactor in front of the $O(\epsilon^{10})$ term is larger, and smaller $\epsilon$ is needed before the expected rate kicks in.

\begin{figure}
\centering
\includegraphics[width=0.49\textwidth,clip]{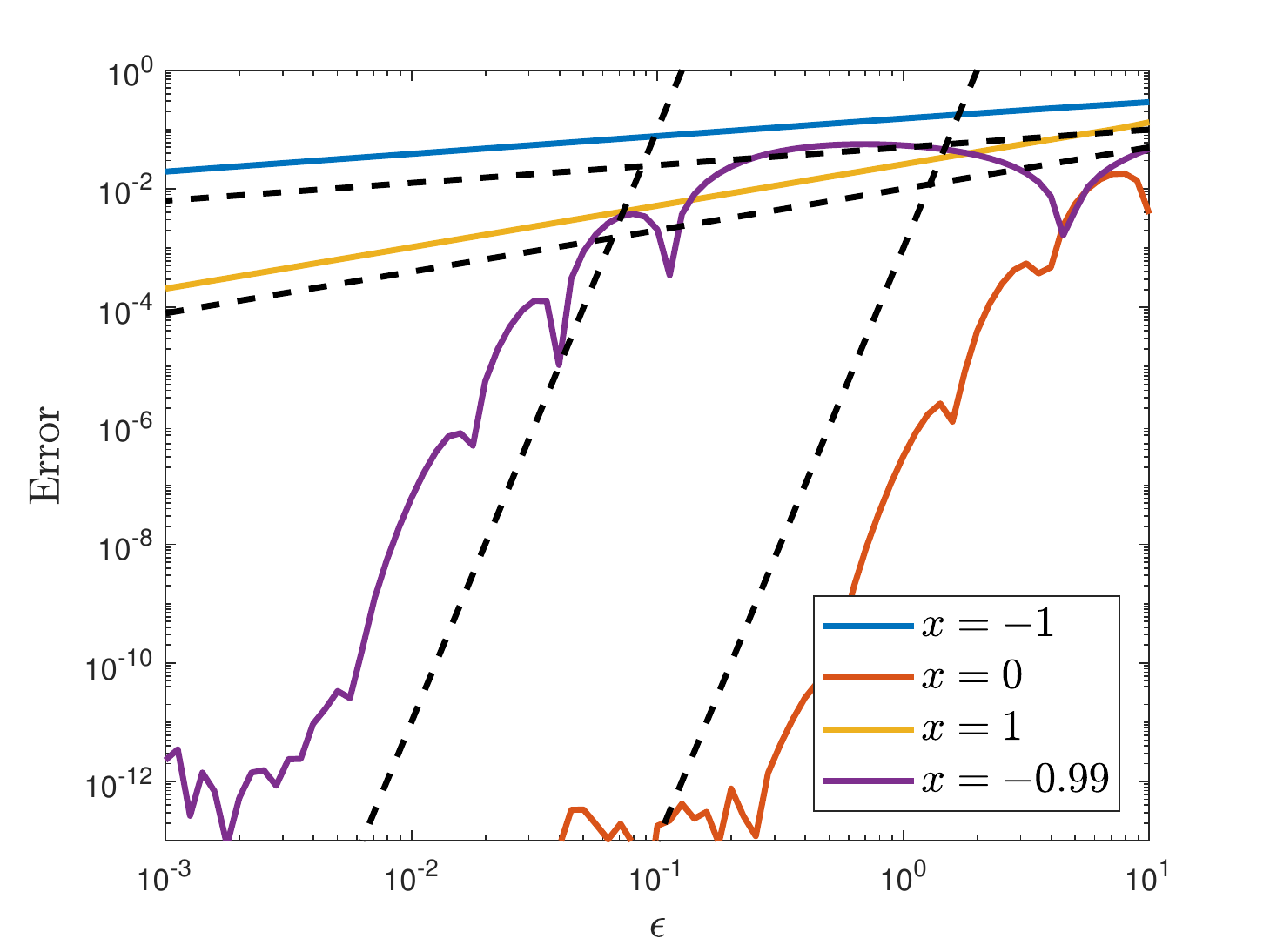}
\includegraphics[width=0.49\textwidth,clip]{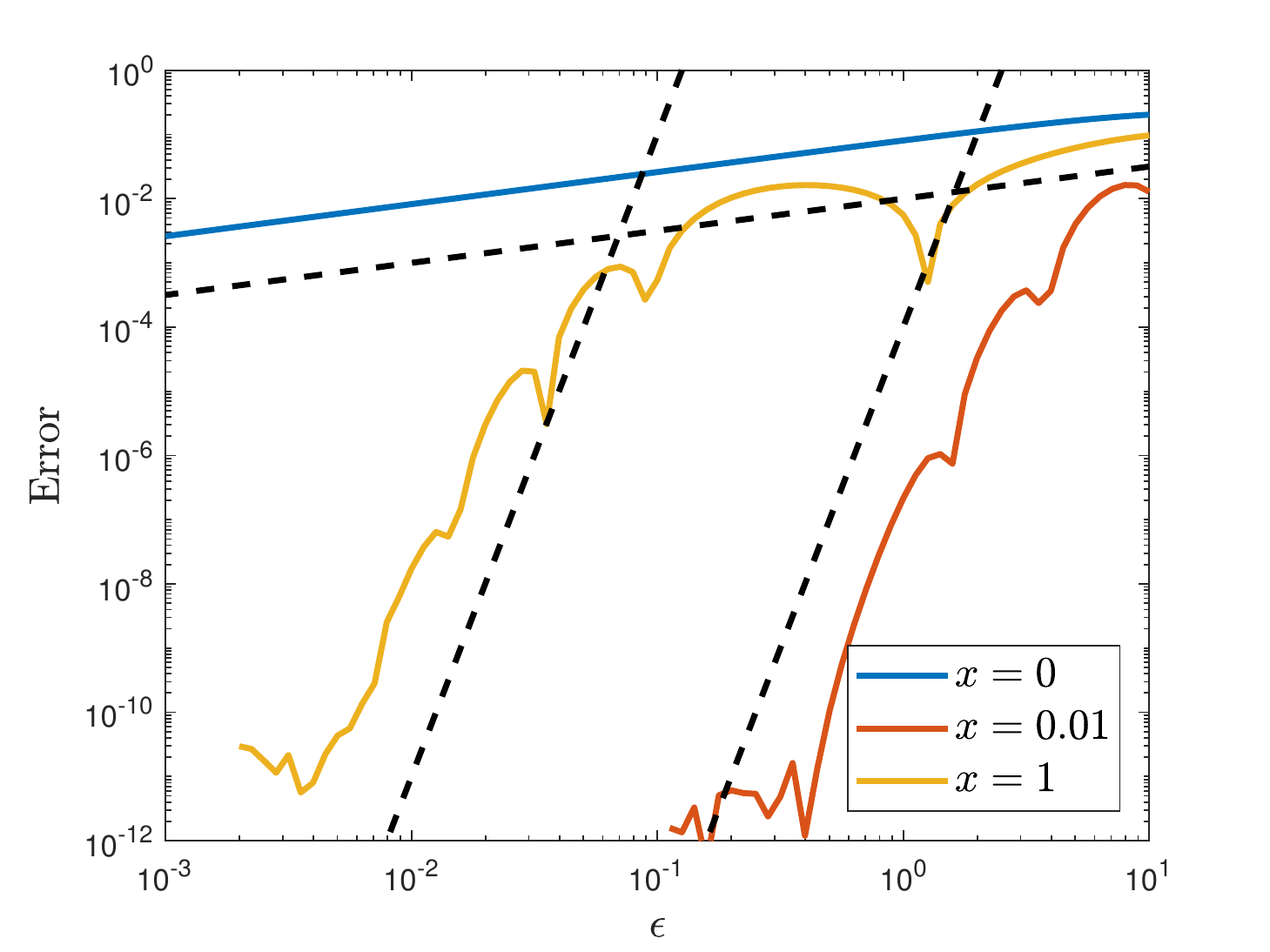}
\caption{Left: Pointwise errors for the Jacobi example ($\alpha=0.7$ and $\beta=0.3$) and ($10$th) iterated Richardson extrapolation. Right: Same but for the Laguerre example ($\alpha=0.5$). The dotted lines show the expected rates of convergence, with $O(\epsilon^{10})$ convergence at smooth parts of the measure.}
\label{Jacobi10}
\end{figure}

\begin{figure}
\centering
\includegraphics[width=0.49\textwidth,clip]{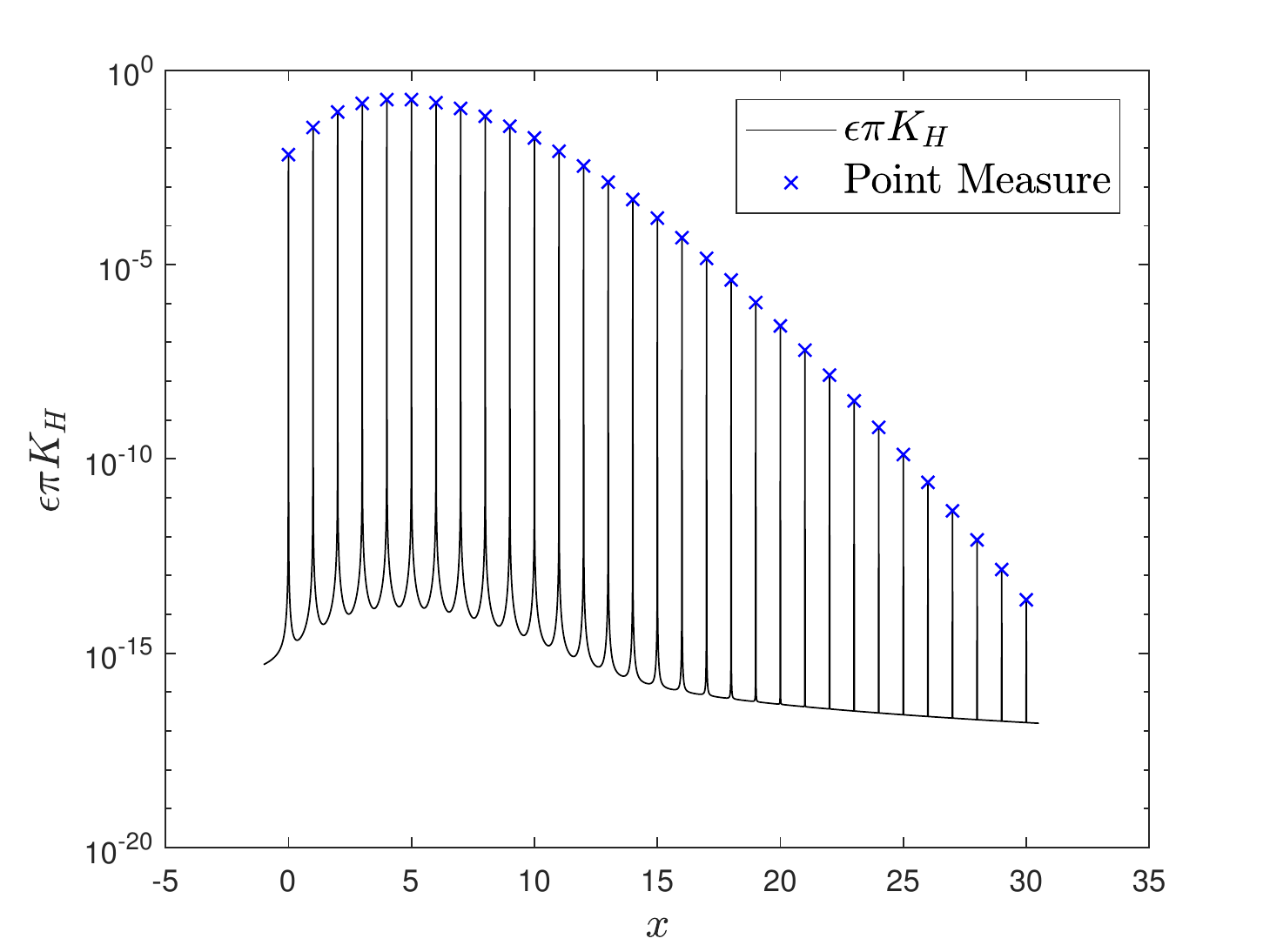}
\includegraphics[width=0.49\textwidth,clip]{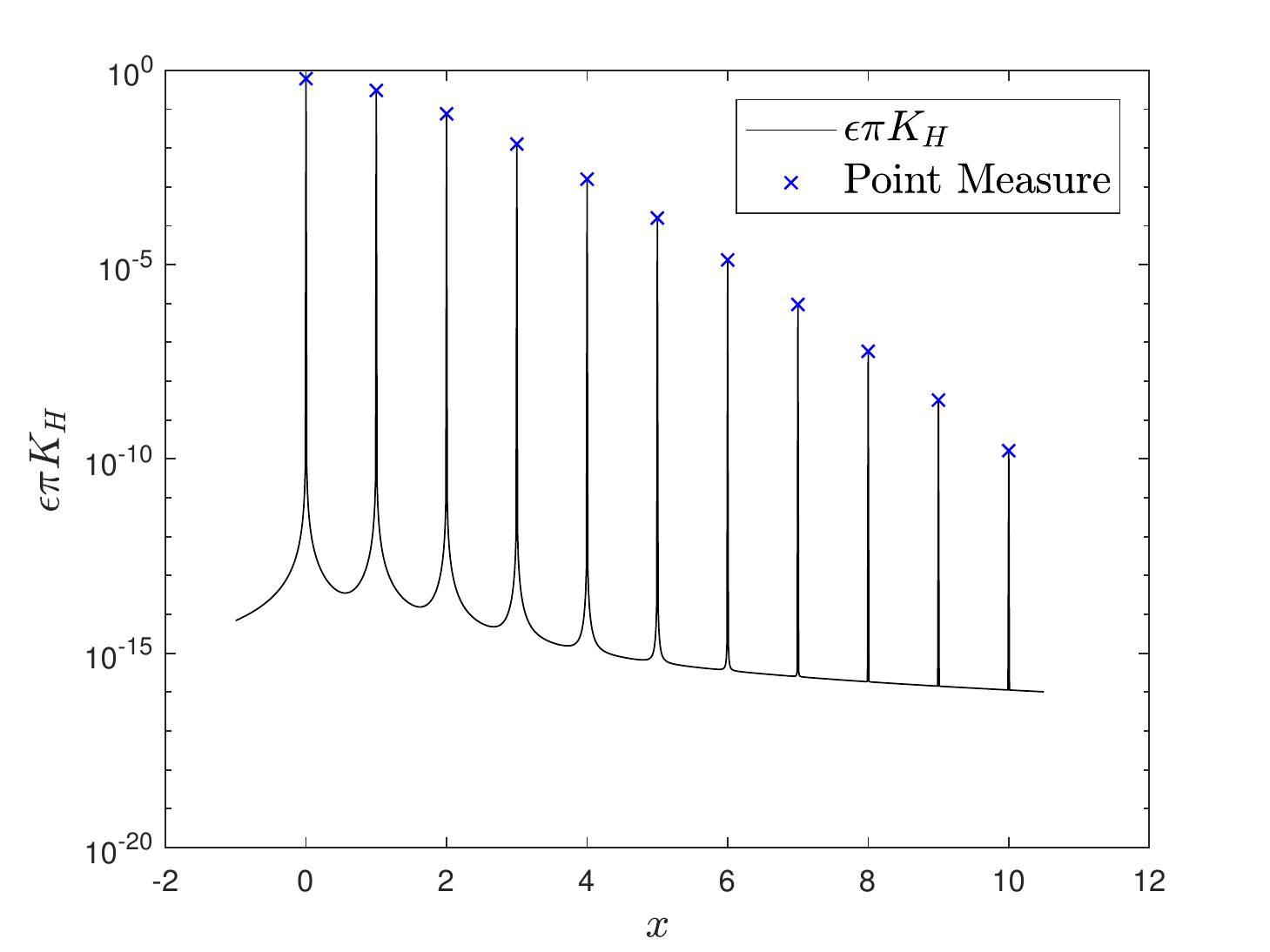} 
\caption{Left: Computation of $\epsilon\pi \langle K_H(x+i\epsilon;T,e_1),e_1\rangle$ (denoted $\epsilon\pi \langle K_H$) for $\epsilon=10^{-7}$ and $\alpha=5$. Right: Same but for $\alpha=0.5$. The blues crosses represent the weights of the atoms of the measure, corresponding to projections onto eigenspaces.}
\label{Charl1}
\end{figure}

Finally, we demonstrate the computation of measures for a Jacobi operator with non-empty discrete spectrum. The Charlier polynomials are generated by
$$
a_k=\sqrt{\alpha k},\quad b_k=k+\alpha-1,
$$
for $\alpha>0$, and have measure
\begin{equation}
d\mu_J=\exp(-\alpha)\sum_{m=0}^\infty\frac{\alpha^m}{m!}\delta_{m},
\end{equation}
where $\delta_m$ denotes a Dirac measure located at the point $m$. Figure \ref{Charl1} shows plots of $\epsilon\pi \langle K_H(x+i\epsilon;T,e_1),e_1\rangle$ for $\epsilon=10^{-7}$, computed using an $(n+1)\times n$ matrix with $n=1000$. The peaks clearly coincide with the atoms of the measure. The difference between the peak values and the weight $\exp(-\alpha)\alpha^m/m!$ was of the order $10^{-13}$ for both examples. This demonstrates that an effective way to compute eigenvalues (particularly the challenging case of those in gaps of the essential spectrum, where spectral pollution occurs, or even those embedded in the essential spectrum) and projections onto eigenspaces may be to find local maxima or spikes of $\epsilon\pi \langle K_H(x+i\epsilon;T,e_1),e_1\rangle$. Such a routine would only require the solution of shifted linear systems (the resolvent), without diagonalisation, and could be executed rapidly in parallel.

\subsection{A global collocation approach}
\label{JACOBI_SEC_2}

Typically, the size of the linear system needed to approximate the resolvent accurately (Theorem \ref{res_est1}) grows as $\epsilon\downarrow 0$ and we approach the spectrum. It is therefore beneficial to increase the rate of convergence of approximating the measures as $\epsilon\downarrow 0$. One local (in terms of $x\in\mathbb{R}$) approach, Richardson extrapolation, was used in \S\ref{JACOBI_SEC_1}. Here we briefly outline a different, more global approach. 

Suppose that we know the spectral measure of an operator $T$ has support included in $I\subset\mathbb{R}$ and is absolutely continuous. Alternatively, we may analytically know, or be able to estimate, the singular part of the measure and subtract this from what follows. In this case, a natural way to approximate the Radon--Nikodym derivative is through a formal basis expansion
$$
\rho_{x,y}^T(\lambda)=\sum_{m=1}^\infty a_m\phi_m(\lambda),
$$
where $\phi_m$ are functions with support $I$ whose Cauchy's transforms are easy to compute. To approximate the coefficients $a_m$, we collocate in the complex plane as follows. Let $\mathcal{C}$ be a finite collection of complex points in the upper half-plane and truncate the approximation of $\rho_{x,y}^T$ to $M$ terms. To generate a linear system for $\{a_m\}_{m=1}^M$, we evaluate the Cauchy transform at points $z\in\mathcal{C}$. The Cauchy transform satisfies
$$
\int_{\mathbb{R}}\frac{\rho_{x,y}^T(\lambda)}{\lambda-z}d\lambda=\langle R(z,T)x,y\rangle,
$$
which can be computed with error control using the results of \S \ref{approx_res_sect}. Call this approximation $b_{x,y}(z)$ and define
$$
\widehat \phi_m(z)=\int_{I}\frac{\phi_m(\lambda)}{\lambda-z}d\lambda.
$$
Then for each $z\in\mathcal{C}$, an approximate linear relation can be written as
$$
\sum_{m=1}^M a_m \widehat\phi_m(z)=b_{x,y}(z).
$$
Evaluating this relation at $\geq M$ points in $\mathcal{C}$ gives rise to a linear system, which can be inverted in the least-squares sense for the approximation of the coefficients $\{a_m\}_{m=1}^M$. If $x=y$, and the basis functions are real, then the coefficients are real. Hence, in this case, taking real and imaginary parts of the linear system gives further linear relations without having to compute further resolvents.

If our basis functions satisfy recursion relations of the form
$$
\phi_{m+1}(\lambda)=\alpha_m \lambda\phi_m(\lambda)+\beta_m \phi_m(\lambda)+\gamma_m\phi_{m-1}(\lambda),
$$
then
\begin{align*}
\widehat\phi_{m+1}(z)&=\alpha_m \int_I\frac{\lambda\phi_m(\lambda)}{\lambda-z}d\lambda+\beta_m\widehat\phi_{m}(z)+\gamma_m\widehat\phi_{m-1}(z)\\
&=\alpha_m\int_I\phi_m(\lambda)d\lambda+z\alpha_m \widehat \phi_m(z)+\beta_m\widehat\phi_{m}(z)+\gamma_m\widehat\phi_{m-1}(z).
\end{align*}
Hence, given the values of the integrals
\begin{equation}
\label{integrals_we_need}
\int_I\phi_m(\lambda)d\lambda,\quad \widehat\phi_{1}(z),
\end{equation}
we can compute $\widehat \phi_m(z)$ for all $m\in\mathbb{N}$. The integrals in (\ref{integrals_we_need}) have simple forms for all the bases used in this paper. This method of computation of $\widehat\phi_m$ is fast, meaning that the most expensive part of the collocation method is the computation of the right hand side of the linear system, that is, computing the resolvent.

\begin{figure}
\centering
\includegraphics[width=0.49\textwidth,clip]{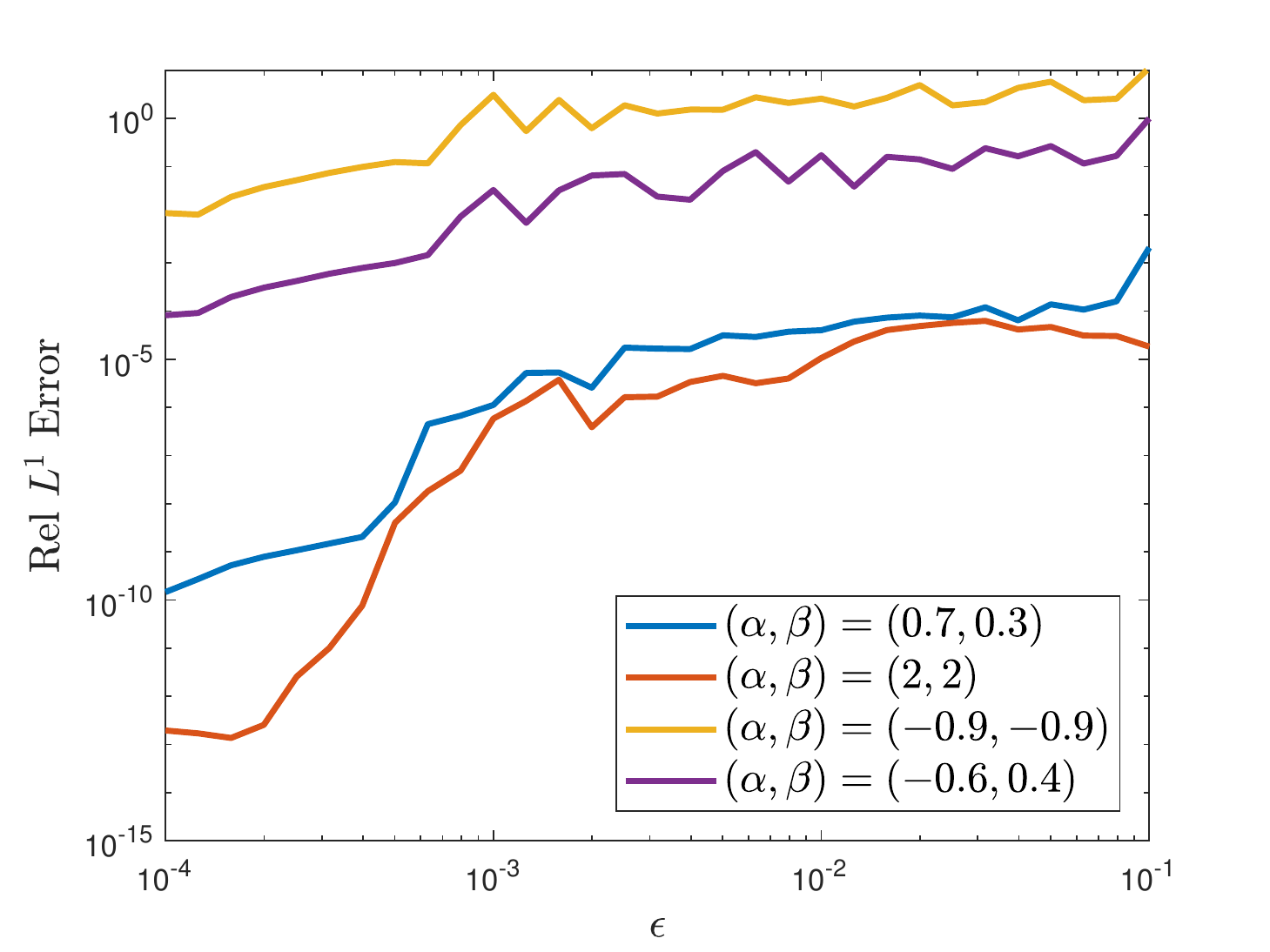}
\includegraphics[width=0.49\textwidth,clip]{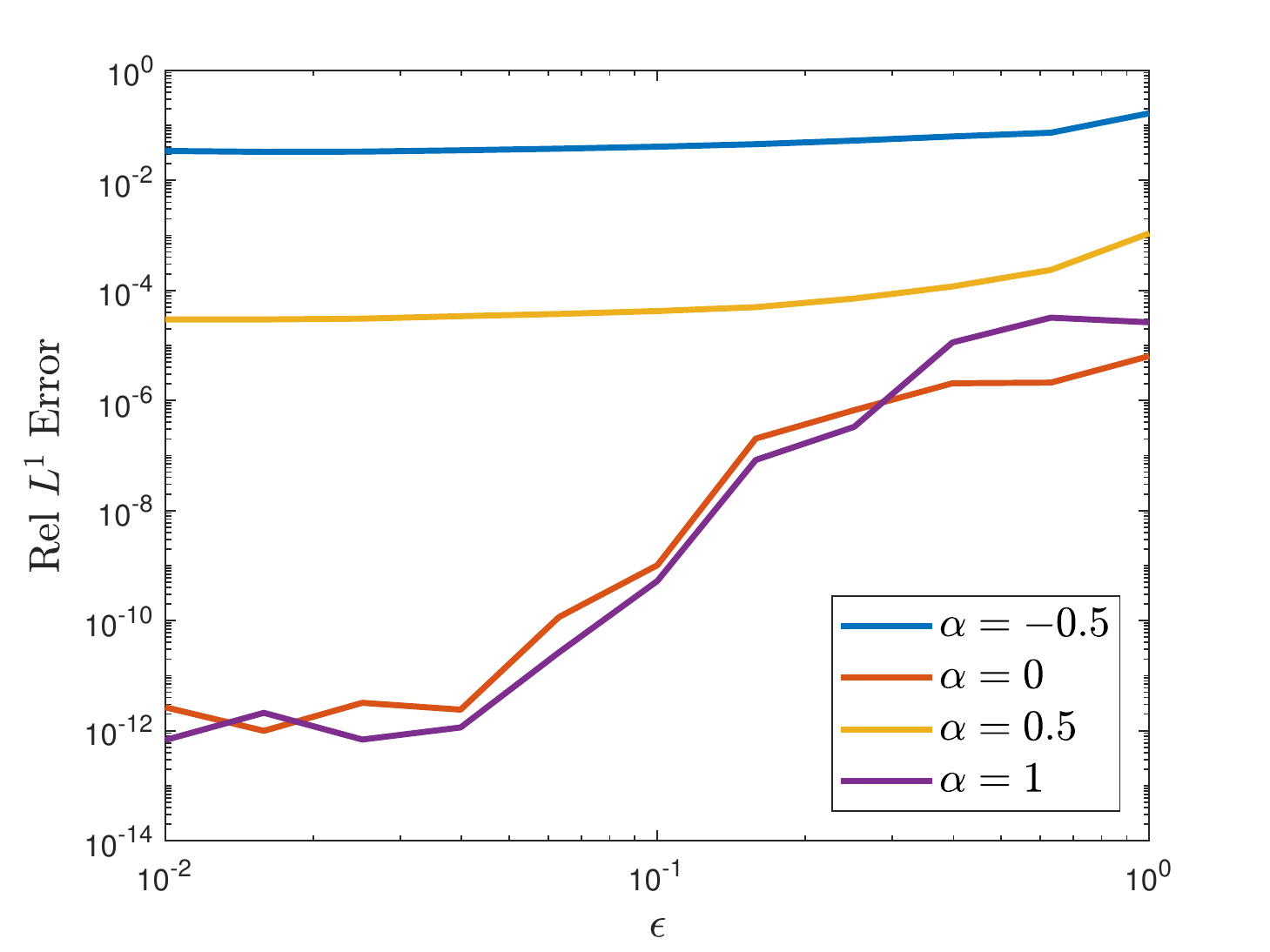}  
\caption{Left: Convergence of collocation method using Chebyshev polynomials. Right: Convergence of collocation method using Laguerre functions on $[0,\infty)$.}
\label{Jacobi2}
\end{figure}

Figure \ref{Jacobi2} (left) revisits the Jacobi polynomials example in \S \ref{JACOBI_SEC_1} and shows the collocation method in the case of using Chebyshev polynomials as the basis functions $\phi_m$ (taking the first $1500$). For collocation points, we took $M$ Chebyshev nodes offset by an additional $\epsilon i$ to lie just above the interval $[-1,1]$ in the complex plane. Note that the rate of convergence is faster than $O(\epsilon^{\min\{2,1+\alpha,1+\beta\}})$ achieved with the methods of \S \ref{applic_sec} after extrapolation. There are at least two prices to pay, though. First, there is no current proof that collocation converges and, second, global regularity of the measure is needed. At the very least, we need to be able to subtract off the singular part of the measure. In practice, even if the measure is absolutely continuous, a large number of basis functions may be needed to capture the Radon--Nikodym derivative. Examples are given in \S \ref{CMV_numerics}, where collocation performs worse than the techniques of \S \ref{applic_sec} due to either of these regularity conditions. Figure \ref{Jacobi2} (right) revisits the (generalised) Laguerre polynomials example in \S \ref{JACOBI_SEC_1} and shows the collocation method in the case of using Laguerre functions (the polynomials multiplied by the square root of the weight function) as the basis functions with $M=1000$. The collocation points were $\{1^2/M^2,2^2/M^2,...,1\}+\epsilon i$. Again, this method converges with faster rates than that in \S \ref{JACOBI_SEC_1}.

\subsection{CMV matrices and extensions to unitary operators}
\label{CMV_numerics}

We now demonstrate that the algorithms extend to the unitary case through use of the functions $K_D$, namely, the convolution with the Poisson kernel of the unit disk. We will consider the class of CMV matrices (named after Cantero, Moral and Vel\'azquez \cite{cantero2003five}) linked with orthogonal polynomials on the unit circle. A full discussion of this subject is beyond the scope of this paper, and we refer the reader to the monographs of Simon \cite{simon2005orthogonal,simon2005orthogonal2}. However, the background for this example is as follows. Given a probability measure $\mu$ on the unit circle $\mathbb{T}$, whose support is not a finite set, we can define a system of orthogonal polynomials $\{\Phi_n\}_{n=0}^\infty$ by applying the Gram--Schmidt process to $\{1,z,z^2,...\}$. Given a polynomial $Q_n(z)$ of degree $n$, we define the reversed polynomial $Q_n^*(z)$ via $Q_n^*(z)=z^n\overline{Q_n(1/\overline{z})}$. Szeg{\H o}'s recurrence relation \cite{szeg1939orthogonal} is given by
\begin{equation}
\Phi_{n+1}(z)=z\Phi_n(z)-\overline{\alpha_n}\Phi_{n}^*(z),
\end{equation}
where the $\alpha_n$ are known as Verblunsky coefficients \cite{verblunsky1936positive} and satisfy $\left|\alpha_j\right|<1$. Verblunsky's theorem \cite{verblunsky1935positive} sets up a one-to-one correspondence between $\mu$ and the coefficients $\{\alpha_j\}_{j=0}^\infty$. Define also
$$
\rho_j=\sqrt{1-\left|\alpha_j\right|^2}>0.
$$
The CMV matrix associated with $\{\alpha_j\}_{j=0}^\infty$ is
\begin{equation}
C=\begin{pmatrix}
\overline{\alpha_0} & \overline{\alpha_1}\rho_0 & \rho_1\rho_0 & 0 & 0 & \cdots\\
\rho_0 & -\overline{\alpha_1}\alpha_0 & -\rho_1\alpha_0 & 0 & 0 & \cdots\\
0 & \overline{\alpha_2}\rho_1 & -\overline{\alpha_2}\alpha_1 & \overline{\alpha_3}\rho_2 & \rho_3\rho_2 & \cdots\\
0 & \rho_2\rho_1 & -\rho_2\alpha_1 & -\overline{\alpha_3}\alpha_2 & -\rho_3\alpha_2& \cdots\\
0 & 0 & 0 & \overline{\alpha_4}\rho_3 & -\overline{\alpha_4}\alpha_3 & \cdots\\
\cdots &\cdots &\cdots &\cdots &\cdots &\cdots 
\end{pmatrix}.
\end{equation}
This matrix is unitary and banded (and lies in $\Omega_{n\rightarrow n+2,0}^{\mathrm{U}}$). This last property does not hold for the so-called GGT representation \cite{geronimus1944polynomials,gragg1993positive,teplyaev1992pure}, which has infinitely many non-zero entries in each row. The GGT representation uses the basis $\{\Phi_n\}_{n=0}^\infty$, whereas the CMV representation obtains a basis via applying Gram--Schmidt to $\{1,z,z^{-1},z^2,z^{2},...\}$. The key result is that $\mu_C:=\mu^C_{e_1,e_1}$ is precisely the measure $\mu$ on the unit circle. Hence our new algorithms can be considered as a computational tool for the correspondence
$$
\{\alpha_j\}_{j=0}^\infty\rightarrow \mu,
$$
in much the same way as for orthogonal polynomials on the real line in \S \ref{JACOBI_SEC_1}.

\begin{figure}
\centering
\includegraphics[width=0.49\textwidth,clip]{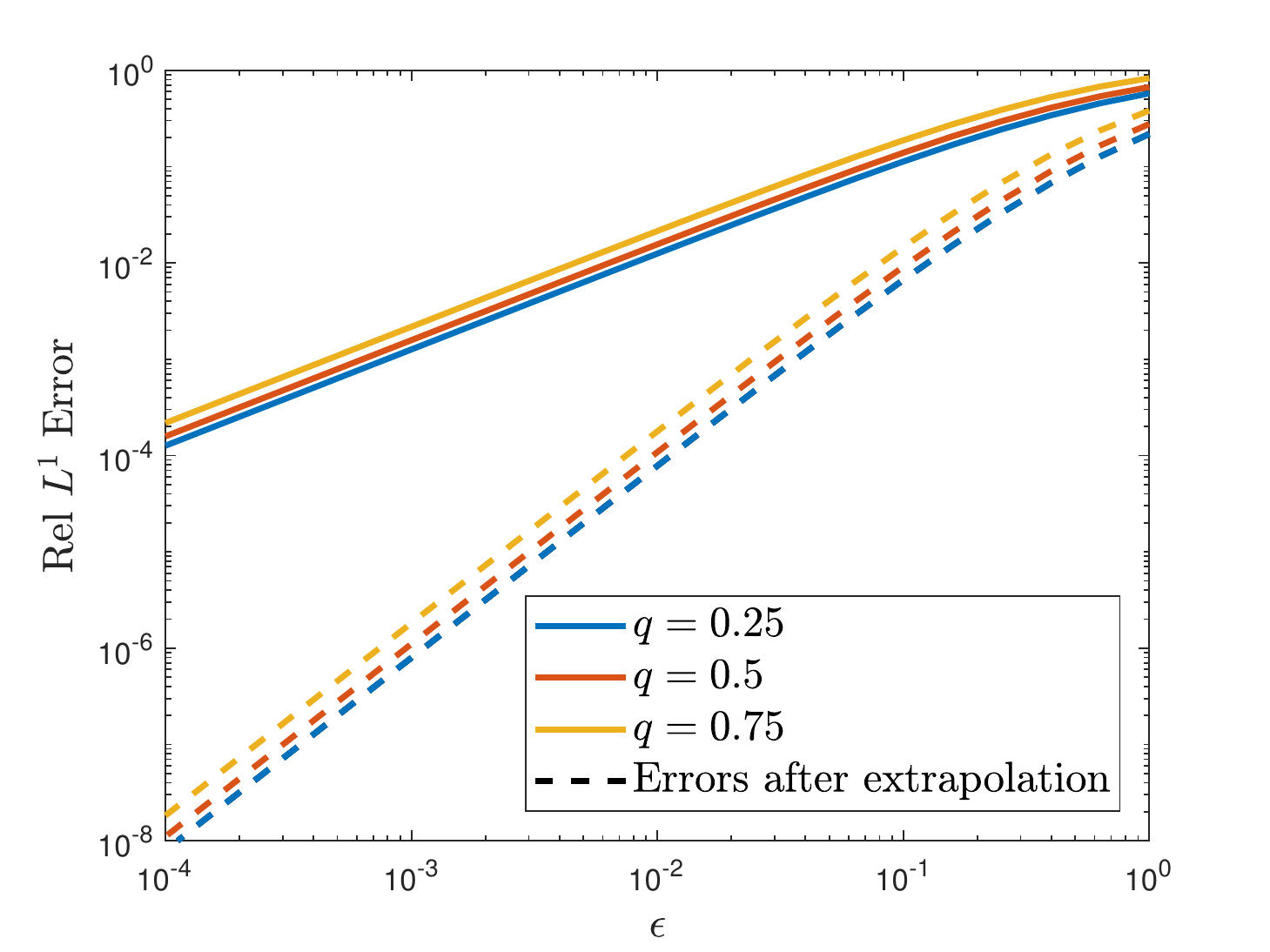}
\includegraphics[width=0.49\textwidth,clip]{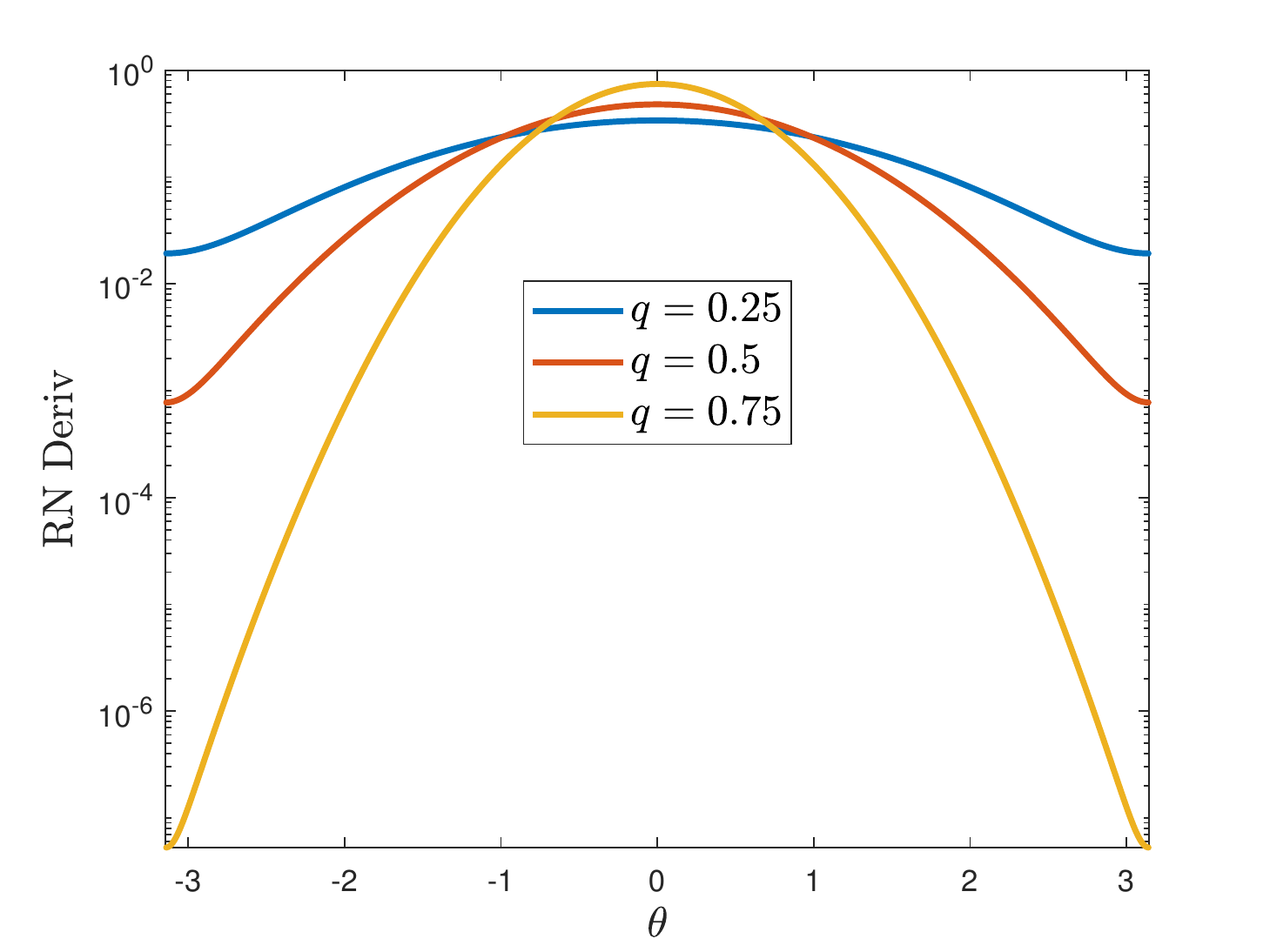}
\caption{Left: Convergence of algorithm for Rogers--Szeg{\H o} polynomials. Right: Corresponding Radon--Nikodym derivatives (densities of the measure).}
\label{CMV1}
\end{figure}

\begin{figure}
\centering
\includegraphics[width=0.49\textwidth,clip]{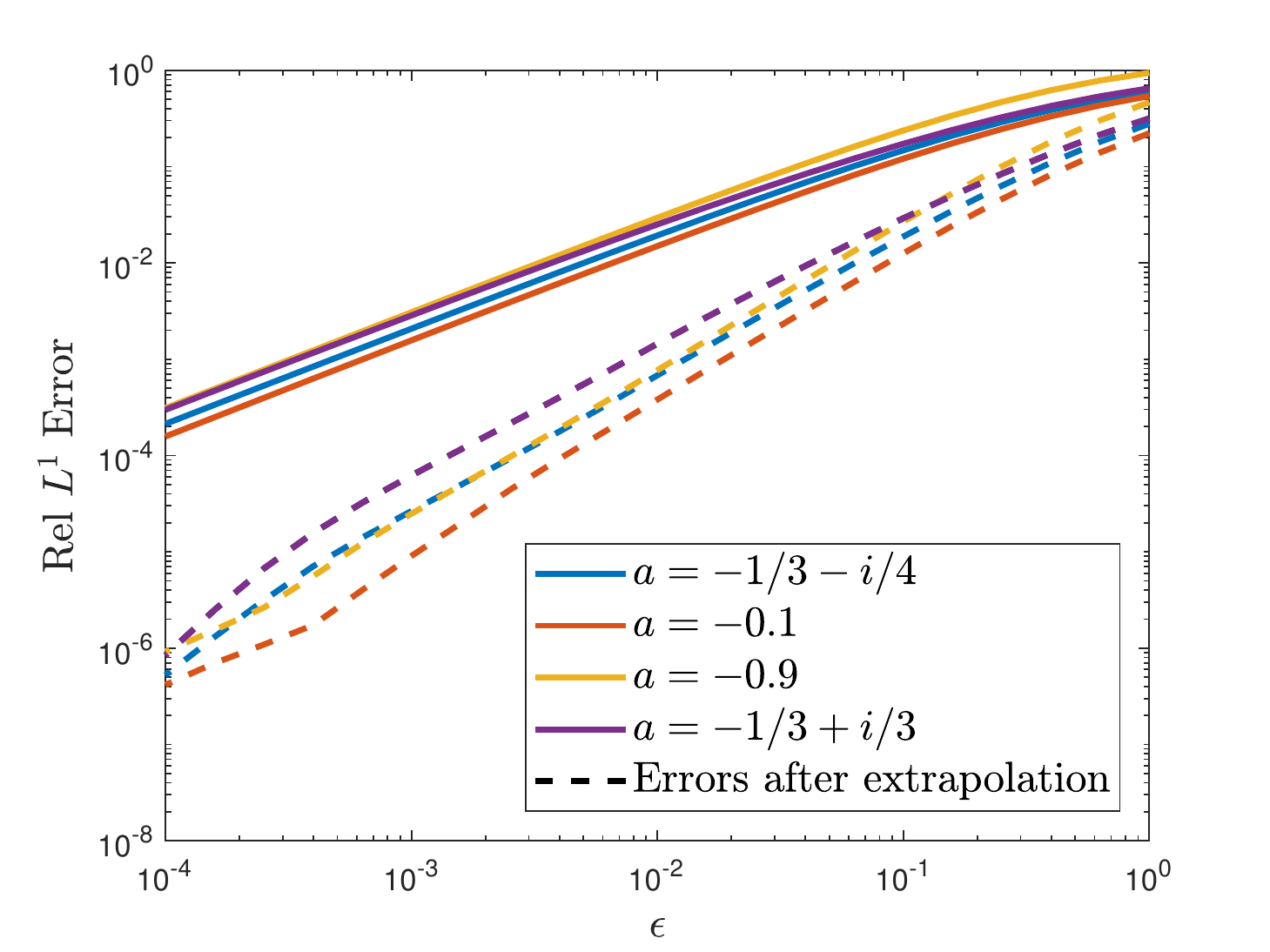}
\includegraphics[width=0.49\textwidth,trim={32mm 92mm 32mm 92mm},clip]{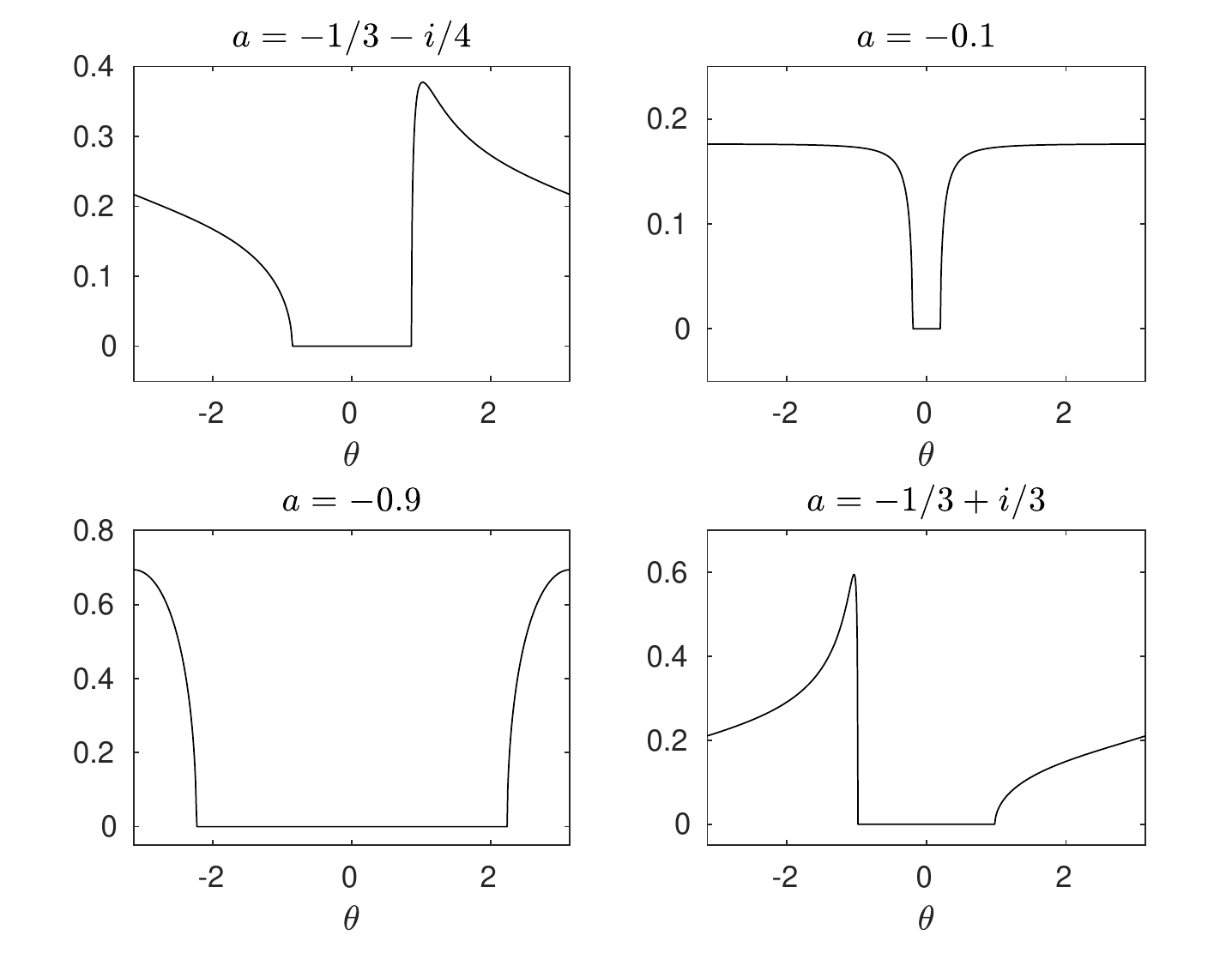}
\caption{Left: Convergence of algorithm for Geronimus polynomials. In this case the algorithm can obtain much more accurate results than collocation owing to the non-smoothness of the density functions. Right: Example density functions for the cases considered.}
\label{ger1}
\end{figure}

\begin{figure}
\centering
\includegraphics[width=0.49\textwidth,clip]{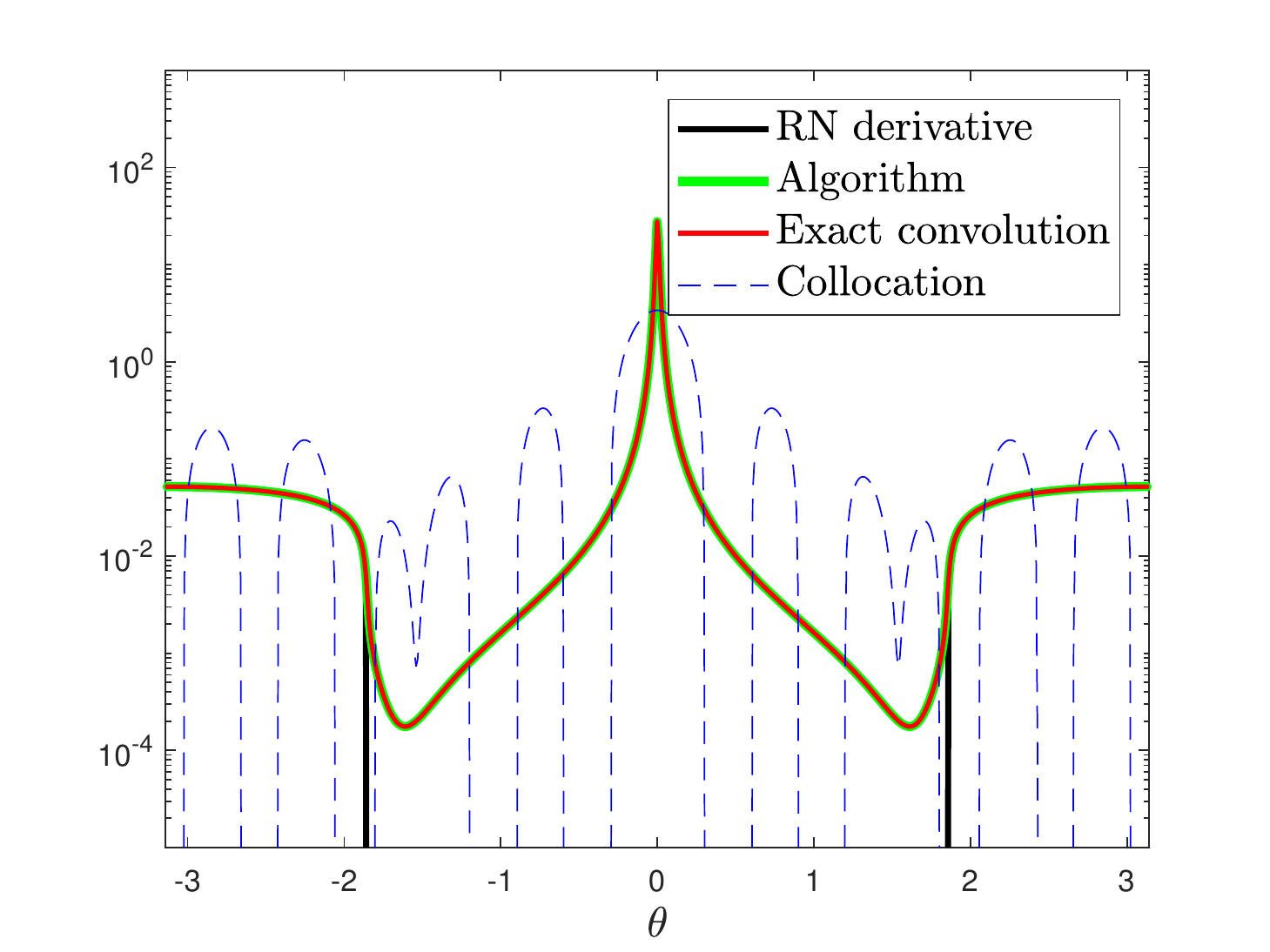}
\includegraphics[width=0.49\textwidth,clip]{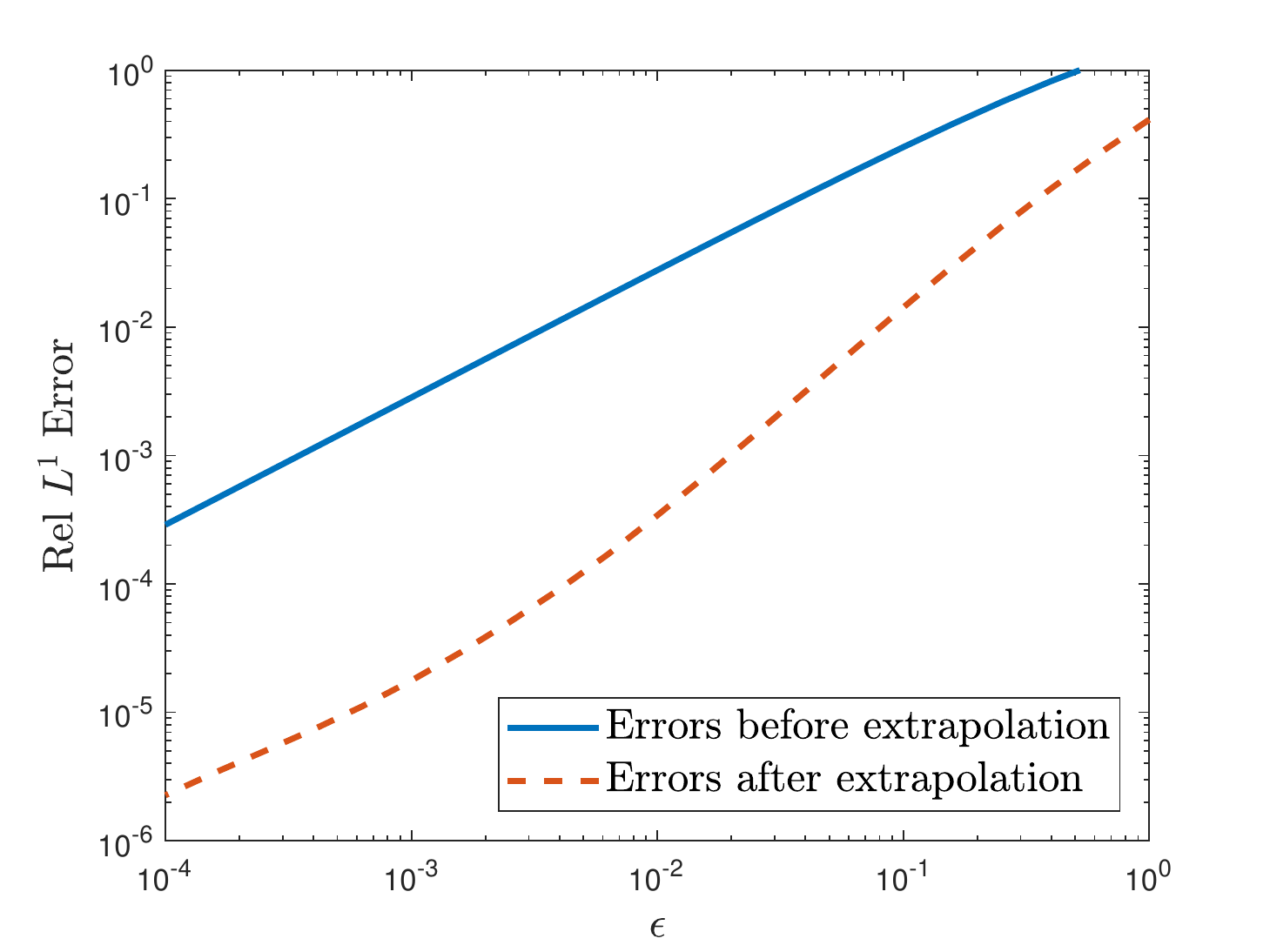}
\caption{Left: The smooth part of the density function (black) for the case of an additional point mass at $\theta=0$. Note that the algorithm's output agrees almost perfectly with the exact convolution of the measure. We have also shown the output of the collocation method, which is unstable in the presence of the point mass. Right: Convergence of the algorithm from Theorem \ref{meas_comp3} over the portion of the spectrum which is purely absolutely continuous.}
\label{ger2}
\end{figure}

The first example we consider are the Rogers--Szeg{\H o} polynomials \cite{rogers1893second} given by
$$
\alpha_j=(-1)^jq^{(j+1)/2},
$$
where $q\in(0,1)$. In this case,
$$
d\mu_C=\frac{1}{\sqrt{2\pi\log(q^{-1})}}\sum_{m\in\mathbb{Z}}\exp\left(-\frac{(\theta-2\pi m)^2}{2\log(q^{-1})}\right)d\theta,
$$
which can be expressed in terms of the theta function. Figure \ref{CMV1} (left) shows the convergence of the new algorithm for various $q$ and we see algebraic convergence as before, with rates $O(\epsilon)$ and $O(\epsilon^2)$ before and after extrapolation respectively (here $\epsilon$ is such that we evaluate the convolutions with the Poisson kernel at radius $r=(1+\epsilon)^{-1}$). Radon--Nikodym derivatives for larger values of $q$ (shown in the right of Figure \ref{CMV1}) have larger derivatives and hence larger pre-factor in front of these rates. We can use a similar collocation method as in \S \ref{JACOBI_SEC_2}, using the standard Fourier basis $\{e^{im\theta}\}_{m\in\mathbb{Z}}$. Note that the relevant Cauchy transforms can be computed explicitly using Cauchy's residue theorem. Collocation points inside and outside the unit disk are needed (collocating inside the unit disk causes the Cauchy transforms of the basis functions with negative $m$ to vanish). This achieved machine precision using just $41$ basis functions and collocating at distance $\epsilon=0.1$ from $\mathbb{T}$.

The next example we consider are the Geronimus polynomials \cite{geronimus1966certain}, which have $\alpha_j=a$ with $\left|a\right|<1$. In this case, for $\left|a+1/2\right|\leq 1/2$,
$$
d\mu_C=\chi_{\left|\theta\right|>\theta_a}\frac{\sqrt{\cos^2(\theta_a/2)-\cos^2(\theta/2)}}{2\pi\left|1+a\right|\left|\sin((\theta-b)/2)\right|}d\theta,
$$
where $\theta\in[-\pi,\pi]$, $b=2\mathrm{arg}(1+\overline{a})$ and $\theta_a=2\sin^{-1}(\left|a\right|)$. Figure \ref{ger1} (right) shows some typical examples, and note that these density functions are not smooth. When $\left|a+1/2\right|>1/2$, there is also a singular part of the measure (see \cite{simon2005orthogonal} for the exact formula).

In the cases shown in Figure \ref{ger1}, the collocation method struggles to gain an accuracy beyond $10^{-4}$ owing to discontinuity in the derivative of the Radon--Nikodym derivative and the algorithm based on convolutions is able to gain much more accurate results. Finally, a typical example is shown in Figure \ref{ger2} (left) for $a=0.8$ (there is now a singular part located at $\theta=0$), where we have shown the output of the algorithm and the exact convolution with the Poisson kernel for $r=1/1.01$, as well as the collocation method using $21$ basis functions and collocation points $\{(1\pm0.1)2\pi/21,(1\pm0.1)2\pi\cdot2/21,...,(1\pm0.1)2\pi\}$. Here, we see an exact agreement between the algorithm and convolution. Unsurprisingly in the presence of point spectra, the collocation method is unstable. Consistent with Theorem \ref{meas_comp3}, the algorithm in \S\ref{app1_RN} converges to the density over the portion of the spectrum, which is purely absolutely continuous. This is shown in the right of Figure \ref{ger2}.

\subsection{Fractional diffusion on a 2D quasicrystal}
\label{penrose_numerics}

\begin{figure}
\centering
\includegraphics[width=0.48\textwidth,trim={35mm 98mm 35mm 92mm},clip]{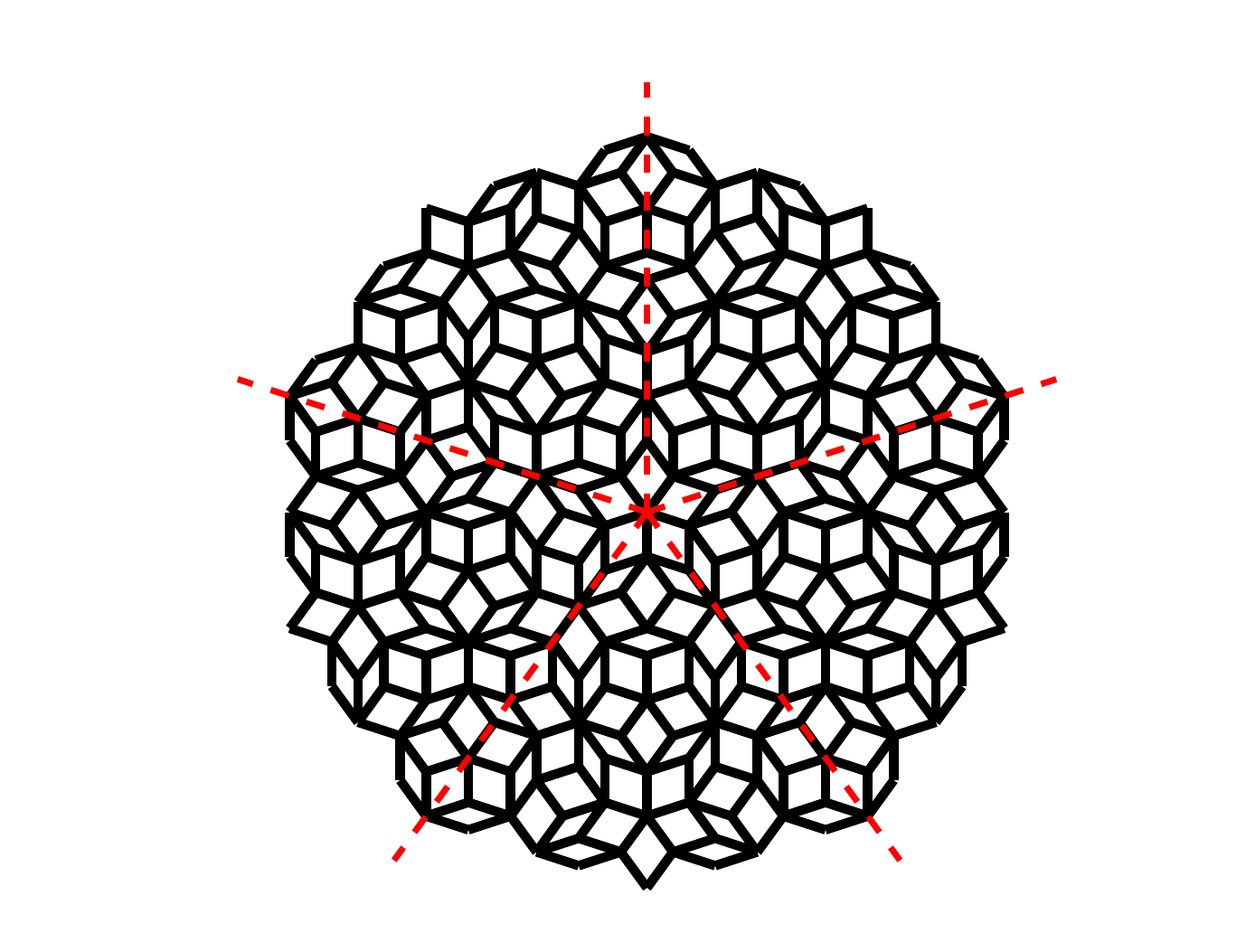}
\includegraphics[width=0.45\textwidth,trim={0mm 0mm 0mm 0mm},clip]{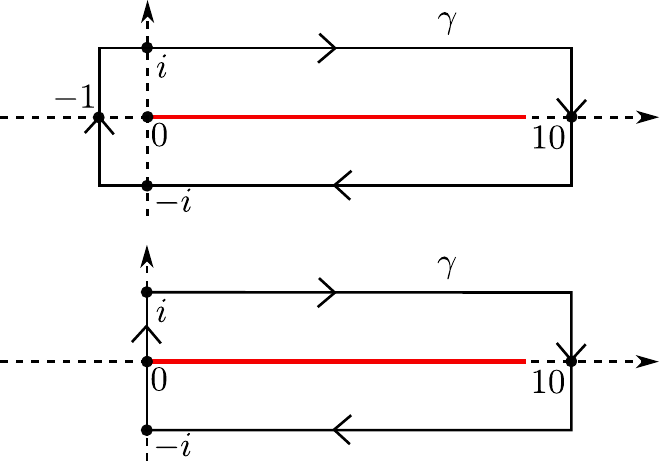}
\caption{Left: Finite portion of Penrose tile showing the fivefold rotational symmetry. We labelled vertices from the centre in a spiral outwards in increasing distance from the origin. Right: Contours used for the fractional diffusion on the Penrose tile ($\alpha\in\mathbb{N}$ top, $\alpha\notin\mathbb{N}$ bottom). The red line represents the interval containing the spectrum of $-H_0$, the branch cut for $z^{\alpha}$ is taken to be $\mathbb{R}_{\leq 0}$.}
\label{pen0}
\end{figure}

\begin{figure}
\centering
\includegraphics[width=0.49\textwidth,clip]{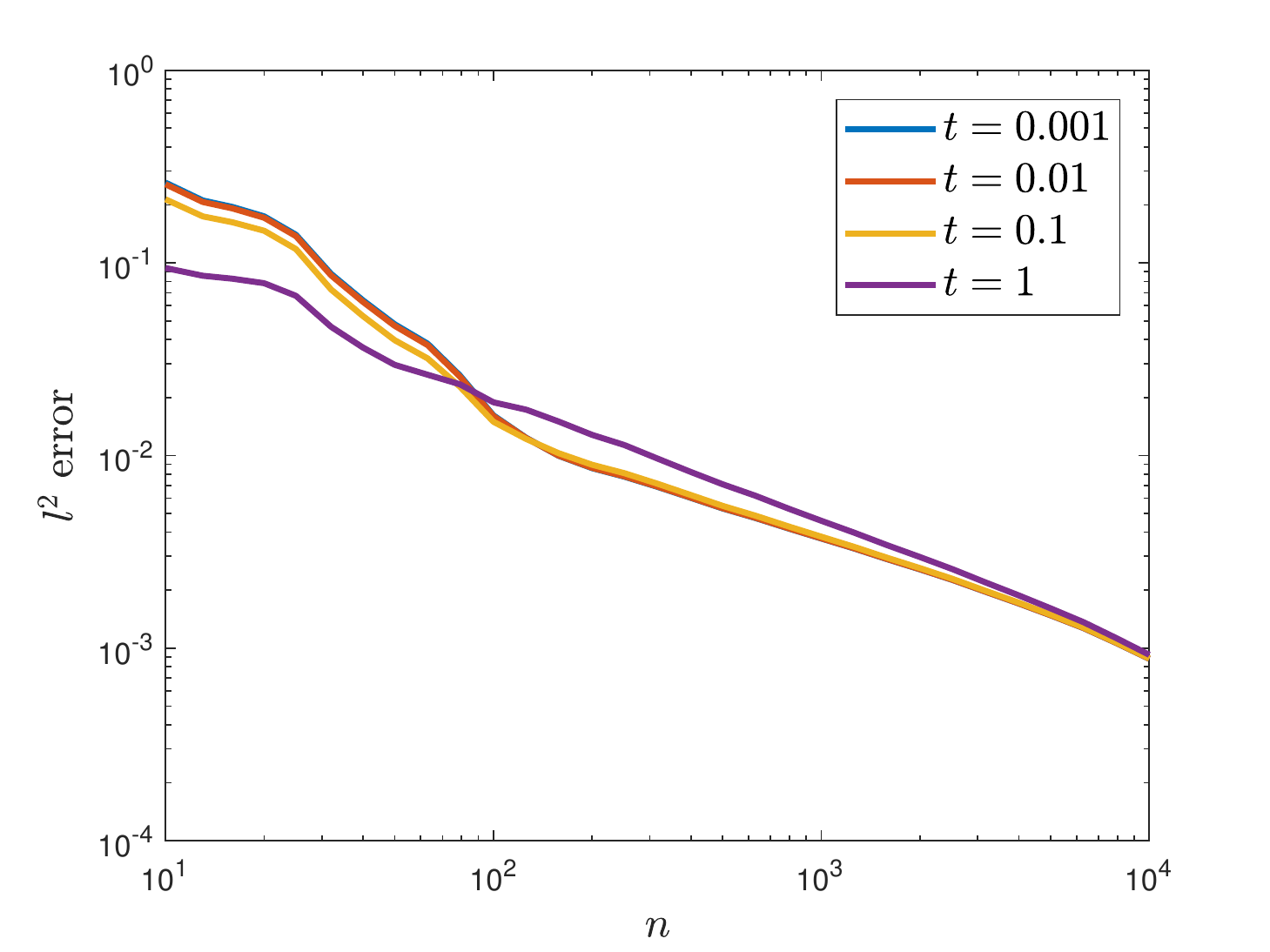}
\includegraphics[width=0.49\textwidth,clip]{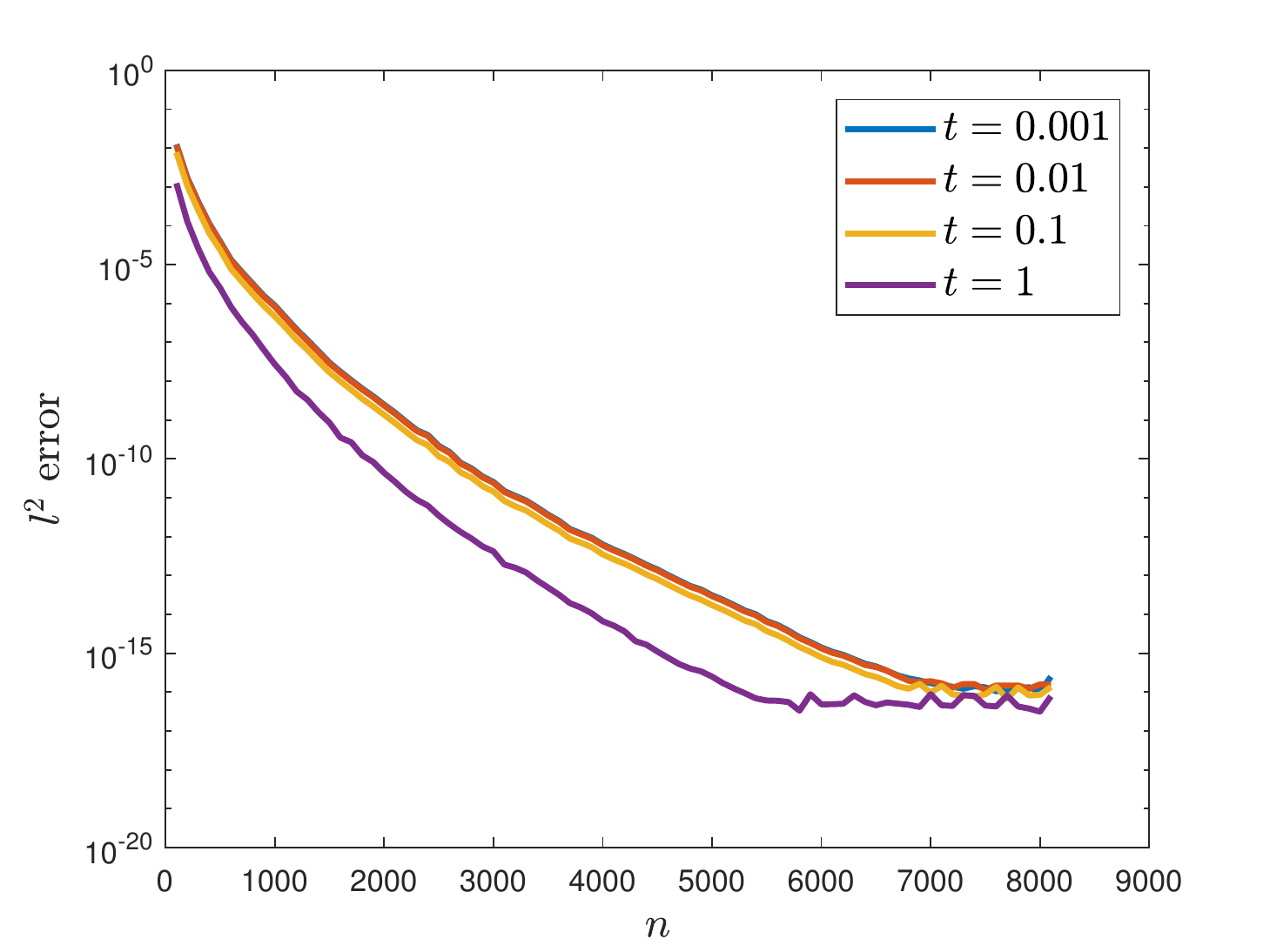} 
\caption{Left: Convergence for $\alpha=1/2$. Right: Convergence for $\alpha=1$. We have plotted the errors as a function of the matrix size (number of matrix columns in the rectangular truncations) used.}
\label{pen1}
\end{figure}

In this example, we demonstrate computation of the functional calculus and consider operators acting on the graph of a Penrose tile - the canonical model of a two-dimensional quasicrystal \cite{vardeny2013optics,takemura2015valence,della2005band} (aperiodic crystals which typically have anomalous spectra/transport properties). Quasicrystals were discovered in 1982 by Shechtman \cite{PhysRevLett.53.1951} who was awarded the 2011 Nobel Prize in Chemistry for his discovery. Since then, quasicrystals have generated considerable interest due to their often exotic physical properties \cite{stadnik2012physical}, with a vast literature on the physics and spectral properties of such aperiodic systems.\footnote{Recently, there has been renewed interest in aperiodic systems connected to graphene in the field of twistronics \cite{chang_2019,lu2019superconductors}.} Unlike the one-dimensional case, little is known about the spectral properties of two-dimensional quasicrystals. A finite portion of the infinite tile is shown in Figure \ref{pen0}, and we consider the natural graph whose vertices are the vertices of the tiling and edges correspond to the edges of the rhombi. Such a graph posses no periodic structure, and it is generically impossible to study its spectral properties analytically with current methods. The free Hamiltonian $H_0$ (Laplacian) is given by 
\begin{equation}\label{H_02}
(H_0\psi)_{i} = \sum_{i \sim j} \left(\psi_j-\psi_i\right),
\end{equation}
with summation over nearest neighbour sites (vertices). The first rigorous computation of the spectrum of $H_0$ (also with additional error control) was performed in \cite{colb1}. We chose a natural ordering of the vertices as in Figure \ref{pen0}, which leads to an operator $H_0$ acting on $l^2(\mathbb{N})$. The local bandwidth grows for this operator (our ordering is asymptotically optimal) and hence computation of powers $H_0^m$ is infeasible for $m\gtrsim 50$, rendering polynomial approximations of the functional calculus intractable. In the above notation, $H_0\in\Omega_{f,0}$ with $f(n)-n=O(\sqrt{n})$, and so this example provides a demonstration of the algorithm for a non-banded matrix. Throughout, we take $u_0=e_1$, though different initial conditions are handled in the same manner.

The ability to compute the functional calculus allows the solution of linear evolution equations. Given $A\in\Omega_{\mathrm{N}}$, a function $F$ (continuous and bounded on $\sigma(A)$) and $u_0\in l^2(\mathbb{N})$, consider the evolution equation
\begin{equation}
\label{evolution_we_can_handle}
\frac{du}{dt}=F(A)u,\quad u_{t=0}=u_0.
\end{equation}
The solution of this equation is
$$
u(t)=\exp(F(A)t)u_0
$$
and can be computed via the algorithm outlined in \S \ref{applic_FCCC}. 

\begin{figure}
\centering
\includegraphics[width=0.28\textwidth,trim={50mm 97mm 62mm 92mm},clip]{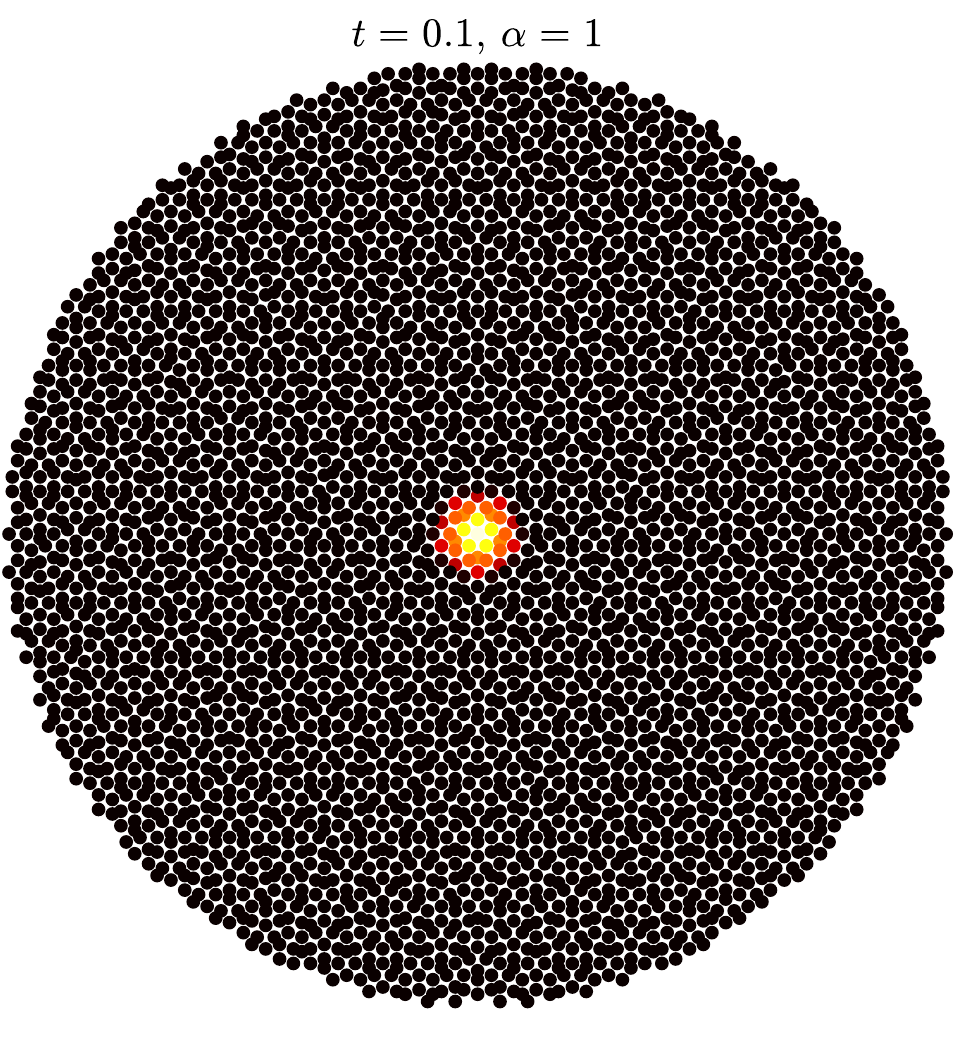}
\includegraphics[width=0.28\textwidth,trim={50mm 97mm 62mm 92mm},clip]{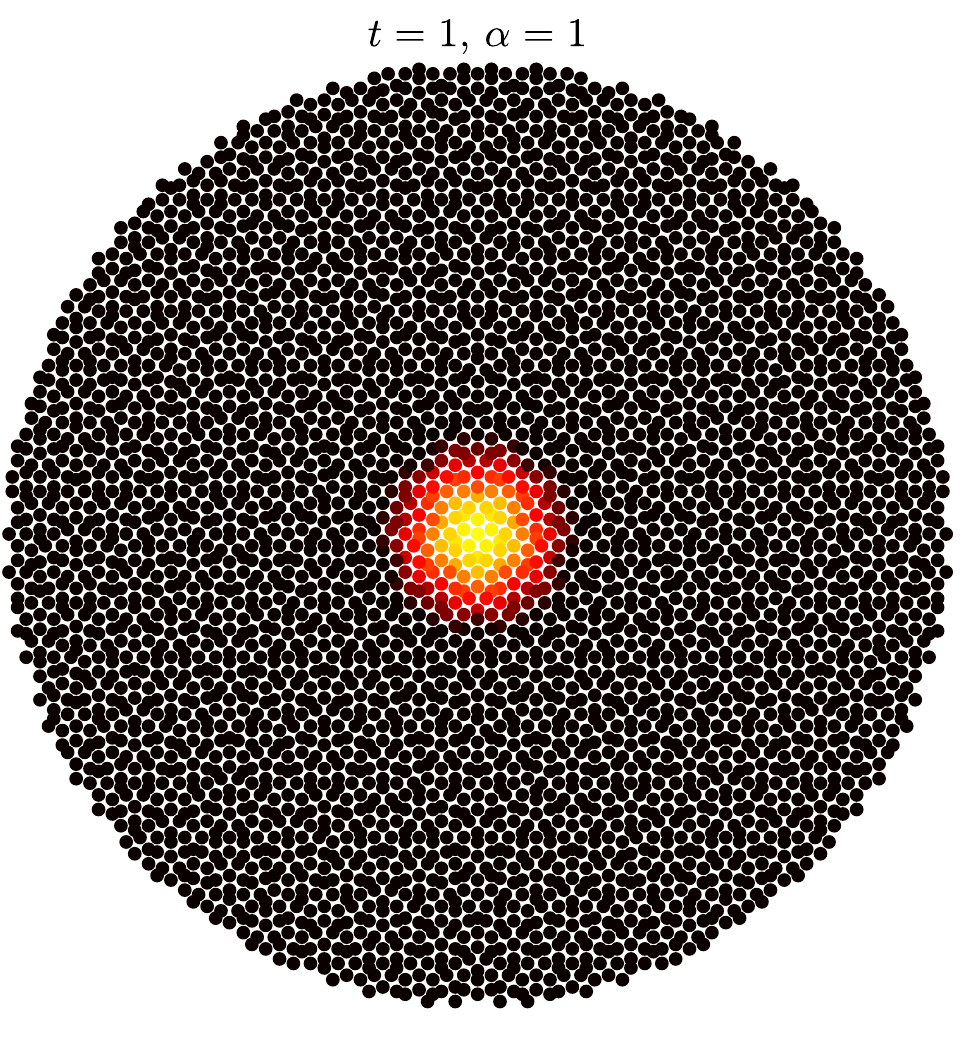}
\includegraphics[width=0.28\textwidth,trim={50mm 97mm 62mm 92mm},clip]{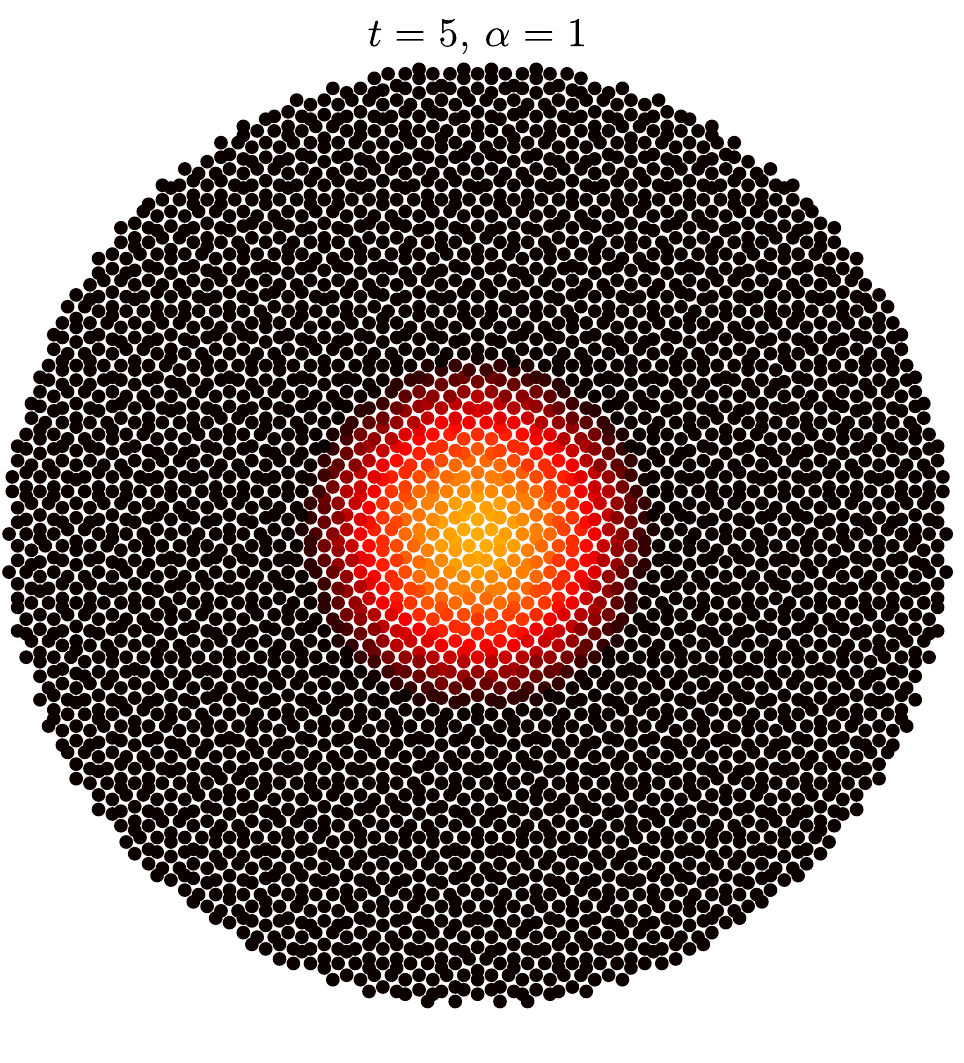}
\includegraphics[width=0.28\textwidth,trim={50mm 97mm 62mm 92mm},clip]{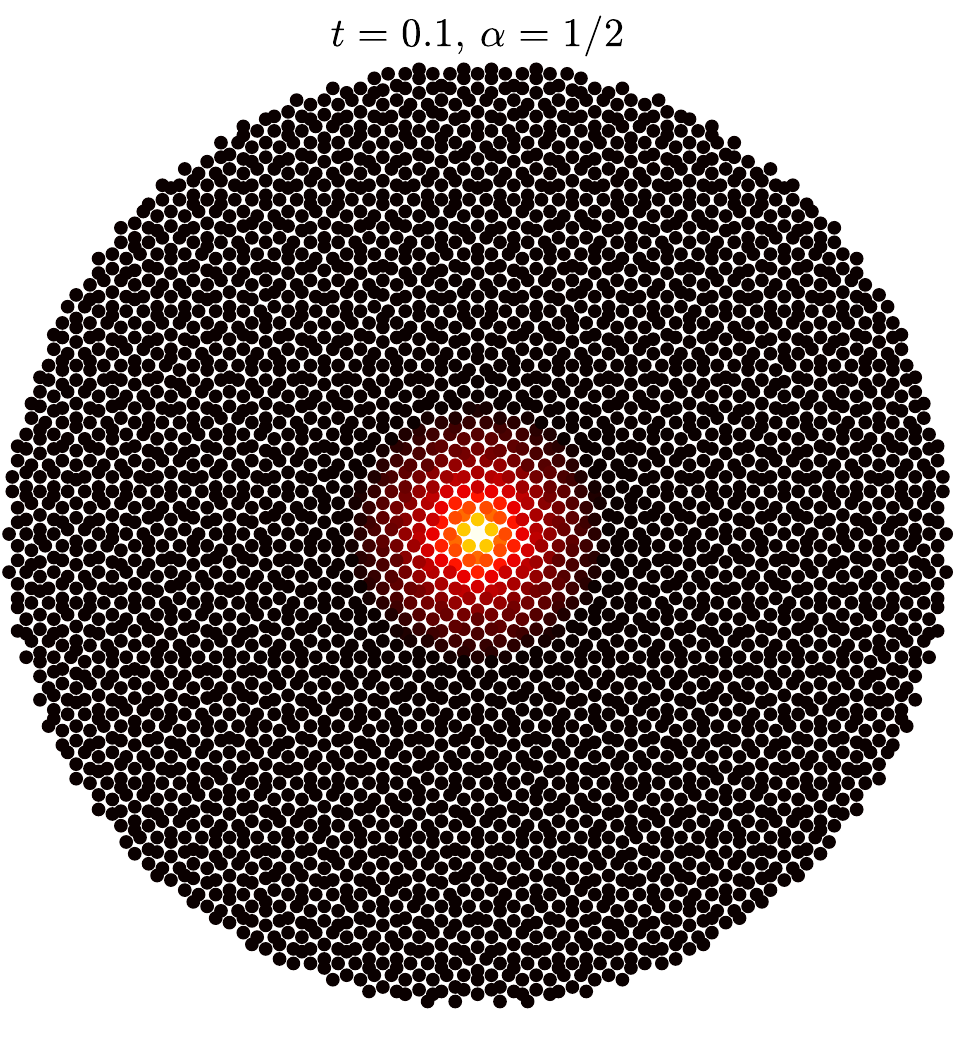}
\includegraphics[width=0.28\textwidth,trim={50mm 97mm 62mm 92mm},clip]{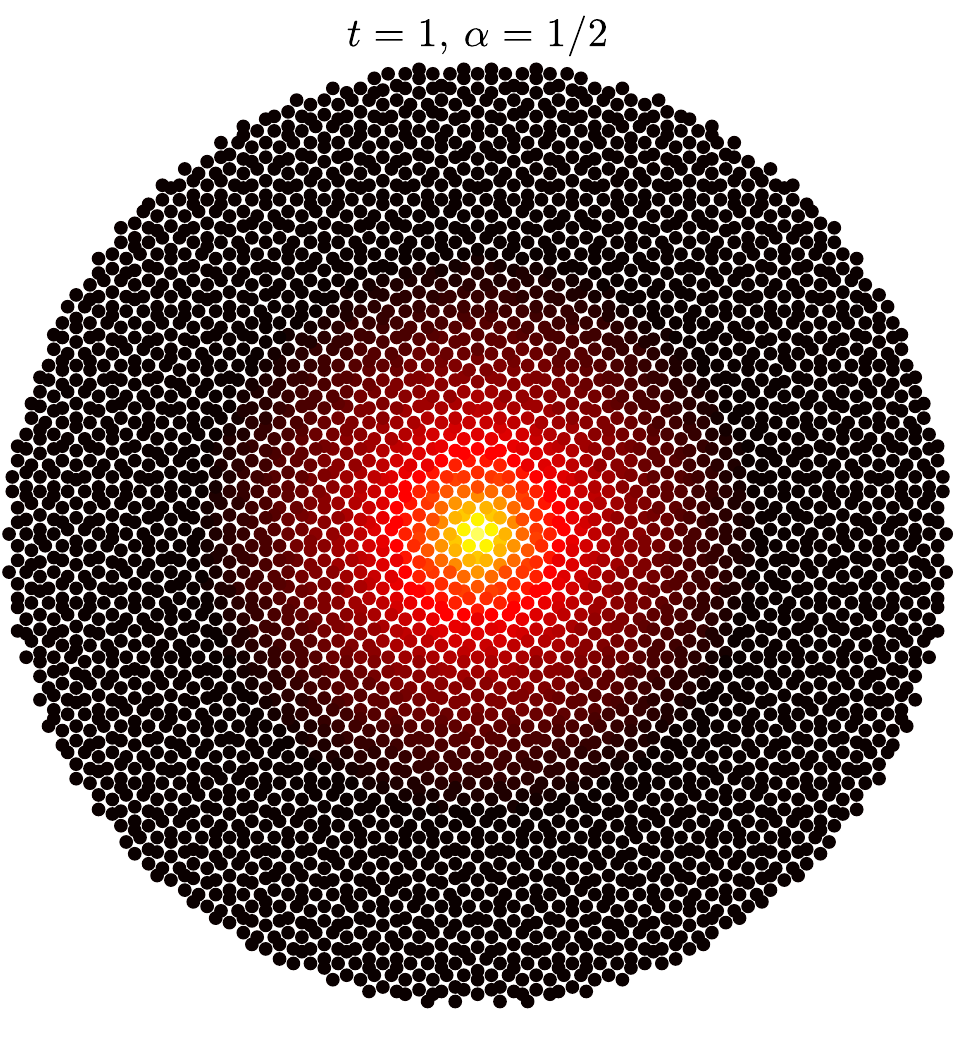}
\includegraphics[width=0.28\textwidth,trim={50mm 97mm 62mm 92mm},clip]{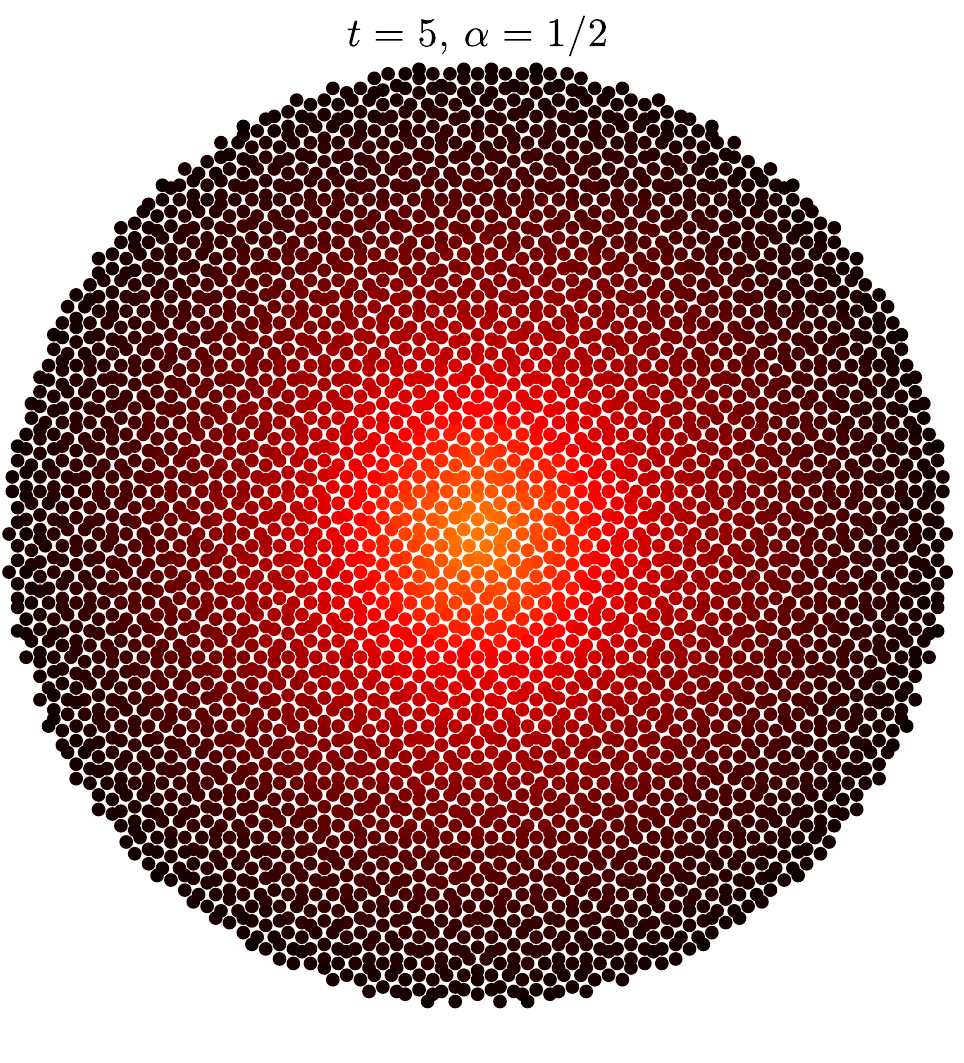}
\includegraphics[width=0.48\textwidth,trim={32mm 102mm 32mm 173mm},clip]{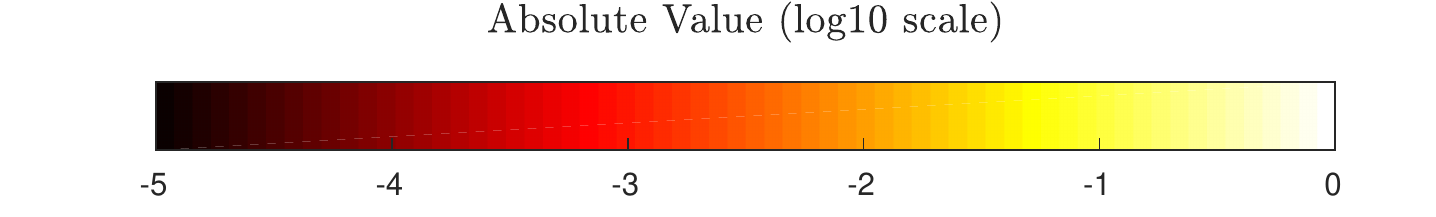} 
\caption{Evolution of initial wavepacket under fractional diffusion.}
\label{pen2}
\end{figure}

We consider fractional diffusion governed by
$$
\frac{du}{dt}=-(-H_0)^{\alpha}u,\quad u_{t=0}=u_0,
$$
for $\alpha>0$. If $\alpha$ is an integer, then the solution can be represented via contour deformation as
\begin{equation}
\label{int_sol}
u(t)=\frac{1}{2\pi i}\int_{\gamma}\exp(-z^{\alpha}t)R(z,-H_0)u_0dz,
\end{equation}
where $\gamma$ is a closed contour looping once around the spectrum. Typically we took the rectangular contour shown in Figure \ref{pen0} and approximated the integral via Gaussian quadrature. This allows us to compute the solution with error control (we known the minimal distance between $\gamma$ and $\sigma(-H_0)$ so can bound the Lipschitz constant of the resolvent) and clearly, this holds for other functions $F$, holomorphic on a neighbourhood of $\sigma(-H_0)$. Note that other methods, such as direct diagonalisation of finite square truncations or discrete time stepping (which is difficult if $\alpha\notin\mathbb{N}$), do not give error control and are slower. In fact, for this example, direct diagonalisation was impractical. When $\alpha\notin\mathbb{N}$, we can still deform the contour, but not at $0$ since $0\in\sigma(-H_0)$. Hence we deform the contour to that shown in Figure \ref{pen0}. For a discussion of contour methods applied to finite matrices (in the case that the spectrum is strictly positive), we refer the reader to \cite{hale2008computing}. Unfortunately, such methods cannot be applied here since $0\in\sigma(-H_0)$. Moreover, $0$ appears to not be an isolated point of the spectrum, meaning that dealing with this point separately is also impossible.

Figure \ref{pen1} shows the convergence of the algorithm for $\alpha=1/2$ and $\alpha=1$. For $\alpha=1/2$, error control is not given by our algorithm, so we computed an error by comparing to a ``converged'' solution using larger $n$. The $l^2$ error refers to the error in $l^2(\mathbb{N})$. The method converges algebraically for $\alpha=1/2$ (owing to the contour touching the spectrum at $0$) but converges exponentially for $\alpha=1$ with similar convergence observed over a large range of times $t$. Figure \ref{pen2} shows the magnitude (log scale) of the computed solution for various times. Note that the techniques presented here can be applied to any evolution equation of the form (\ref{evolution_we_can_handle}) on \textit{infinite-dimensional} Hilbert spaces. They may also be useful for splitting methods/exponential integrators, which require fast computation of matrix/operator exponentials (see \cite{hochbruck2010exponential,mclachlan2002splitting} and the references therein) and more generally in the field of infinite-dimensional ODE/PDE systems. For methods that compute general semigroups on infinite-dimensional separable Hilbert spaces with error control, see \cite{colb_semigp}.

\section*{Acknowledgments} This work was supported by EPSRC grant EP/L016516/1. I wish to thank Arieh Iserles, Andrew Horning, Sheehan Olver, Christian Remling and Marcus Webb for engaging and helpful discussions during the completion of this work. I also wish to thank the anonymous referees whose suggestions led to an improvement in the structure of the paper and in the clarity of some of the proofs.

\newpage
\appendix

\section{The SCI Hierarchy - a Framework for Computation}
\label{append1}

Cornerstones in the SCI hierarchy are the definitions of a computational problem, a general algorithm and towers of algorithms. The basic objects in a computational problem are as follows:
\begin{itemize}
\itemsep0em
\item[(i)] $\Omega$ is some set, called the \emph{domain}.
\item[(ii)] $\Lambda$ is a set of complex-valued functions on $\Omega$ called the \emph{evaluation} set.
\item[(iii)] $\mathcal{M}$ is a metric space with metric $d_{\mathcal{M}}$.
\item[(iv)] $\Xi:\Omega\to \mathcal{M}$ is called the \emph{problem function}.
\end{itemize}
The set $\Omega$ is the set of objects that give rise to our computational problems. The problem function $\Xi : \Omega\to \mathcal{M}$ is what we are interested in computing. Moreover, the set $\Lambda$ is the collection of functions that provide us with the information we are allowed to read.

\begin{remark}
\label{important_about_lam}
Throughout the paper we have relaxed the condition that $\lambda\in\Lambda$ maps into $\mathbb{C}$ by considering evaluation functions consisting of intervals exhausting a set, piecewise constant functions of compact support etc. These seemingly more complicated objects can be effectively encoded by functions that map into $\mathbb{C}$. For example, when considering the decomposition (\ref{open_union}):
\begin{equation*}
U=\bigcup_{m}(a_m(U),b_m(U))
\end{equation*}
of an open set $U$, we consider $\lambda_{1,m},\lambda_{2,m}$ with $\lambda_{1,m}(U)=a_m$ and $\lambda_{2,m}(U)=b_m$. For the sake of clarity of presentation of the proofs, such encodings are used implicitly throughout the paper. 
\end{remark}

This leads to the following definition.

\begin{definition}[Computational Problem]
Given a primary set $\Omega$, an evaluation set $\Lambda$, a metric space $\mathcal{M}$ and a problem function $\Xi:\Omega\to \mathcal{M}$ we call the collection $\{\Xi,\Omega,\mathcal{M},\Lambda\}$ a {computational problem}.
\end{definition}

The goal is to find algorithms that approximate the function $\Xi$. More generally, we need the concept of a tower of algorithms, which is needed to describe problems that need several limits in the computation. However, first one needs the definition of a general algorithm.
\begin{definition}[General Algorithm]\label{Gen_alg}
Given a computational problem $\{\Xi,\Omega,\mathcal{M},\Lambda\}$, a {general algorithm} is a mapping $\Gamma:\Omega\to \mathcal{M}$ such that for each $A\in\Omega$
\begin{itemize}
\item[(i)] There exists a finite non-empty subset of evaluations $\Lambda_\Gamma(A) \subset\Lambda$, 
\item[(ii)] The action of $\,\Gamma$ on $A$ only depends on $\{A_f\}_{f \in \Lambda_\Gamma(A)}$ where $A_f := f(A),$
\item[(iii)] For every $B\in\Omega$ such that $B_f=A_f$ for every $f\in\Lambda_\Gamma(A)$, it holds that $\Lambda_\Gamma(B)=\Lambda_\Gamma(A)$.
\end{itemize}
\end{definition}

Note that the definition of a general algorithm is more general than the definition of a Turing machine or a Blum--Shub--Smale (BSS) machine. A general algorithm has no restrictions on the operations allowed. The only restriction is that it can only take a finite amount of information, though it is allowed to \emph{adaptively} choose the finite amount of information it reads depending on the input. Condition (iii) assures that the algorithm consistently reads the information. Note that the purpose of such a general definition is to get strong lower bounds. In particular, the more general the definition is, the stronger a lower bound will be.

With a definition of a general algorithm, we can define the concept of towers of algorithms.

\begin{definition}[Tower of Algorithms]\label{tower_funct}
Given a computational problem $\{\Xi,\Omega,\mathcal{M},\Lambda\}$, a {tower of algorithms of height $k$
 for $\{\Xi,\Omega,\mathcal{M},\Lambda\}$} is a family of sequences of functions
 $$\Gamma_{n_k}:\Omega
\rightarrow \mathcal{M},\ \Gamma_{n_k, n_{k-1}}:\Omega
\rightarrow \mathcal{M},\dots,\ \Gamma_{n_k, \hdots, n_1}:\Omega \rightarrow \mathcal{M},
$$
where $n_k,\hdots,n_1 \in \mathbb{N}$ and the functions $\Gamma_{n_k, \hdots, n_1}$ at the ``lowest level'' of the tower are general algorithms in the sense of Definition \ref{Gen_alg}. Moreover, for every $A \in \Omega$,
$$
\Xi(A)= \lim_{n_k \rightarrow \infty} \Gamma_{n_k}(A), \quad \Gamma_{n_k, \hdots, n_{j+1}}(A)= \lim_{n_j \rightarrow \infty} \Gamma_{n_k, \hdots, n_j}(A) \quad j=k-1,\dots,1.
$$
\end{definition}

In addition to a general tower of algorithms (defined above), we will focus on arithmetic towers. The definition of a general algorithm allows for strong lower bounds; however, to produce upper bounds we must add structure to the algorithm and towers of algorithms. An arithmetic tower allows for arithmetic operations and comparisons.

\begin{definition}[Arithmetic towers]
Given a computational problem $\{\Xi,\Omega,\mathcal{M},\Lambda\}$, where $\Lambda$ is countable, we define the following: An {arithmetic tower of algorithms} of height $k$
 for $\{\Xi,\Omega,\mathcal{M},\Lambda\}$ is a tower of algorithms where the lowest functions $\Gamma = \Gamma_{n_k, \hdots, n_1} :\Omega \rightarrow \mathcal{M}$ satisfy the following:
 For each $A\in\Omega$ the mapping $(n_k, \hdots, n_1) \mapsto \Gamma_{n_k, \hdots, n_1}(A) = \Gamma_{n_k, \hdots, n_1}(\{A_f\}_{f \in \Lambda})$ is recursive, the action of $\,\Gamma$ on $A$ consists of only finitely many arithmetic operations and comparisons on $\{A_f\}_{f \in \Lambda_\Gamma(A)}$, and $\Gamma_{n_k, \hdots, n_1}(A)$ is a finite string of complex numbers that can be identified with an element in $\mathcal{M}$.
\end{definition} 

\begin{remark}[Recursiveness] By recursive we mean the following. If $f(A) \in \mathbb{Q}$ for all $f \in \Lambda$, $A \in \Omega$, and $\Lambda$ is countable, then $\Gamma_{n_k, \hdots, n_1}(\{A_f\}_{f \in \Lambda})$ can be executed by a Turing machine \cite{turing1937computable}, that takes $(n_k, \hdots, n_1)$ as input, and that has an oracle tape consisting of $\{A_f\}_{f \in \Lambda}$. If $f(A) \in \mathbb{R}$ (or $\mathbb{C}$) for all $f \in \Lambda$, then $\Gamma_{n_k, \hdots, n_1}(\{A_f\}_{f \in \Lambda})$ can be executed by a Blum--Shub--Smale machine \cite{BCSS} that takes $(n_k, \hdots, n_1)$, as input, and that has an oracle that can access any $A_f$ for $f \in \Lambda$. 
\end{remark}

Given the above definitions, we can now define the Solvability Complexity Index: 

\begin{definition}[Solvability Complexity Index]\label{complex_ind}
A computational problem $\{\Xi,\Omega,\mathcal{M},\Lambda\}$ is said to have {Solvability Complexity Index $\mathrm{SCI}(\Xi,\Omega,\mathcal{M},\Lambda)_{\alpha} = k$}, with respect to a tower of algorithms of type $\alpha$, if $k$ is the smallest integer for which there exists a tower of algorithms of type $\alpha$ of height $k$. If no such tower exists then $\mathrm{SCI}(\Xi,\Omega,\mathcal{M},\Lambda)_{\alpha} = \infty.$ If there exists a tower $\{\Gamma_n\}_{n\in\mathbb{N}}$ of type $\alpha$ and height one such that $\Xi = \Gamma_{n_1}$ for some $n_1 < \infty$, then we define $\mathrm{SCI}(\Xi,\Omega,\mathcal{M},\Lambda)_{\alpha} = 0$. We may sometimes write $\mathrm{SCI}(\Xi,\Omega)_{\alpha}$ to simplify notation when $\mathcal{M}$ and $\Lambda$ are obvious. 
\end{definition}

The definition of the SCI immediately induces the SCI hierarchy:

\begin{definition}[The Solvability Complexity Index Hierarchy]
\label{1st_SCI}
Consider a collection $\mathcal{C}$ of computational problems and let $\mathcal{T}$ be the collection of all towers of algorithms of type $\alpha$ for the computational problems in $\mathcal{C}$.
Define 
\begin{equation*}
\begin{split}
\Delta^{\alpha}_0 &:= \{\{\Xi,\Omega\} \in \mathcal{C} \ \vert \   \mathrm{SCI}(\Xi,\Omega)_{\alpha} = 0\}\\
\Delta^{\alpha}_{m+1} &:= \{\{\Xi,\Omega\}  \in \mathcal{C} \ \vert \   \mathrm{SCI}(\Xi,\Omega)_{\alpha} \leq m\}, \qquad \quad m \in \mathbb{N},
\end{split}
\end{equation*}
as well as
\[
\Delta^{\alpha}_{1} := \{\{\Xi,\Omega\}  \in \mathcal{C}   \  \vert \ \exists \ \{\Gamma_n\}_{n\in \mathbb{N}} \in \mathcal{T}\text{ s.t. } \forall A \ d(\Gamma_n(A),\Xi(A)) \leq 2^{-n}\}. 
\]
\end{definition}

In certain cases, the SCI hierarchy can also be refined with notions of error control - see \cite{ben2015can,colbrook2020foundations,colbrook3,colbrook4}. We also need the following result.

\begin{proposition}\label{PCholesky}
Given a matrix $B\in\mathbb{C}^{m\times n}$ and a number $\epsilon>0$ one can test with finitely many arithmetic operations on the entries of $B$ whether the smallest singular value $\sigma_1(B)$ of $B$ is greater than $\epsilon$.
\end{proposition}
\begin{proof}
The matrix $B^*B$ is self-adjoint and positive semidefinite, hence has its eigenvalues in $[0,\infty)$. The singular values of $B$ are the square roots of these eigenvalues of $B^*B$.
The smallest singular value is greater than $\epsilon$ if and only if the smallest eigenvalue of $B^*B$ is greater than $\epsilon^2$, which is the case if and only if $C:=B^*B-\epsilon^2I$ is positive definite. The matrix $C$ is positive definite if and only if the pivots left after Gaussian elimination (without row exchange) are all positive. 
Thus, if $C$ is positive definite, Gaussian elimination leads to pivots that are all positive, and this requires finitely many arithmetic operations. If $C$ is not positive definite, then at some point a pivot is zero or negative, at this point the algorithm aborts. An alternative is the Cholesky decomposition. Although forming the lower triangular $L \in \mathbb{C}^{n \times n}$ (if it exists) such that $C = LL^*$ requires the use of radicals, the existence of $L$ can be determined using finitely many arithmetic operations. This follows from the standard Cholesky algorithm, and we omit the details.  
\end{proof}

Finally, we also need the following result.

\begin{proposition}
Let $(\mathcal{M}',d')$ be the discrete space $\{0,1\}$, let $\Omega'$ denote the collection of all infinite sequence $\{a_{j}\}_{j\in\mathbb{N}}$ with entries $a_{j}\in\{0,1\}$, let $\Lambda'$ consist of pointwise evaluations of the $\{a_j\}$ and consider the problem function
\begin{equation*}
\Xi'(\{a_{j}\}):\text{ `Does }\{a_{j}\}\text{ have infinitely many non-zero entries?'}
\end{equation*}
Then $\Delta^G_2 \not\owns  \{\Xi',\Omega',\Lambda'\} \in \Delta^A_3$
\end{proposition}
\begin{proof}
Consider the arithmetic tower of algorithms defined by
$$
\Gamma_{m,n}(\{a_j\})=\begin{dcases}
1\quad\text{if }\sum_{j=1}^na_j>m,\\
0\quad\text{otherwise}.\end{dcases}
$$
This provides a height two arithmetic tower for $\Xi'$ and hence $\{\Xi',\Omega',\Lambda'\} \in \Delta^A_3$.

To prove the lower bound, suppose for a contradiction that $\{\Gamma_n\}$ is a sequence of general algorithms, using $\Lambda'$, such that
$$
\lim_{n\rightarrow\infty}\Gamma_{n}(\{a_j\})=\Xi'(\{a_{j}\}).
$$
We will construct a sequence $\{a_j\}$ such that $\Gamma_{n}(\{a_j\})$ does not converge, providing the required contradiction.

Set $\{a^1_j\}_{j\in\mathbb{N}}=\{0,0,...\}$, then there exists $n_1\in\mathbb{N}$ such that $\Gamma_{n_1}(\{a^1_j\})=0$. Moreover, $\Lambda_{\Gamma_{n_1}}(\{a^1_j\})\subset \{\{a_j\}_{j\in\mathbb{N}}\rightarrow a_m:m\leq N_1\}$ for some integer $N_1$ by (i) of Definition \ref{Gen_alg}. We choose $\{a^2_j\}_{j\in\mathbb{N}}$ such that $a^2_j=a^1_j$ for $j\leq N_1$ and $a_j^2=1$ otherwise. Then there exists $n_2\in\mathbb{N}$ such that $n_2>n_1$ and $\Gamma_{n_2}(\{a^2_j\})=1$. Moreover, $\Lambda_{\Gamma_{n_2}}(\{a^2_j\})\subset \{\{a_j\}_{j\in\mathbb{N}}\rightarrow a_m:m\leq N_2\}$ for some integer $N_2>N_1$ by (i) of Definition \ref{Gen_alg}. We now repeat this inductively. Explicitly, given the construction up to the $k$th step, we define $\{a^{k+1}_j\}_{j\in\mathbb{N}}$ by $a^{k+1}_j=a^k_j$ for all $j\leq N_k$ and $a^{k+1}_j=(1+(-1)^{k+1})/2$ otherwise. Then there exists $n_{k+1}\in\mathbb{N}$ such that $n_{k+1}>n_k$ and $\Gamma_{n_{k+1}}(\{a^{k+1}_j\})=(1+(-1)^{k+1})/2$. Moreover, $\Lambda_{\Gamma_{n_{k+1}}}(\{a^{k+1}_j\})\subset \{\{a_j\}_{j\in\mathbb{N}}\rightarrow a_m:m\leq N_{k+1}\}$ for some integer $N_{k+1}>N_k$ by (i) of Definition \ref{Gen_alg}.

Finally set $a_j=a^k_j$ for $j\leq N_{k}$. It is clear from (iii) of Definition \ref{Gen_alg}, that $\Gamma_{n_k}(\{a_j\})=\Gamma_{n_{k}}(\{a^k_j\})=(1+(-1)^{k})/2$ and this implies that $\Gamma_{n}(\{a_j\})$ cannot converge, the required contradiction.
\end{proof}

\section{Partial Differential Operators - Proof of Theorem \ref{WowPDE}}
\label{append2}

Before stating and proving a more mathematically precise version of Theorem \ref{WowPDE}, we need to be precise about the computational problems involved. Recall that we consider $L$ formally defined on $L^2(\mathbb{R}^d)$ by
$$
Lu(x)=\sum_{k\in\mathbb{Z}_{\geq0}^d,\left|k\right|\leq N}a_k(x)\partial^ku(x)
$$
with the class $\Omega_{\mathrm{PDE}}$ consisting of self-adjoint $L$ satisfying certain conditions given in \S \ref{PDE_SEC_WOW}. We take $\Lambda_{\mathrm{PDE}}$ to consist of
$$
S_{k,q,m}:\Omega_{\mathrm{PDE}}\rightarrow \mathbb{Q}+i\mathbb{Q},
$$
where $|S_{k,q,m}(L)-a_k(q)|\leq 2^{-m}$ for all $q\in\mathbb{Q}$, together with
$$
b_n:\Omega_{\mathrm{PDE}}\rightarrow \mathbb{Q}_{>0}
$$
such that
$$
\sup_{n\in\mathbb{N}}\max_{\left|k\right|\leq N}\frac{\left\|a_k\right\|_{\mathcal{A}_n}}{b_n(L)}<\infty.
$$
We also implicitly assume that for any given $L$, the dimension $d$ and integer $N$ (order) are known. As well as this, we need to consider the pairs of vectors (members of $L^2(\mathbb{R}^d)$) with which we compute the spectral properties. By the polarisation identity, we can consider equal functions of norm $1$. We let $V_{\mathrm{PDE}}$ consist of all $f\in L^2(\mathbb{R}^d)$ of norm $1$ such that
\begin{enumerate}
	\item[(1)] There exists a positive constant $C$ and an integer $D$ (both possibly unknown) such that 
	$$
	\left|f(x)\right|\leq C\left(1+\left|x\right|^{2D}\right),
	$$
	almost everywhere on $\mathbb{R}^d$, that is, $f$ is polynomially bounded.
	\item[(2)] The restrictions $f|_{[-r,r]^d}\in\mathcal{A}_r$ for all $r>0$.
\end{enumerate}
We then add to $\Lambda_{\mathrm{PDE}}$ the evaluation functions
$$
S_{q,m}:V_{\mathrm{PDE}}\rightarrow \mathbb{Q}+i\mathbb{Q},
$$
where $|S_{q,m}(f)-f(q)|\leq 2^{-m}$ for all $q\in\mathbb{Q}$, together with
$$
c_n:V_{\mathrm{PDE}}\rightarrow \mathbb{Q}_{>0}
$$
such that
$$
\sup_{n\in\mathbb{N}}\frac{\left\|f\right\|_{\mathcal{A}_n}}{c_n(f)}<\infty.
$$

With an abuse of notation, we can consider the computational problem from Theorem \ref{meas_comp1}
\begin{align*}
\Xi_{\mathrm{meas}}:\Omega_{\mathrm{PDE}}\times V_{\mathrm{PDE}}\times \mathcal{U}&\rightarrow L^2(\mathbb{R}^d)\\
(L,f,U)&\rightarrow E_U^Lf.
\end{align*}
Recall that $\mathcal{U}$ is the collection of non-trivial open sets and we have access to a finite or countable collection $a_m(U),b_m(U)\in\mathbb{R}\cup\{\pm\infty\}$ such that $U\in\mathcal{U}$ can be written as a disjoint union
$$
U=\bigcup_{m}\left(a_m(U),b_m(U)\right).
$$
Similarly, from Theorem \ref{F_calc} we consider
\begin{align*}
\Xi_{\mathrm{fun}}:\Omega_{\mathrm{PDE}}\times V_{\mathrm{PDE}}\times C_b(\mathbb{R})&\rightarrow L^2(\mathbb{R}^d)\\
(L,f,F)&\rightarrow F(L)f.
\end{align*}
We assume that given $F\in C_b(\mathbb{R})$, we have access to piecewise constant functions $F_n$ supported in $[-n,n]$ such that $\|F-F_n\|_{L^\infty([-n,n])}\leq n^{-1}$. Finally, from Theorem \ref{meas_comp3} we consider
\begin{align*}
\Xi_{\mathrm{RN}}:\Omega_{\mathrm{PDE}}\times V_{\mathrm{PDE}}\times \mathcal{U}&\rightarrow L^1(\mathbb{R})\\
(L,f,U)&\rightarrow \rho_{f,f}^L|_{U}.
\end{align*}
We restrict this map to the quadruples $(L,f,U)$ such that $U$ is strictly separated from $\mathrm{supp}(\mu_{f,f,\mathrm{sc}}^L)\cup\mathrm{supp}(\mu_{f,f,\mathrm{pp}}^L)$ and denote this subclass by $\widetilde{\Omega}_{\mathrm{PDE}}$.

We can now state the precise form of Theorem \ref{WowPDE}.

\begin{theorem}[Precise form of Theorem \ref{WowPDE}]
\label{WowPDE2}
Given the above set-up,
$$
\{\Xi_{\mathrm{meas}},\Omega_{\mathrm{PDE}}\times V_{\mathrm{PDE}}\times \mathcal{U},\Lambda_{\mathrm{PDE}}\},\{\Xi_{\mathrm{fun}},\Omega_{\mathrm{PDE}}\times V_{\mathrm{PDE}}\times C_b(\mathbb{R}),\Lambda_{\mathrm{PDE}}\},\{\Xi_{\mathrm{RN}},\widetilde{\Omega}_{\mathrm{PDE}},\Lambda_{\mathrm{PDE}}\}\in\Delta_2^A.
$$
In other words, we can construct a convergent sequence of arithmetic algorithms for each problem.
\end{theorem}

\begin{remark}
The proof will make clear that we can assume different conditions on the operator $L$ and function $f$. We simply choose an appropriate basis so that we can apply Theorem \ref{res_est3}. We can also extend this to scalar measures (through inner products) and also to the towers of algorithms (of height $\geq 2$) used to compute the decompositions into pure point, absolutely continuous and singular continuous parts of measures and spectra.
\end{remark}

To prove Theorem \ref{WowPDE2}, we need the following theorem, which is similar to the results of \S \ref{approx_res_sect}.

\begin{theorem}
\label{res_est3}
Consider the class $\Omega_{\mathrm{SA}}\times S_{l^2(\mathbb{N})}$, where $S_{l^2(\mathbb{N})}$ denotes the unit sphere (vectors of norm $1$). Assume that for each $(T,x)\in\Omega_{\mathrm{SA}}\times S_{l^2(\mathbb{N})}$ we have access to evaluation functions $\hat\Lambda=\{f_{i,j,m}^{(1)},f_{i,j,m}^{(2)},f^{(3)}_{i,m}:i,j,m\in\mathbb{N}\}$ with
$$
\left|f_{i,j,m}^{(1)}(T)-\langle Te_j,e_i\rangle\right|, \left|f_{i,j,m}^{(2)}(T)-\langle Te_j,Te_i\rangle\right|,\left|f_{i,m}^{(3)}(x)-\langle x,e_i\rangle\right|=O(2^{-m}), \quad \forall i,j\in\mathbb{N},
$$
where the hidden constant depends only on $T$ and $x$. Then there exists a sequence of arithmetic algorithms using $\hat\Lambda$
$$
\Gamma_N: \Omega_{\mathrm{SA}}\times S_{l^2(\mathbb{N})}\times\mathbb{C}\backslash\mathbb{R}\rightarrow l^2(\mathbb{N})
$$
with the following properties:
\begin{enumerate}
	\item For all $(T,x,z)\in\Omega_{\mathrm{SA}}\times S_{l^2(\mathbb{N})}\times\mathbb{C}\backslash\mathbb{R}$, $\Gamma_N(T,x,z)$ has finite support with respect to the canonical basis and converges to $R(z,T)x$ in $l^2(\mathbb{N})$ as $N\rightarrow\infty$.
	\item For any $(T,x)\in\Omega_{\mathrm{SA}}\times S_{l^2(\mathbb{N})}$, there exists a constant $C(T,x)$ such that for all $z\in\mathbb{C}\backslash\mathbb{R}$,
	\begin{equation}
	\label{key_bd_PDE}
	\|\Gamma_N(T,x,z)-R(z,T)x\|\leq \frac{C(T,x)}{N\left|\mathrm{Im}(z)\right|}.
	\end{equation}
\end{enumerate}
\end{theorem}

\begin{proof}
The proof is similar to that of Theorem \ref{res_est1}. Let $(T,x,z)\in \Omega_{\mathrm{SA}}\times S_{l^2(\mathbb{N})}\times \mathbb{C}\backslash{\mathbb{R}}$. We have that $n=\mathrm{rank} (P_n)=\mathrm{rank} ((T-zI)P_n)$ for large $n$ since $\sigma_1(T-zI)>0$ (recall that $z\notin\sigma(T)$). Hence we can define
$$
\widetilde\Gamma_{n}(T,x,z):=\begin{cases}
0 \quad\quad\quad  \text{ if } \sigma_1(P_n(T^*-\overline{z}I)(T-zI)P_n)\leq\frac{1}{n}\\
[P_n(T^*-\overline{z}I)(T-zI)P_n]^{-1}P_n(T^*-\overline{z}I)x \quad \text{ otherwise.}
\end{cases}
$$
Suppose that $n$ is large enough so that $\sigma_1(P_n(T^*-\overline{z}I)(T-zI)P_n)>1/n$. Then $\widetilde\Gamma_{n}(T,x,z)$ is a least-squares solution of the optimisation problem $\mathrm{argmin}_{y}\|(T-zI)P_ny-x\|$. The linear space $\mathrm{span}\{e_n:n\in\mathbb{N}\}$ forms a core of $T$ and hence also of $T-zI$. It follows by invertibility of $T-zI$ that given any $\epsilon>0$, there exists an $m=m(\epsilon)$ and a $y=y(\epsilon)$ with $P_my=y$ such that
$$
\|(T-zI)y-x\|\leq\epsilon.
$$
It follows that for all $n\geq m$, $\|(T-zI)\widetilde\Gamma_{n}(T,x,z)-x\|\leq \|(T-zI)y-x\|\leq \epsilon$ and hence that
$$
\|\widetilde\Gamma_{n}(T,x,z)-R(z,T)x\|\leq \frac{\epsilon}{\left|\mathrm{Im}(z)\right|}.
$$
Since $\epsilon>0$ was arbitrary, we see that $\widetilde\Gamma_{n}(T,x,z)$ converges to $R(z,T)x$.

For $n,m\in\mathbb{N}$, define the finite matrices
$$
B_n=P_n(T^*-\overline{z}I)(T-zI)P_n,\quad C_{m,n}=P_n(T^*-\overline{z}I)P_{m}.
$$
Given the evaluation functions in $\hat\Lambda$, we have access to the entries of these matrices to asymptotic accuracy (i.e. for a given diverging subsequence $a_n$, to precision $O(2^{-a_n})$). It follows that we have access to approximations of $B_n$ and $C_{m,n}$ denoted $\widetilde B_n$ and $\widetilde C_{m,n}$ respectively with
$$
\|B_n-\widetilde B_n\|=O(n^{-1}),\quad \|C_{m,n}-\widetilde C_{m,n}\| =O(n^{-1}).
$$
Recall that the $O(\cdot)$ notation also means independently of $z$ and other parameters (though it may depend on $T$ and $x$). Note that $\widetilde B_n^{-1}$ can be computed using finitely many arithmetic operations and comparisons and the resolvent identity implies that
$$
\|B_n^{-1}-\widetilde B_n^{-1}\|=O(n^{-1}).
$$
From the proof of Proposition \ref{PCholesky} and a simple search routine, we can also compute $\sigma_1(\widetilde B_n)$ to accuracy $n^{-2}$ using finitely many arithmetic operations and comparisons. Denote the approximation via $\tau_n$. We then define
$$
\Gamma_{m,n}(T,x,z):=\begin{cases}
\quad\quad\quad 0 & \text{ if } \tau_n\leq\frac{1}{n}\\
\widetilde B_n^{-1}\widetilde C_{m,n} x_{(m)}& \text{ otherwise,}
\end{cases}
$$
where $x_{(m)}=P_mx_{(m)}$ is an approximation of $P_mx$ to accuracy $O(m^{-1})$. It follows that $\Gamma_{m,n}(T,x,z)$ can be computed using finitely many arithmetic operations. We also have that
$$
\|\Gamma_{m,n}(T,x,z)-\widetilde\Gamma_{n}(T,x_{(m)},z)\|\leq \|B_n^{-1}-\widetilde B_n^{-1}\|\|C_{m,n}\|\|x_{(m)}\|+\|\widetilde B_n^{-1}\|\|C_{m,n}-\widetilde C_{m,n}\|\|x_{(m)}\|,
$$
so that $\Gamma_{m,n}(T,x,z)$ converges to $R(z,T)x_{(m)}$ as $n\rightarrow\infty$. By construction, $\Gamma_{m,n}(T,x,z)$ has finite support with respect to the canonical basis. Furthermore, since $P_mx_{(m)}=x_{(m)}$, the following error bound holds for any $l\geq m$
\begin{align*}
\|\Gamma_{m,n}(T,x,z)-R(z,T)x_{(m)}\|&\leq \|R(z,T)\|\|(T-zI)\Gamma_{m,n}(T,x,z)-x_{(m)}\|\\
&\leq \frac{\|(I-P_{l})(T-zI)P_n\|\|\Gamma_{m,n}(T,x,z)\|+\|P_{l}(T-zI)\Gamma_{m,n}(T,x,z)-x_{(m)}\|}{\left|\mathrm{Im}(z)\right|}.
\end{align*}
Since we have access to both $\langle Te_j, Te_j \rangle$ (norms of the columns of $T$) and $\langle Te_i, e_j \rangle$, we can estimate $\|(I-P_{l})(T-zI)P_n\|$ to asymptotic accuracy $O(n^{-1})$. It follows that we can can compute, in finitely many arithmetic operations and comparisons, $l(n)\geq n$ such that
$$
\|(I-P_{l(n)})(T-zI)P_n\|=O(n^{-1})
$$
Similarly, we can estimate $\|P_{l(n)}(T-zI)\Gamma_{m,l(n)}(T,x,z)-x_{(m)}\|$ to accuracy $O(n^{-1})$ and call this approximation $v_{m,n}$, and estimate $\|\Gamma_{m,l(n)}(T,x,z)\|$ to accuracy $O(n^{-1})$ and call this approximation $w_{m,n}$. It follows that we have
$$
\|\Gamma_{m,n}(T,x,z)-R(z,T)x_{(m)}\|\leq  \frac{(w_{m,n}+O(n^{-1}))O(n^{-1})+v_{m,n}+O(n^{-1})}{\left|\mathrm{Im}(z)\right|}.
$$
It follows that
$$
\|\Gamma_{m,n}(T,x,z)-R(z,T)x\|\leq  \frac{\|x_{(m)}-x\|+(w_{m,n}+O(n^{-1}))O(n^{-1})+v_{m,n}+O(n^{-1})}{\left|\mathrm{Im}(z)\right|}.
$$
For a fixed $m$, $v_{m,n}\rightarrow 0$ and $w_{m,n}\rightarrow \|R(z,T)x_{(m)}\|$ as $n\rightarrow\infty$. It follows that we can compute $n(m)$ (again in finitely many arithmetic operations and comparisons) such that
\begin{equation}
\label{key_bd_PDE2}
\|\Gamma_{m,n(m)}(T,x,z)-R(z,T)x\|\leq  \frac{\|x_{(m)}-x\|+O(m^{-1})}{\left|\mathrm{Im}(z)\right|}.
\end{equation}
We know that $\|x\|$ has norm $1$ and hence we must have
$$
\label{cheat_equ}
\|x_{(m)}-x\|^2\leq \sum_{j=m+1}^\infty|x_j|^2+O(m^{-2})=1-\sum_{j=1}^m|x_{(m)}|^2+O(m^{-2}).
$$
However, we can compute $1-\sum_{j=1}^m|x_{(m)}|^2$, which converges to zero as $m\rightarrow\infty$. For a given $N$, it follows that we can compute $m(N)\geq N$ such that
$$
1-\sum_{j=1}^m|x_{(m)}|^2\leq N^{-2}.
$$
This implies that $\|x_{(m)}-x\|=O(N^{-1})$ and hence by setting
$$
\Gamma_N(T,x,z)=\Gamma_{m(N),n(m(N))}(T,x,z)
$$
we see that (\ref{key_bd_PDE2}) implies (\ref{key_bd_PDE}).
\end{proof}

\begin{proof}[Proof of Theorem \ref{WowPDE2}]
We choose an orthonormal basis of $L^2(\mathbb{R}^d)$ so that we can carry over the results for $l^2(\mathbb{N})$ proven in this paper. In \cite{colbrook3} it was shown that (the orthonormal basis of) tensor products of Hermite functions form a core for any $L\in\Omega_{\mathrm{PDE}}$.  Namely for $d=1$ we choose the Hermite functions
$$
\psi_m(x)=(2^{m}m!\sqrt{\pi})^{-1/2}e^{-x^2/2}H_{m}(x),m\in\mathbb{Z}_{\geq 0},
$$
where
$$
H_n(x)=(-1)^n\exp(x^2)\frac{d^n}{dx^n}\exp(-x^2).
$$
For $d>1$ we abuse notation and write $\psi_m = \psi_{m_1}\otimes ...\otimes \psi_{m_d}$. The point of this is that by a suitable ordering of $\{\psi_m\}_{m\in\mathbb{N}^d}=\{\psi_{m(1)},\psi_{m(2)},...\}$, any $L\in\Omega_{\mathrm{PDE}}$ can be represented by $T\in\Omega_{\mathrm{SA}}$ and $f\in V_{\mathrm{PDE}}$ by $\hat f\in l^2(\mathbb{N})$ with the inner products
\begin{align}
\langle Te_j,Te_i \rangle &= \int_{\mathbb{R}^d} (L\psi_{m(j)}(x))\overline{(L\psi_{m(i)}(x))}dx \label{PDENEED1}\\
\langle Te_j,e_i \rangle &= \int_{\mathbb{R}^d} (L\psi_{m(j)}(x))\psi_{m(i)}dx \label{PDENEED2}\\
\langle \hat f,e_i \rangle &= \int_{\mathbb{R}^d} f(x)\psi_{m(i)}(x)dx. \label{PDENEED3}
\end{align}
In \cite{colbrook3} it was shown, using the theory of quasi-Monte Carlo numerical integration, that the inner products in (\ref{PDENEED1})--(\ref{PDENEED3}) can be computed using $\Lambda_{\mathrm{PDE}}$ to asymptotic error control. In other words, that we can compute any of the evaluation functions in $\hat\Lambda$ in the statement of Theorem \ref{res_est3} in finitely many arithmetic operations and comparisons. The result now follows by using Theorem \ref{res_est3} instead of Corollary \ref{res_est2} in simple adaptations of the proofs of Theorems \ref{meas_comp1}, \ref{F_calc} and \ref{meas_comp3}.
\end{proof}

\small
\linespread{0.95}\selectfont{}
\bibliographystyle{plain}
\bibliography{spec3}

\end{document}